\documentclass[reqno]{amsart}
\usepackage{yhmath,amsmath,amsfonts,amssymb,enumerate,amsthm,appendix,caption,lipsum,}
\usepackage{enumitem}
\usepackage[english]{babel}
\usepackage{cmap,mathtools}
\usepackage{graphics} 
\usepackage{epsfig} 
\usepackage{graphicx}\usepackage{epstopdf}
\usepackage[pagewise]{lineno}

\usepackage{a4wide}
\usepackage[margin=0.95in]{geometry}
\usepackage{cite}
\usepackage[utf8]{inputenc}
\usepackage{xcolor}
\usepackage{hyperref}
\usepackage[hyperpageref]{backref}
\hypersetup{
    colorlinks=true,
    linkcolor=myblue,
    filecolor=magenta,      
    urlcolor=mygreen,
    citecolor=mygreen,
}
\definecolor{mygreen}{rgb}{0.01,0.6,0.2}
\definecolor{myblue}{rgb}{0.01, 0.18, 1.0}
\newtheorem{theorem}{Theorem}
\newtheorem{proposition}[theorem]{Proposition}
\newtheorem{lemma}[theorem]{Lemma}

\theoremstyle{definition}
\newtheorem{definition}[theorem]{Definition}
\newtheorem{remark}[theorem]{Remark}

\numberwithin{equation}{section}
\numberwithin{theorem}{section}
\numberwithin{equation}{section}
\numberwithin{theorem}{section}
\title[Fractional system of Schr\"{o}dinger equations with Hardy potential]{On critically coupled $(s_1,s_2)$-fractional system of Schr\"{o}dinger equations with Hardy potential}
\author[R. Kumar, T. Mukherjee \& A. Sarkar]{Rohit Kumar$^1$, Tuhina Mukherjee$^{1,*}$ and Abhishek Sarkar$^{1}$}

\thanks{$^*$Corresponding author.}

\begin{document}

\maketitle 
\centerline{$^{1}$Department of Mathematics, Indian Institute of Technology Jodhpur,}
\centerline{Rajasthan 342030, India}
 \begin{abstract}
  In this article, our main concern is to study the existence of bound and ground state solutions for the following fractional system of Schr\"{o}dinger equations with Hardy potentials:
  \begin{equation*}
    \left\{
     \begin{aligned}
         (-\Delta)^{s_{1}} u - \lambda_{1} \frac{u~~}{|x|^{2s_{1}}} - u^{2_{s_{1}}^{*}-1} = \nu \alpha h(x) u^{\alpha-1}v^{\beta} & \quad \mbox{in} ~ \mathbb{R}^{N}, \\
            (-\Delta)^{s_{2}} v - \lambda_{2} \frac{v~~}{|x|^{2s_{2}}} - v^{2_{s_{2}}^{*}-1} = \nu \beta h(x) u^{\alpha}v^{\beta-1} & \quad \mbox{in} ~ \mathbb{R}^{N},\\
             u,v >0 \quad \mbox{in} ~ \mathbb{R}^{N} \setminus \{0\},
    \end{aligned}
    \right.
\end{equation*}
 where $s_{1},s_{2} \in (0,1)~\text{and}~\lambda_{i}\in (0, \Lambda_{N,s_{i}})$ with $\Lambda_{N,s_{i}} = 2 \pi^{N/2} \frac{\Gamma^{2}(\frac{N+2s_i}{4}) \Gamma(\frac{N+2s_i}{2})}{\Gamma^{2}(\frac{N-2s_i}{4}) ~|\Gamma(-s_{i})|}, (i=1,2)$.
By imposing certain assumptions on the parameters {$\nu, \alpha,\beta$} and on the function $h$, we obtain ground-state solutions using the concentration-compactness principle and the mountain-pass theorem. 
 \end{abstract} 
\section{Introduction}\label{S1}
The study of elliptic equations and systems involving fractional Laplacian is attracting many researchers over the last decade. The keen aspect of studying such equations is due to physical models in many different applications, e.g., geostrophic flows, crystal dislocation, water waves, etc. we refer to \cite{ Bisci2016variational, Bucur2016nonlocal,Dipierro2017book} and the references therein for more details. In this article, we are concerned with the system of fractional Schr\"{o}dinger equations with singular Hardy potential and coupled with terms up to critical power on the entire $\mathbb{R}^N$ given below
\begin{equation} \label{main problem}
    \left\{
    \begin{aligned}
         (-\Delta)^{s_{1}} u - \lambda_{1} \frac{u~~}{|x|^{2s_{1}}} - u^{2_{s_{1}}^{*}-1} = \nu \alpha h(x) u^{\alpha-1}v^{\beta} & \quad \mbox{in} ~ \mathbb{R}^{N}, \\
            (-\Delta)^{s_{2}} v - \lambda_{2} \frac{v~~}{|x|^{2s_{2}}} - v^{2_{s_{2}}^{*}-1} = \nu \beta h(x) u^{\alpha}v^{\beta-1} & \quad \mbox{in} ~ \mathbb{R}^{N},\\
             u,v >0 \quad \mbox{in} ~ \mathbb{R}^{N} \setminus \{0\},
    \end{aligned}
    \right.
\end{equation}
where $s_{1},s_{2} \in (0,1)~\text{and}~\lambda_{i}\in (0, \Lambda_{N,s_{i}})$ with $\Lambda_{N,s_{i}} = 2 \pi^{N/2} \frac{\Gamma^{2}(\frac{N+2s_i}{4}) \Gamma(\frac{N+2s_i}{2})}{\Gamma^{2}(\frac{N-2s_i}{4}) ~|\Gamma(-s_{i})|}, (i=1,2)$. The constant $\Lambda_{N,s_{i}}, (i=1,2)$ is an optimal constant for the fractional Hardy inequality \cite[Theorem 1.1]{Frank2008}.
Further $2_{s_i}^{*} = \frac{2N}{N-2s_i},$ $(2s_i < N ~\text{and}~i=1,2)$ is the critical Sobolev exponent; the parameters $\nu,\alpha ~\text{and}~\beta$ are positive reals such that 
\begin{align} \label{ alpha beta condition}
    \alpha,\beta>1~~\text{and}~~\alpha+\beta \leq \min \{ 2_{s_{1}}^{*},2_{s_{2}}^{*}\},
\end{align}
 and $h$ is a  function defined on $\mathbb{R}^{N}$ satisfying  
 \begin{align} \label{condition on h}
    0<h \in L^{1}(\mathbb{R}^{N})\cap L^{\infty}(\mathbb{R}^{N}).
 \end{align}
For $s_1 = s_2=1$, the system \eqref{main problem} reduces to the coupled system of local nonlinear Schr\"{o}dinger equations with singular Hardy potential
\begin{equation} \label{classical problem}
    \begin{cases}
        \begin{aligned}
            -\Delta u - \lambda_{1} \frac{u~}{|x|^{2}} - u^{2^{*}-1} = \nu \alpha h(x) |u|^{\alpha-2}|v|^{\beta}u &  \mbox{ in } \mathbb{R}^{N}, \\
            -\Delta v - \lambda_{2} \frac{v~}{|x|^{2}} - v^{2^{*}-1} = \nu \beta h(x) |u|^{\alpha}|v|^{\beta-2}v &  \mbox{ in }  \mathbb{R}^{N}.
        \end{aligned}
    \end{cases}
\end{equation}
{ For} $\nu=0$, the system \eqref{classical problem} becomes a single nonlinear elliptic equation and {the author} in  \cite[Terracini]{Terracini1996} discussed the existence of positive solutions as well as their qualitative properties. In 2009, Abdellaoui et al. (see \cite{Abdellaoui2009}) dealt with the local system \eqref{classical problem}, and they obtained the existence of positive ground state solutions depending on the parameter $\nu>0$ (large or small) and the non-negative function $h(x) \in L^{1}(\mathbb{R}^N) \cap L^{\infty}(\mathbb{R}^N)$ for $2<\alpha + \beta < 2^*$ and $h(x) \in L^{\infty}(\mathbb{R}^N)$ for $\alpha + \beta = 2^*$. Later in  2014, Kang \cite{Kang2014} proved the existence of a positive solution to \eqref{classical problem} considering $ h(x),\lambda_{1}(x),\lambda_{2}(x) \in C(\mathbb{R}^N)$ with additional assumptions. In 2015, Chen and Zou \cite{Chen2015classical} dealt with the critical case i.e., $\alpha + \beta = 2^*$ with $h(x)=1$, and proved the existence of positive solutions which are radially symmetric. Further, Zhong and Zou \cite{Zhong2015critical} proved the existence of ground state solutions by allowing $h(x)$ to change its sign with coupling parameter $\nu =1$. Recently Colorado et al. \cite{Colorado2022} studied the problem \eqref{classical problem} with $\alpha+\beta \leq 2^*$ and $0\leq h(x) \in L^{\infty}(\mathbb{R}^{N})$, and obtained positive ground and bound state solutions depending on the behaviour of parameter $\nu>0$. The similar types of results were also derived by Colorado et al. \cite{Colorado2021bound} when $\alpha=2$ and $\beta=1$,  $0 \leq h(x) \in L^{\infty}(\mathbb{R}^{N})$.

 While dealing with the fractional system \eqref{main problem}, if we suppose $s_1 =s_2=s$ and $\nu=0$, then this system reduces to a fractional doubly critical equation 
 \begin{equation} \label{single problem}
    (-\Delta)^{s_{}} u - \lambda_{} \frac{u~~}{|x|^{2s}} = u^{2_{s_{}}^{*}-1} \quad  \text{in}~ \mathbb{R}^N.
\end{equation}
In 2016 Dipierro et al. in their paper \cite[Theorem 1.5]{Dipierro2016} proved the existence of a positive solution using variational approach for any $0\leq \lambda < \Lambda_{N,s}$. Moreover, they used the moving plane method to obtain the qualitative behavior (such as radial symmetry, asymptotic behaviors etc.) of solutions of \eqref{single problem}.
 In 2020, He and Peng \cite{He2020} considered the following fractional system in $\mathbb{R}^{N}$
\begin{equation} \label{Qihan}
    \left\{
        \begin{array}{ll}
            (-\Delta)^{s} u + P(x)u - \mu_{1}|u|^{2p-2}u = \beta  |v|^{p}|u|^{p-2}u & \quad \text{in} ~ \mathbb{R}^{N}, \\
            (-\Delta)^{s} v + Q(x)u - \mu_{2}|v|^{2p-2}v = \beta  |u|^{p}|v|^{p-2}v & \quad \text{in} ~ \mathbb{R}^{N},\\
             u,v \in H^{s}(\mathbb{R}^{N}), 
        \end{array}
    \right.
\end{equation}
where $N\geq 2,~ 0 < s < 1,~ 1 < p < \frac{N}{N-2s} ,~ \mu_{1}>0,~ \mu_{2}>0$ and $\beta \in \mathbb{R}$ is a coupling constant, and $P(x),Q(x)$ are continuous bounded radial functions. The authors used variational methods to obtain the existence of infinitely many non-radial positive solutions. Observe that the above system contains only the subcritical nonlinear terms and coupled terms up to subcritical power. Recently, Shen \cite{Shen2022brezis} considered the following fractional elliptic systems with Hardy-type singular potentials and coupled by critical homogeneous nonlinearities on the bounded domain $\Omega \subset \mathbb{R}^{N}$
\begin{equation} \label{Bounded domain problem}
    \begin{cases}
            (-\Delta)^{s} u - \lambda_{1} {  \frac{u}{|x|^{2s}} } - |u|^{2_{s}^{*}-2}u = \frac{n\alpha}{2_{s}^{*}} |u|^{\alpha-2}u|v|^{\beta}+ \frac{1}{2}Q_{u}(u,v) & \text{ in} ~ \Omega, \\
            (-\Delta)^{s}v - \lambda_{2} { \frac{v}{|x|^{2s}} } - |v|^{2_{s}^{*}-2}v = \frac{n\beta}{2_{s}^{*}} |u|^{\alpha}|v|^{\beta-2}v + \frac{1}{2}Q_{v}(u,v) &  \text{ in} ~ \Omega,\\
             u=v=0 ~~\text{in}~~\mathbb{R}^{N} \backslash \Omega,
     \end{cases}
\end{equation}
where $\lambda_{1}, \lambda_{2} \in (0, \Lambda_{N,s})$ and $2_{s}^{*} = \frac{2N}{N-2s}$ is the fractional critical Sobolev exponent. The existence of positive solutions to the systems through variational methods was ascertained for the critical case, i.e., $\alpha+\beta = 2_{s}^{*}$ on the bounded domain $\Omega$.

Motivated by the aforementioned articles and their results, we are interested in finding out the existence of positive solutions to the system (\ref{main problem}). As far as we know, there is no literature in this direction concerning the critical case on the whole domain $\mathbb{R}^N$. { The lack of compactness due to the nonlinearities in the source terms and the singular Hardy potential terms make it delicate to employ the variational methods to the problem \eqref{main problem}.} To deal with such kind of non-compactness, we will use the concentration compactness principle for fractional problems in unbounded domains considered in \cite[Bonder et al.]{Bonder2018}, \cite[Chen et al.]{Chen2018} and \cite[Pucci and Temperini]{Pucci2021}, and the Mountain pass theorem. These concentration compactness principles are fractional analogous to the celebrated concentration compactness principles discussed by P. L. Lions \cite{lions1985Partone,lions1985Parttwo}.
\begin{remark}
In case $s_1=s_2$, the order between the parameters $\lambda_1$ and $\lambda_2$ determines the order between the \textit{semi-trivial} energy levels. Indeed, if $\lambda_2>\lambda_1$ and $\nu$ is small enough, $J_\nu(0,z_{\mu,s}^{\lambda_2})= \frac{s}{N} S^{\frac{N}{2s}}(\lambda_2) < \frac{s}{N} S^{\frac{N}{2s}}(\lambda_1) = J_\nu(z_{\mu,s}^{\lambda_2},0)$ i.e., the pair $(0,z_{\mu,s}^{\lambda_2})$ is a ground state solution of (1.1), see Theorem \eqref{third theorem}. In this case, the order of $\lambda_1$ and $\lambda_2$ plays a vital for determining the existence of positive and ground-state \textit{semi-trivial} solutions. But in the case of $s_1 \neq s_2$, the order between $\lambda_1$ and $\lambda_2$ fails to determine the order between the \textit{semi-trivial} energy levels. Therefore, we find positive ground state solutions independent of the order between $\lambda_1$ and $\lambda_2$ while dealing with the system involving two different fractional Laplacians.
\end{remark} 
Our paper is organized in the following manner. First, we give some preliminary results and functional analysis settings in Section \ref{S2}. Further, Section \ref{S3} deals with the results in which the functional $J_\nu$ satisfies the Palais-Smale condition for both the cases, i.e., subcritical and critical cases. Also, the local behavior of the semi-trivial solutions is given in this section depending on the parameters $\alpha, \beta$, and $\nu$. In Section \ref{S4}, we prove the main results of this article concerned with positive bound and ground state solutions.
\section{Preliminaries and functional setting}\label{S2}

In this section, we give an appropriate variational setting for the system \eqref{main problem}. First, we define the energy functional $J_\nu$ associated with the system \eqref{main problem} given as
\begin{align} \label{energy functional}
\begin{split}
    J_{\nu}(u,v) &= \frac{1}{2} \iint_{\mathbb{R}^{2N}} \frac{|u(x)-u(y)|^{2}}{|x-y|^{N+2s_{1}}} \mathrm{d}x \mathrm{d}y + \frac{1}{2} \iint_{\mathbb{R}^{2N}} \frac{|v(x)-v(y)|^{2}}{|x-y|^{N+2s_{2}}} \mathrm{d}x \mathrm{d}y -  \frac{\lambda_{1}}{2} \int_{\mathbb{R}^{N}} \frac{u^{2}~~}{|x|^{2s_{1}}}\mathrm{d}x\\
    &~~~~~  - \frac{\lambda_{2}}{2} \int_{\mathbb{R}^{N}} \frac{v^{2}~~}{|x|^{2s_{2}}}\mathrm{d}x - \frac{1}{2_{s_{1}}^{*}}\int_{\mathbb{R}^{N}} |u|^{2_{s_{1}}^{*}}\mathrm{d}x  -\frac{1}{2_{s_{2}}^{*}}\int_{\mathbb{R}^{N}} |v|^{2_{s_{2}}^{*}}\mathrm{d}x - \nu \int_{\mathbb{R}^{N}}h(x)|u|^{\alpha}|v|^{\beta}\, \mathrm{d}x,
\end{split}
\end{align} 
defined on the product space $\mathbb{D} = \mathcal{D}^{s_{1},2}(\mathbb{R}^{N}) \times \mathcal{D}^{s_{2},2}(\mathbb{R}^{N})$. The space $\mathcal{D}^{s_i,2}(\mathbb{R}^{N}),~( i=1,2)$ is the closure of $C_{0}^{\infty}(\mathbb{R}^{N})$ with respect to the Gagliardo seminorm
\[ \|u\|_{s_i} := \bigg( \iint_{\mathbb{R}^{2N}} \frac{|u(x)-u(y)|^{2}}{|x-y|^{N+2s_{i}}}\,\mathrm{d}x\mathrm{d}y  \bigg)^{\frac{1}{2}},~~\text{for}~i=1,2.\]
We refer to the articles \cite{Brasco2021characterisation,Brasco2019note} by Brasco et al. for more details about the space $\mathcal{D}^{s_i,2}(\mathbb{R}^{N}),~( i=1,2)$. Further, we endow the following norm with the product space $\mathbb{D}$ given by
$$ \|(u,v)\|^{2}_{\mathbb{D}} = \|u\|_{\lambda_{1}, s_{1}}^{2} + \|v\|_{\lambda_{2}, s_{2}}^{2},  $$
where
$$ \|u\|_{\lambda_i, s_i}^{2} = \iint_{\mathbb{R}^{2N}} \frac{|u(x)-u(y)|^{2}}{|x-y|^{N+2s_{i}}}\,\mathrm{d}x\mathrm{d}y - \lambda_i \int_{\mathbb{R}^{N}}\frac{u^{2}}{|x|^{2s_i}}\,\mathrm{d}x,~\text{for}~i =1,2. $$
The above norm is well defined due to the fractional Hardy inequality \cite[Theorem 1.1]{Frank2008} given by
\begin{equation} \label{fractional Hardy inequality}
    \Lambda_{N,s_i} \int_{\mathbb{R}^{N}}\frac{{u^{2}}}{|x|^{2s_i}}\mathrm{d}x \leq \iint_{\mathbb{R}^{2N}} \frac{|u(x)-u(y)|^{2}}{|x-y|^{N+2s_{i}}} \mathrm{d}x\mathrm{d}y, ~\text{for}~i =1,2.
\end{equation}
where $\Lambda_{N,s_{i}} = 2 \pi^{N/2} \frac{\Gamma^{2}(\frac{N+2s_i}{4}) \Gamma(\frac{N+2s_i}{2})}{\Gamma^{2}(\frac{N-2s_i}{4}) ~|\Gamma(-s_{i})|}, (i=1,2)$ is the sharp constant for the inequality \eqref{fractional Hardy inequality}. We can note that the norms $\|\cdot\|_{\lambda_i,s_i}$ and $\|\cdot\|_{s_i}$ for any $\lambda_i \in (0, \Lambda_{N,s_i})~\text{with}~ i=1,2$ are equivalent due to the Hardy’s inequality \eqref{fractional Hardy inequality}. 

Let us recall that the solutions of \eqref{single problem} arise as minimizers $z^{\lambda_i}_{\mu,s_i}~(i=1,2)$ of the Rayleigh quotient given by (see \cite{Dipierro2016})
\begin{equation}\label{S lambda}
    S(\lambda_i) := \inf\limits_{u \in \mathcal{D}^{s_i,2}(\mathbb{R}^{N}), u \not\equiv 0} \frac{\|u\|_{\lambda_{i},s_i}^{2}}{\|u\|_{2_{s_i}^{*}}^{2}~~} = \frac{\|z^{\lambda_i}_{\mu,s_i}\|_{\lambda_{i},s_i}^{2}}{\|z^{\lambda_i}_{\mu,s_i}\|_{2_{s_i}^{*}}^{2}~~}, (i=1,2).
\end{equation}
Moreover, we have
\begin{equation} \label{extremal value at norms}
  \|z^{\lambda_i}_{\mu,s_i}\|_{\lambda_{i},s_i}^{2} = \|z^{\lambda_i}_{\mu,s_i}\|_{2_{s_i}^{*}}^{2_{s_i}^{*}} = S^{\frac{N}{2s_i}}(\lambda_i),~\text{for}~i=1,2. 
\end{equation}
If $\lambda_i =0$ for $i=1,2$, then { $S(\lambda_i)=S_i$} which is known to be achieved by the extremals of the type {$C(N,s_i)(1+|x|^2)^{-\frac{N-2s_i}{2}}$} (see \cite{Lieb2002sharp}), where {$C(N,s_i)$} is a positive constant depending on $N$ and {$s_i$} only. { We write $S_1=S_2=S$ when $s_1=s_2.$}
 { From \eqref{S lambda} we have the following Sobolev embeddings
\begin{equation} \label{embeddings}
\begin{split}\begin{cases}
    S(\lambda_1) \|u\|_{2_{s_i}^{*}}^{2} \leq \|u\|_{\lambda_{1},s_1}^{2},\\
    S(\lambda_2) \|v\|_{2_{s_2}^{*}}^{2} \leq \|v\|_{\lambda_{2},s_2}^{2}.\end{cases}
\end{split}
\end{equation} 
By applying H\"{o}lder's inequality, it follows that
\begin{align} \label{f2}
\displaystyle |J_{\nu}(u,v)|\leq
\begin{cases}
  \frac{1}{2} \|(u,v)\|_{\mathbb{D}}^2 + \frac{1}{2_{s_{1}}^{*}} \|u\|_{2_{s_{1}}^{*}}^{2_{s_{1}}^{*}}  +\frac{1}{2_{s_{2}}^{*}}\|v\|_{2_{s_{2}}^{*}}^{2_{s_{2}}^{*}} + \nu \|h\|_{1}^{1-\frac{\alpha}{2_{s_1}^{*}} - \frac{\beta}{2_{s_2}^{*}}} \|h\|_{\infty}^{\frac{\alpha}{2_{s_1}^{*}} + \frac{\beta}{2_{s_2}^{*}}} \|u\|_{2_{s_1}^{*}}^{\alpha} \|v\|_{2_{s_2}^{*}}^{\beta},\quad \text{if} \frac{\alpha}{2_{s_1}^{*}} + \frac{\beta}{2_{s_2}^{*}} <1,\\
         \frac{1}{2} \|(u,v)\|_{\mathbb{D}}^2 + \frac{1}{2_{s_{1}}^{*}} \|u\|_{2_{s_{1}}^{*}}^{2_{s_{1}}^{*}}  +\frac{1}{2_{s_{2}}^{*}}\|v\|_{2_{s_{2}}^{*}}^{2_{s_{2}}^{*}} + \nu \|h\|_{\infty}  \|u\|_{2_{s_1}^{*}}^{\alpha} \|v\|_{2_{s_2}^{*}}^{\beta}, \quad \text{if} \frac{\alpha}{2_{s_1}^{*}} + \frac{\beta}{2_{s_2}^{*}} =1.
\end{cases}
\end{align} 
In both cases the right-hand side is finite for every $(u,v) \in \mathbb{D}$ due to the Sobolev embeddings given by \eqref{embeddings} and thanks to \eqref{condition on h}. Hence, the functional $J_\nu$ is well-defined on the product space $\mathbb{D}$.
}
Let us re-write the functional $J_{\nu}(u,v)$ as 
\begin{equation}
    J_{\nu}(u,v) = J_{\lambda_1}(u_{}) + J_{\lambda_2}(v_{}) - \nu \int_{\mathbb{R}^{N}}h(x)|u|^{\alpha}|v|^{\beta}\mathrm{d}x,
\end{equation}
where
\begin{align} \label{value of functional componentwise}
    J_{\lambda_i}(u_{}) = \frac{1}{2} \iint_{\mathbb{R}^{2N}} \frac{|u(x)-u(y)|^{2}}{|x-y|^{N+2s_{i}}} \,\mathrm{d}x\mathrm{d}y  -  \frac{\lambda_{i}}{2} \int_{\mathbb{R}^{N}} \frac{u^{2}}{|x|^{2s_{i}}}\,\mathrm{d}x  - \frac{1}{2_{s_{i}}^{*}}\int_{\mathbb{R}^{N}} |u|^{2_{s_{i}}^{*}}\,\mathrm{d}x,~ \text{for}~i=1,2.
\end{align}
{  It is easy to verify that the functional $J_\nu$ is Fr\'{e}chet differentiable on $\mathbb{D}$. The functional is given by
\begin{align*}
    J_{\nu}(u,v) &= \frac{1}{2} A(u,v) - B(u,v)-\nu I(u,v),
\end{align*} 
where $ A(u,v) = \|(u,v)\|_{\mathbb{D}}^2,~ B(u,v)= \frac{1}{2_{s_{1}}^{*}} \|u\|_{2_{s_{1}}^{*}}^{2_{s_{1}}^{*}}  +\frac{1}{2_{s_{2}}^{*}}\|v\|_{2_{s_{2}}^{*}}^{2_{s_{2}}^{*}} ,~ I(u,v)= \int_{\mathbb{R}^{N}}h(x)|u|^{\alpha}|v|^{\beta}\, dx$. If $A,B,I \in C^{1}$ on the product space $\mathbb{D}$, then the functional $J_\nu$ is also in $C^1$ on $\mathbb{D}$. It is clear that $A \in C^1$ as it is the square of a norm on $\mathbb{D}$. By following a similar approach as given in \cite[Lemma 1]{Baldelli2021}, we can prove that the functionals $B$ and $I$ are in $C^1$ on the product space $\mathbb{D}$. \\ }
 For  $(u_0,v_0) \in \mathbb{D}$, the Fr\'{e}chet derivative of $J_\nu$ at $(u,v) \in \mathbb{D}$ is given as follow 
\begin{align*}
     \langle J'_{\nu}(u,v) | (u_0,v_0)\rangle &= \iint_{\mathbb{R}^{2N}} \frac{(u(x)-u(y))(u_0(x)-u_0(y))}{|x-y|^{N+2s_{1}}} \,\mathrm{d}x\mathrm{d}y  \\ &\quad+ \iint_{\mathbb{R}^{2N}} \frac{(v(x)-v(y))(v_0(x)-v_0(y))}{|x-y|^{N+2s_{2}}} \,\mathrm{d}x\mathrm{d}y  - \lambda_1 \int_{\mathbb{R}^{N}} \frac{u\cdot u_0}{|x|^{2s_{1}}}\,\mathrm{d}x - \lambda_2 \int_{\mathbb{R}^{N}} \frac{v \cdot v_0}{|x|^{2s_2}}\,\mathrm{d}x  \\&\quad- \int_{\mathbb{R}^{N}} |u|^{{2^{*}_{s_{1}}}-2} u \cdot u_0 \,\mathrm{d}x
        - \int_{\mathbb{R}^{N}} |v|^{{2^{*}_{s_{2}}}-2} v \cdot v_0 \,\mathrm{d}x\\
       &~~~~~ - \nu \alpha \int_{\mathbb{R}^{N}}h(x)|u|^{\alpha -2}u \cdot u_0|v|^{\beta}\,\mathrm{d}x
        - \nu \beta \int_{\mathbb{R}^{N}}h(x)|u|^{\alpha}|v|^{\beta-2}v \cdot v_0 \,\mathrm{d}x,
\end{align*}
where $J'_{\nu}(u,v)$ is the Fr\'{e}chet derivative of $J_\nu$ at $(u,v)$, and the duality bracket between the product space $\mathbb{D}$ and its dual $\mathbb{D}^*$ is represented as $\langle \cdot,\cdot \rangle$.
From \eqref{energy functional} and for any $\tau>0$, we get
\begin{align} 
    J_{\nu}(\tau u,\tau v) &= \frac{\tau^2}{2} \iint_{\mathbb{R}^{2N}} \frac{|u(x)-u(y)|^{2}}{|x-y|^{N+2s_{1}}} \mathrm{d}x\mathrm{d}y + \frac{\tau^2}{2} \iint_{\mathbb{R}^{2N}} \frac{|v(x)-v(y)|^{2}}{|x-y|^{N+2s_{2}}} \mathrm{d}x\mathrm{d}y -  \frac{\lambda_{1} \tau^2}{2} \int_{\mathbb{R}^{N}} \frac{u^{2}~~}{|x|^{2s_{1}}}\mathrm{d}x \notag\\
    &\quad - \frac{\lambda_{2}\tau^2}{2} \int_{\mathbb{R}^{N}} \frac{v^{2}~~}{|x|^{2s_{2}}}\mathrm{d}x - \frac{\tau^{2_{s_{1}}^{*}}}{2_{s_{1}}^{*}}\int_{\mathbb{R}^{N}} |u|^{2_{s_{1}}^{*}}\mathrm{d}x  -\frac{\tau^{2_{s_{2}}^{*}}}{2_{s_{2}}^{*}}\int_{\mathbb{R}^{N}} |v|^{2_{s_{2}}^{*}}\mathrm{d}x - \nu \tau^{\alpha+\beta}\int_{\mathbb{R}^{N}}h(x)|u|^{\alpha}|v|^{\beta}\mathrm{d}x.
\end{align} 
Clearly, $J_{\nu}(\tau u_{},\tau v_{}) \rightarrow -\infty ~\text{as}~\tau \rightarrow +\infty$ which implies that the functional $J_\nu$ is unbounded from below on $\mathbb{D}$. Here the concept of the Nehari manifold plays its role in minimizing the functional for finding the critical point in $\mathbb{D}$ by using a variational approach. We introduce the Nehari manifold $\mathcal{N}_{\nu}$ associated with the functional $J_\nu$ as
\[\mathcal{N}_{\nu} = \{  (u, v) \in \mathbb{D} \backslash \{(0, 0)\} : \Phi_{\nu}(u,v) = 0    \},\]
where 
\begin{equation} \label{phi function}
    \Phi_{\nu}(u,v) = \langle J'_{\nu}(u,v) | (u,v) \rangle.
\end{equation}
We can see that all the critical points $(u,v) \in \mathbb{D}\backslash \{(0,0)\}$ of the energy functional $J_\nu$ lie in the set $\mathcal{N}_{\nu}$. On Nehari manifolds, we recall some well-known facts for the reader's convenience.

Let $(u,v)$ be an element of the Nehari manifold $\mathcal{N}_{\nu}$. Then the following holds:
\begin{align} \label{equivalent norm}
\begin{split}
    \|(u,v)\|_{\mathbb{D}}^{2} 
    = \|u\|_{{2_{s_{1}}^{*}}}^{2_{s_{1}}^{*}} + \|v\|_{{2_{s_{2}}^{*}}}^{2_{s_{2}}^{*}} + \nu (\alpha + \beta) \int_{\mathbb{R}^{N}} h(x)|u|^{\alpha}|v|^{\beta}\mathrm{d}x.
\end{split}
\end{align}
If we restrict the functional $J_\nu$ on the Nehari manifold $\mathcal{N}_{\nu}$, the functional takes the following form
\begin{align} \label{energy functional on Nehari manifold}
    J_{\nu}|_{\mathcal{N}_{\nu}}(u,v) = \frac{s_{1}}{N} \|u\|_{{2_{s_{1}}^{*}}}^{2_{s_{1}}^{*}} + \frac{s_{2}}{N} \|v\|_{{2_{s_{2}}^{*}}}^{2_{s_{2}}^{*}} + \nu \bigg(  \frac{\alpha+\beta -2}{2}  \bigg) \int_{\mathbb{R}^{N}} h(x)|u|^{\alpha}|v|^{\beta}\mathrm{d}x.
\end{align}
Now suppose that $(\tau u,\tau v) \in \mathcal{N}_{\nu}$ for all $(u, v) \in \mathbb{D} \backslash \{(0, 0)\}$. Then using \eqref{equivalent norm} we get the following 
\begin{align} \label{Algebraic equation}
     \|(u,v)\|_{\mathbb{D}}^{2} =  \tau^{2_{s_{1}}^{*} -2} \|u\|_{{2_{s_{1}}^{*}}}^{2_{s_{1}}^{*}} +  \tau^{2_{s_{2}}^{*} -2} \|v\|_{{2_{s_{2}}^{*}}}^{2_{s_{2}}^{*}} + \nu (\alpha+\beta ) \tau^{\alpha+ \beta -2} \int_{\mathbb{R}^{N}} h(x)|u|^{\alpha}|v|^{\beta}\mathrm{d}x.
\end{align}
The above equation is an algebraic equation in $\tau$ and a cautious analysis of equation \eqref{Algebraic equation} shows that this algebraic equation has a unique positive solution. Thus, we can infer that there exists a unique positive $\tau= \tau_{(u,v)}$ such that $(\tau u,\tau v) \in \mathcal{N}_{\nu}$ for all $(u, v) \in \mathbb{D} \backslash \{(0, 0)\}$. By combining (\ref{equivalent norm}) and (\ref{ alpha beta condition}) we obtain that, for any $(u,v) \in \mathcal{N}_{\nu}$ 
\begin{align*} 
       J''_{\nu}(u,v)[u,v]^{2} &= \langle \Phi'_{\nu}(u,v) | (u,v) \rangle  \\
        &= 2 \iint_{\mathbb{R}^{2N}} \frac{|u(x)-u(y)|^{2}}{|x-y|^{N+2s_{1}}} \mathrm{d}x\mathrm{d}y + 2 \iint_{\mathbb{R}^{2N}} \frac{|v(x)-v(y)|^{2}}{|x-y|^{N+2s_{2}}} \mathrm{d}x\mathrm{d}y  
    -  2\lambda_{1} \int_{\mathbb{R}^{N}} \frac{u^{2}}{|x|^{2s_{1}}}\mathrm{d}x\\
    &~~~~ - 2\lambda_{2}\int_{\mathbb{R}^{N}} \frac{v^{2}}{|x|^{2s_{2}}}dx -2_{s_{1}}^{*} \int_{\mathbb{R}^{N}} |u|^{2_{s_{1}}^{*}}\mathrm{d}x 
    -2_{s_{2}}^{*}\int_{\mathbb{R}^{N}} |v|^{2_{s_{2}}^{*}}\mathrm{d}x -\\ &\quad- \nu (\alpha+\beta)^{2} \int_{\mathbb{R}^{N}}h(x)|u|^{\alpha}|v|^{\beta}\mathrm{d}x\\ 
    &= 2 \|(u,v)\|_{\mathbb{D}}^{2}-2_{s_{1}}^{*} \|u\|_{2_{s_{1}}^{*}}^{2_{s_{1}}^{*}} 
    -2_{s_{2}}^{*}\|v\|_{2_{s_{2}}^{*}}^{2_{s_{2}}^{*}} - \nu (\alpha+\beta)^{2} \int_{\mathbb{R}^{N}}h(x)|u|^{\alpha}|v|^{\beta}\mathrm{d}x.
    \end{align*} Further calculations give us
    \begin{align} \label{second order derivative}
    J''_{\nu}(u,v)[u,v]^{2}&=(2-\alpha-\beta) \|(u,v)\|_{\mathbb{D}}^{2} + (\alpha + \beta) ( \|u\|_{{2_{s_{1}}^{*}}}^{2_{s_{1}}^{*}}+ \|v\|_{{2_{s_{2}}^{*}}}^{2_{s_{2}}^{*}} )
        - 2_{s_{1}}^{*} \|u\|_{{2_{s_{1}}^{*}}}^{2_{s_{1}}^{*}} - 2_{s_{2}}^{*}\|v\|_{{2_{s_{2}}^{*}}}^{2_{s_{2}}^{*}}\notag\\
       & =(2-\alpha-\beta) \|(u,v)\|_{\mathbb{D}}^{2} + (\alpha + \beta -2_{s_{1}}^{*} )  \|u\|_{{2_{s_{1}}^{*}}}^{2_{s_{1}}^{*}}+ (\alpha + \beta -2_{s_{2}}^{*})\|v\|_{{2_{s_{2}}^{*}}}^{2_{s_{2}}^{*}} ~< 0 .
 \end{align}
Further, by using \eqref{equivalent norm} we can prove the existence of a constant $r_{\nu} > 0$ such that
\begin{align} \label{ norm equal  r}
    \|(u,v)\|_{\mathbb{D}} > r_{\nu} ~~\mbox{for~all}~(u,v) \in \mathcal{N}_{\nu}.
\end{align}
Now by the Lagrange multiplier method, if $(u, v) \in \mathbb{D}$ is a critical point of $J_{\nu}$ on the Nehari manifold $ \mathcal{N}_{\nu}$,  then there exists a $\rho \in \mathbb{R}$ called Lagrange multiplier such that
\[ (J_{\nu}|_{\mathcal{N}_{\nu} } )'(u,v) =  J_{\nu}'(u,v) - \rho \Phi'_{\nu}(u,v) = 0.\]
Thus from the above, we calculate that $\rho \langle \Phi'_{\nu}(u,v)|(u,v) \rangle = \langle J'_{\nu}(u,v)|(u,v) \rangle =0 $. It is clear that $\rho=0$, otherwise the inequality (\ref{second order derivative}) fails to hold and as a result $J_{\nu}'(u,v) = 0$. Hence, there is a one-to-one correspondence between the critical points of $J_{\nu}$ and the critical points of $J_{\nu}|_{\mathcal{N}_{\nu}}$. The functional $J_{\nu}$ restricted on the Nehari manifold ${\mathcal{N}_{\nu}}$ is also written as
\begin{equation} \label{two one four}
        ( J_{\nu}|_{\mathcal{N}_{\nu}})(u,v) = \bigg( \frac{1}{2} - \frac{1}{\alpha+\beta}  \bigg) \|(u,v)\|_{\mathbb{D}}^{2} +  \frac{1}{\alpha+\beta} \big( \|u\|_{{2_{s_{1}}^{*}}}^{2_{s_{1}}^{*}}+ \|v\|_{{2_{s_{2}}^{*}}}^{2_{s_{2}}^{*}} \big)  -  \frac{1}{2_{s_{1}}^{*}} \|u\|_{{2_{s_{1}}^{*}}}^{2_{s_{1}}^{*}} - \frac{1}{2_{s_{2}}^{*}}\|v\|_{{2_{s_{2}}^{*}}}^{2_{s_{2}}^{*}}. 
\end{equation} 
Thus, combining the hypotheses (\ref{ alpha beta condition}) and (\ref{ norm equal  r}) with \eqref{two one four}, we deduce
\[J_{\nu}(u,v) > \bigg( \frac{1}{2} - \frac{1}{\alpha+\beta}  \bigg) r_{\nu}^{2} ~~\mbox{for~all}~(u,v) \in \mathcal{N}_{\nu}.\]
We come to the conclusion that the functional $J_\nu$ restricted on $\mathcal{N}_{\nu}$ is bounded from below. Hence, we continue our study to get the solution of \eqref{main problem} by minimizing the energy functional $J_\nu$ on the Nehari manifold $\mathcal{N}_\nu$.
\begin{definition}
If $(u_1,v_1) \in \mathbb{D} \backslash \{(0, 0)\}$ is a critical point of $J_\nu$ over $\mathbb{D}$, then we say that the pair $(u_1,v_1)$ is a bound state solution of \eqref{main problem}. This bound state solution $(u_{1},v_{1})$ is said to be a ground state solution if its energy is minimal among all the bound state solutions i.e.
 \begin{equation} \label{ground state level}
     c_{\nu} = J_{\nu}(u_{1},v_{1}) = \min \{ J_{\nu}(u,v): (u, v) \in \mathbb{D} \backslash \{(0, 0)\} ~\text{and}~ J'_{\nu}(u,v)=0 \}.
 \end{equation}
\end{definition}
\section{The Palais-Smale Condition}\label{S3}
\begin{lemma} \label{equivalent of critical points lemma}
Let us assume that (\ref{ alpha beta condition}) and (\ref{condition on h}) are satisfied and also that $\{(u_{n},v_{n})\} \subset \mathcal{N}_{\nu}$ is a Palais-Smale sequence for $J_{\nu}$ restricted on the Nehari manifold ${\mathcal{N}_{\nu}}$ at level $c \in \mathbb{R}$, then $\{(u_{n},v_{n})\}$ is a bounded (PS) sequence for $J_{\nu}$ in $\mathbb{D}$, i.e.,
\begin{equation} \label{palais smale condition}
    J'_{\nu}(u_{n},v_{n}) \rightarrow 0 ~\mbox{as}~n \rightarrow \infty~\mbox{in the dual space}~\mathbb{D}^{*}.
\end{equation}
\end{lemma}
\begin{proof}
Assume that $\{(u_{n},v_{n})\}\subset \mathcal{N}_{\nu} $ be a Palais-Smale sequence for $J_{\nu}$ at level $c$, then 
\begin{align}
    \begin{split}
        J(u_{n},v_{n}) \rightarrow c~~as~~n \rightarrow \infty ,~\text{i.e.,}~ c+ o(1) = J(u_{n},v_{n}) \hspace{1cm}
    \end{split}
\end{align}
and we recall that
\begin{align}
     J(u_{n},v_{n}) = \frac{1}{2} \|(u_{n},v_{n})\|_{\mathbb{D}}^{2} - \frac{1}{2_{s_{1}}^{*}}\|u_{n}\|_{{2_{s_{1}}^{*}}}^{2_{s_{1}}^{*}} - \frac{1}{2_{s_{2}}^{*}}\|v_{n}\|_{{2_{s_{2}}^{*}}}^{2_{s_{2}}^{*}} - \nu \int_{\mathbb{R}^{N}} h(x)|u_{n}|^{\alpha}|v_{n}|^{\beta}\mathrm{d}x.
\end{align}
For $(u_{n},v_{n}) \in \mathcal{N}_{\nu},$ we have
\begin{align}
      \|(u_{n},v_{n})\|_{\mathbb{D}}^{2} =\|u_{n}\|_{{2_{s_{1}}^{*}}}^{2_{s_{1}}^{*}} + \|v_{n}\|_{{2_{s_{2}}^{*}}}^{2_{s_{2}}^{*}} + \nu (\alpha + \beta) \int_{\mathbb{R}^{N}} h(x)|u_{n}|^{\alpha}|v_{n}|^{\beta}\mathrm{d}x 
\end{align}
By combining the above two equations, we get
\begin{align*}
    J(u_{n},v_{n}) &= \frac{1}{2} \|(u_{n},v_{n})\|_{\mathbb{D}}^{2} - \frac{1}{2_{s_{1}}^{*}}\|u_{n}\|_{{2_{s_{1}}^{*}}}^{2_{s_{1}}^{*}} - \frac{1}{2_{s_{2}}^{*}}\|v_{n}\|_{{2_{s_{2}}^{*}}}^{2_{s_{2}}^{*}}
    -\frac{1}{\alpha+\beta} \bigg( \|(u_{n},v_{n})\|_{\mathbb{D}}^{2}-\|u_{n}\|_{{2_{s_{1}}^{*}}}^{2_{s_{1}}^{*}} -\|v_{n}\|_{{2_{s_{2}}^{*}}}^{2_{s_{2}}^{*}} \bigg) \\
    &= \bigg(\frac{1}{2}  -\frac{1}{\alpha+\beta}\bigg) \|(u_{n},v_{n})\|_{\mathbb{D}}^{2}+ \bigg(  \frac{1}{\alpha+\beta}-\frac{1}{2_{s_{1}}^{*}}\bigg) \|u_{n}\|_{{2_{s_{1}}^{*}}}^{2_{s_{1}}^{*}} + \bigg(  \frac{1}{\alpha+\beta}-\frac{1}{2_{s_{2}}^{*}}\bigg)\|v_{n}\|_{{2_{s_{2}}^{*}}}^{2_{s_{2}}^{*}}.\\
  & \geq \bigg(\frac{1}{2}  -\frac{1}{\alpha+\beta}\bigg) \|(u_{n},v_{n})\|_{\mathbb{D}}^{2}
\end{align*}
Thus, we have
\[c + o(1) \geq \bigg(\frac{1}{2}  -\frac{1}{\alpha+\beta}\bigg) \|(u_{n},v_{n})\|_{\mathbb{D}}^{2}.\]
Thus the sequence $\{ (u_{n},v_{n})\}$ is bounded in $\mathbb{D}$. Furthermore, we deduce the following by considering the functional $\Phi_{\nu}$ given by (\ref{phi function}) with the inequalities (\ref{second order derivative}) and (\ref{ norm equal  r})  
\begin{align} \label{phi r inequality}
    \langle\Phi'_{\nu}(u_{n},v_{n})|(u_{n},v_{n})\rangle \leq (2-\alpha-\beta) r_{\nu}^{2}.
\end{align}
By the Lagrange multiplier method, we can assume the sequence of multipliers $\{ \omega_{n}\} \subset \mathbb{R}$ such that
\begin{align} \label{J restricted on N with LM}
     (J_{\nu}|_{\mathcal{N}_{\nu} } )'(u_{n},v_{n}) =  J_{\nu}'(u_{n},v_{n}) - \omega_{n} \Phi'_{\nu}(u_{n},v_{n}) ~\text{in the dual space}~\mathbb{D}^{*}.
\end{align}
Since $(J_{\nu}|_{\mathcal{N}_{\nu} } )'(u_{n},v_{n})$ converges to $0$ as $n \rightarrow \infty$ in the dual space $\mathbb{D}^{*}$, this implies that $$\langle (J_{\nu}|_{\mathcal{N}_{\nu} } )'  (u_{n},v_{n}) |  (u_{n},v_{n}) \rangle \to 0 \text{ as } n \rightarrow \infty. $$ This further implies that $ - \omega_{n} \langle \Phi'_{\nu}(u_{n},v_{n}) | (u_{n},v_{n}) \rangle  \to 0$. Finally, we know that $\langle \Phi'_{\nu}(u_{n},v_{n}) | (u_{n},v_{n}) \rangle <0$ and hence, we have $\omega_{n} \rightarrow 0 ~\text{in}~\mathbb{R}~\text{as}~n \rightarrow \infty$. Thus, (\ref{J restricted on N with LM}) directly implies (\ref{palais smale condition}).
\end{proof}
 Further, we prove the boundedness of the Palais-Smale sequence in $\mathbb{D}$.
\begin{lemma} \label{boundedness lemma}
Let us assume that (\ref{ alpha beta condition}) and (\ref{condition on h}) are satisfied and that $\{(u_{n},v_{n})\} \subset \mathbb{D}$ be a (PS) sequence for the functional $J_{\nu}$ at level $c \in \mathbb{R}$. Then the sequence $\{(u_{n},v_{n})\}$ is bounded in $\mathbb{D}$.
\end{lemma}
\begin{proof}
Given $\{(u_{n},v_{n})\}\subset \mathbb{D}$ is a (PS) sequence for $J_{\nu}$ at level $c$, then as $n \rightarrow \infty$
\begin{align}
    J_{\nu}(u_{n},v_{n}) &\rightarrow c~\text{in}~\mathbb{R} ,\label{e1}\\
     J'_{\nu}(u_{n},v_{n}) &\rightarrow 0 ~\text{in} ~\mathbb{D}^*. \label{e2}
\end{align}
Using \eqref{e2}, we can write
\begin{align*}
    \bigg< J'_{\nu}(u_{n},v_{n}) \bigg| \frac{(u_{n},v_{n})}{\|(u_{n},v_{n})\|_{\mathbb{D}}} \bigg> \rightarrow 0 ~\text{in}~\mathbb{R}.
\end{align*}
Thus from the above, we have
    \begin{align*}
     \|(u_{n},v_{n})\|_{\mathbb{D}}^{2} -\|u_{n}\|_{{2_{s_{1}}^{*}}}^{2_{s_{1}}^{*}} - \|v_{n}\|_{{2_{s_{2}}^{*}}}^{2_{s_{2}}^{*}} - \nu (\alpha + \beta) \int_{\mathbb{R}^{N}} h(x)|u_{n}|^{\alpha}|v_{n}|^{\beta}\mathrm{d}x = o(\|(u_{n},v_{n})\|_{\mathbb{D}})~\text{as}~n \rightarrow \infty.
   \end{align*}
Also, from \eqref{e1} we obtain the following
\begin{align*}
   \frac{1}{2} \|(u_{n},v_{n})\|_{\mathbb{D}}^{2} - \frac{1}{2_{s_{1}}^{*}}\|u_{n}\|_{{2_{s_{1}}^{*}}}^{2_{s_{1}}^{*}} - \frac{1}{2_{s_{2}}^{*}}\|v_{n}\|_{{2_{s_{2}}^{*}}}^{2_{s_{2}}^{*}} - \nu \int_{\mathbb{R}^{N}} h(x)|u_{n}|^{\alpha}|v_{n}|^{\beta}\mathrm{d}x = c+o(1) ~\text{as}~n \rightarrow \infty.
\end{align*}
Thus, we can write
\begin{align*}
   J_{\nu}(u_{n},v_{n}) - \frac{1}{\alpha+\beta} \langle J'_{\nu}(u_{n},v_{n}) | { \frac{(u_{n},v_{n})}{\|(u_{n},v_{n})\|_{\mathbb{D}}} }\rangle = c + o(1)+ o(\|(u_{n},v_{n})\|_{\mathbb{D}}) ~\text{as}~n \rightarrow \infty,
\end{align*}
and hence,
\begin{align*}
    \bigg(\frac{1}{2}  -\frac{1}{\alpha+\beta} \bigg) \|(u_{n},v_{n})\|_{\mathbb{D}}^{2} &\leq \bigg(\frac{1}{2}  -\frac{1}{\alpha+\beta} \bigg) \|(u_{n},v_{n})\|_{\mathbb{D}}^{2}+ \\ &\quad +\bigg(\frac{1}{\alpha+\beta} -\frac{1}{{2_{s_{1}}^{*}}} \bigg) \|u_{n}\|_{{2_{s_{1}}^{*}}}^{2_{s_{1}}^{*}}+ \bigg(\frac{1}{\alpha+\beta} -\frac{1}{{2_{s_{2}}^{*}}} \bigg)\|v_{n}\|_{{2_{s_{2}}^{*}}}^{2_{s_{2}}^{*}}\\
    &= c + o(1)+ o(\|(u_{n},v_{n})\|_{\mathbb{D}}) ~\text{as}~ n\rightarrow \infty.
\end{align*}
Thus we can conclude that the sequence $\{(u_{n},v_{n})\}$ is bounded in $\mathbb{D}$.
\end{proof}
Now we derive the non-local version of Lemma 3.3 in \cite{Abdellaoui2009}.
\begin{lemma}\label{Abdelloui fractional version}
Let $C,D>0$ and $\delta \geq 2$ be fixed . Also, assume that for any $\nu >0$
$$ T_{\nu} = \{ \varrho \in \mathbb{R}^{+} ~~|~~ C\varrho^{\frac{N-2s}{N}} \leq \varrho + D\nu \varrho^{\frac{\delta}{2} \big(\frac{N-2s}{N}\big)} \}.$$
Then for every $\epsilon>0,$ there is a $\nu_{1}>0$ depending only on $\epsilon,C,D,\nu,N$ and $s$ such that $$ \inf{T_{\nu}} \geq (1-\epsilon)C^{\frac{N}{2s}}~\text{for all}~0<\nu<\nu_{1}.$$
\end{lemma}
\begin{proof}
For $\varrho \in T_{\nu}$,
\begin{align*}
    &C\varrho^{\frac{N-2s}{N}} \leq \varrho + D\nu \varrho^{\frac{\delta}{2} \big(\frac{N-2s}{N}\big)} = \varrho + D\nu \varrho^{\frac{\delta}{2_{s}^{*}}}\\
     &C\varrho^{1-\frac{2s}{N} -\frac{\delta}{2_{s}^{*}} } - \varrho^{1-\frac{\delta}{2_{s}^{*}}} \leq D\nu \\
      &F(\varrho) \leq D\nu, ~\text{where}~ F(\varrho) = C\varrho^{1-\frac{2s}{N} -\frac{\delta}{2_{s}^{*}} } - \varrho^{1-\frac{\delta}{2_{s}^{*}}}. 
\end{align*}
Hence, we notice that $T_{\nu}  = \{ \varrho \in \mathbb{R}^{+}~|~ F(\varrho) \leq D\nu \}$. Also observe that $F(C^{\frac{N}{2s}}) = 0$ and $$ F'(\varrho) = \varrho^{-\frac{\delta}{2_{s}^{*}}} \bigg[ \bigg( \frac{2-\delta}{2_{s}^{*}}\bigg)C\varrho^{-\frac{2s}{N}} - \bigg( \frac{2_{s}^{*}-\delta}{2_{s}^{*}} \bigg) \bigg].$$
If $\delta \leq 2_{s}^{*}$, then $F'(\varrho) <0$ i.e. the function $F$ is strictly decreasing function. If $\delta > 2_{s}^{*}$ and $ F'(\varrho) = 0$ then $$ \varrho = \bigg(  \frac{C(\delta-2)}{\delta-2_{s}^{*}}\bigg)^{\frac{N}{2s}}> A^{\frac{N}{2s}},$$
which implies that $F$ has a global negative minimum at $\varrho = \big(  \frac{C(\delta-2)}{\delta-2_{s}^{*}}\big)^{\frac{N}{2s}}> C^{\frac{N}{2s}}$  and $F$ tends to $0$ as $\varrho \rightarrow +\infty.$ In any case, $F$ is strictly decreasing in $(0,C^{\frac{N}{2s}}]$ and it has only one zero at $C^{\frac{N}{2s}}$ with $\lim_{\varrho \rightarrow 0^{+}}F(\varrho) = +\infty$ and $F(\varrho) < 0 $ in $(C^{\frac{N}{2s}}, +\infty)$. Hence, $\inf{T_{\nu}} = F^{-1}(D\nu) \rightarrow C^{\frac{N}{2s}}~as~\nu \rightarrow 0^{+}$ and the conclusion follows.
\end{proof}
\subsection{The case \texorpdfstring{$\alpha +\beta < \min\{2_{s_{1}}^{*},2_{s_{2}}^{*}\}$}{TEXT}}
In the following, we prove the Palais-Smale compactness condition of the functional $J_\nu$ at level $c$.
\begin{lemma} \label{PS compactness lemma}
Suppose $\alpha +\beta < \min\{2_{s_{1}}^{*},2_{s_{2}}^{*}\}$ and (\ref{condition on h}). Then, the functional $J_{\nu}$ satisfies the (PS) condition for any level $c$ satisfying
\begin{align} \label{energy level PS}
    c <  \min \bigg\{\frac{s_{1}}{N}S^{\frac{N}{2s_{1}}}(\lambda_{1}), \frac{s_{2}}{N} S^{\frac{N}{2s_{2}}}(\lambda_{2})\bigg\}.
\end{align}
\end{lemma}
\begin{proof}
We know by Lemma \ref{boundedness lemma} that any {(PS)} sequence $\{(u_{n},v_{n})\}$ is bounded in $\mathbb{D}$. So, there exists a subsequence denoted by $\{(u_{n},v_{n})\}$ itself and a $(\Tilde{u},\Tilde{v})\in \mathbb{D}$ satisfying the following
\begin{align*}
    (u_{n},v_{n}) &\rightharpoonup (\Tilde{u},\Tilde{v}) ~\text{weakly in}~\mathbb{D}, \\
    (u_{n},v_{n}) &\rightarrow (\Tilde{u},\Tilde{v})~\text{strongly in}~ L_{loc}^{q_{1}}(\mathbb{R}^{N}) \times L_{loc}^{q_{2}}(\mathbb{R}^{N})  ~\text{for}~1\leq q_{1} < 2_{s_{1}}^{*},1\leq q_{2} < 2_{s_{2}}^{*},\\
    (u_{n},v_{n}) &\rightarrow (\Tilde{u},\Tilde{v})  ~\text{a.e.~in}~~\mathbb{R}^{N}.  
\end{align*}
Now by using the concentration–compactness principle of Bonder \cite[Theorem 1.1]{Bonder2018}, Chen \cite[Lemma 4.5]{Chen2018} and an analogous version of Pucci \cite[Theorem 1.2]{Pucci2021}, there exist a subsequence, still denoted as $\{(u_{n},v_{n})\}$, two at most countable sets of points $\{x_{j}\}_{j \in \mathcal{J}} \subset \mathbb{R}^{N}$ and $\{y_{k}\}_{k \in \mathcal{K}} \subset \mathbb{R}^{N}$, and non-negative numbers $$\{(\mu_{j},\rho_{j})\}_{j \in \mathcal{J}},~ \{(\Bar{\mu}_{k},~\Bar{\rho}_{k})\}_{k \in \mathcal{K}},\mu_{0},~\rho_{0},~\gamma_{0},~\Bar{\mu}_{0},~\Bar{\rho}_{0} \text{ and }\Bar{\gamma}_{0}$$ such that the following convergences hold $weakly^*$ in the sense of measures,
\begin{align} \label{Concentration compactness}
\begin{split}
    |D^{s_{1}}u_{n}|^{2}~& \rightharpoonup \mathrm{d}\mu \geq  |D^{s_{1}}\Tilde{u}|^{2} + \Sigma_{j \in \mathcal{J}} \mu_{j}\delta_{x_{j}} + \mu_{0}\delta_{0},\\
    |D^{s_{2}}v_{n}|^{2}~ &\rightharpoonup \mathrm{d}\Bar{\mu} \geq  |D^{s_{2}}\Tilde{v}|^{2} + \Sigma_{k \in \mathcal{K}} \Bar{\mu}_{k}\delta_{y_{k}} + \Bar{\mu}_{0}\delta_{0},\\
    |u_{n}|^{2_{s_{1}}^{*}}~ &\rightharpoonup \mathrm{d}\rho =|\Tilde{u}|^{2_{s_{1}}^{*}} + \Sigma_{j \in \mathcal{J}} \rho_{j}\delta_{x_{j}} + \rho_{0}\delta_{0}, \\
     |v_{n}|^{2_{s_{2}}^{*}}~ &\rightharpoonup \mathrm{d}\Bar{\rho} =|\Tilde{v}|^{2_{s_{2}}^{*}} + \Sigma_{k \in \mathcal{K}} \Bar{\rho}_{k}\delta_{y_{k}} + \Bar{\rho}_{0}\delta_{0}, \hspace{0.6cm}\\
     \frac{u_{n}^{2}}{|x|^{2s_{1}}} ~ &\rightharpoonup \mathrm{d}\gamma = \frac{\Tilde{u}^{2}}{|x|^{2s_{1}}} + \gamma_{0}\delta_{0}, \\
     \frac{v_{n}^{2}}{|x|^{2s_{2}}} ~ &\rightharpoonup \mathrm{d}\Bar{\gamma} = \frac{\Tilde{v}^{2}}{|x|^{2s_{2}}} + \Bar{\gamma_{0}}\delta_{0},
\end{split}
\end{align}
where $\delta_{0},\delta_{x_{j}},\delta_{y_{k}}$ are the Dirac functions at the points $0,x_{j}~\text{and}~y_{k}$ of $\mathbb{R}^{N}$ respectively. Now from inequality (1.6) of \cite[Theorem 1.1]{Bonder2018} and inequalities (1.7) of \cite[Theorem 1.2]{Pucci2021}, we deduce the inequalities given below
\begin{align} \label{inequality with S}
    \begin{split}
        { S_1} \rho_{j}^{\frac{2}{2_{s_{1}}^{*}}} &\leq \mu_{j} ~\text{for all} ~j \in \mathcal{J} \cup \{ 0\},\\
        { S_2} \Bar{\rho}_{k}^{\frac{2}{2_{s_{2}}^{*}}} &\leq \Bar{\mu}_{k} ~\text{for all}~ k \in \mathcal{K} \cup \{ 0\},
    \end{split}
\end{align}
Also, by taking $\alpha = 2s_{1}$ and $\alpha = 2s_{2}$ in the inequality (4.21) of \cite[Lemma 4.5]{Chen2018} we deduce that
\begin{align} \label{inequality with lambda}
    \begin{split} 
       \Lambda_{N,s_{1}} \gamma_{0} &\leq \mu_{0},  \\
        \Lambda_{N,s_{2}} \Bar{\gamma}_{0} &\leq \Bar{\mu}_{0}.  
    \end{split} 
\end{align}
We denote the concentration of the sequence $\{ u_{n} \}$ at infinity by the following numbers
\begin{align} \label{Concentration at infty}
    \begin{split}
        \rho_{\infty} &= \lim_{R \rightarrow \infty} \limsup_{n \rightarrow \infty} \int_{|x|>R} |u_{n}|^{2_{s_{1}}^{*}} \mathrm{d}x, \\
        \mu_{\infty} &= \lim_{R \rightarrow \infty} \limsup_{n \rightarrow \infty} \int_{|x|>R} |D^{s_{1}}u_{n}|^{2} \mathrm{d}x,\\
        \gamma_{\infty} &= \lim_{R \rightarrow \infty} \limsup_{n \rightarrow \infty} \int_{|x|>R} \frac{u_{n}^{2}}{|x|^{2s_{1}}} \mathrm{d}x. \\
    \end{split}
\end{align}
In a similar way, we can define the concentrations of the sequence $\{ v_{n} \}$ at infinity by the numbers ${\Bar{\mu}_{\infty}}, {\Bar{\rho}_{\infty}}$ and ${\Bar{\gamma}_{\infty}}$.\

Further, we assume that the function $\Psi_{j,\epsilon}(x)$ is a smooth cut-off function centered at points $\{x_{j}\}$, $j \in \mathcal{J}$, satisfying
\begin{align} \label{test function phi at j}
    \Psi_{j,\epsilon} = 1 ~\text{in}~~B_{\frac{\epsilon}{2}}(x_{j}),~~ \Psi_{j,\epsilon} = 0 ~~\text{in}~~B^{c}_{\epsilon}(x_{j}), ~0\leq \Psi_{j,\epsilon} \leq 1~\text{and}~~|\nabla  \Psi_{j,\epsilon}| \leq \frac{4}{\epsilon},
\end{align}
where $B_{r}(x_{j}) = \{y \in \mathbb{R}^N: |y-x_j|<r\}$. Now, testing $J'_{\nu}(u_{n},v_{n})$ with $(u_{n}\Psi_{j,\epsilon}, 0)$ we get
\begin{align} \label{functional with first component}
    \begin{split}
       0 &= \lim_{n \rightarrow +\infty} \big< J'_{\nu}(u_{n},v_{n})| (u_{n}\Psi_{j,\epsilon}, 0)\big> \\
       &= \lim_{n \rightarrow +\infty} \bigg( \iint_{\mathbb{R}^{2N}} \frac{|u_{n}(x)-u_{n}(y)|^{2}}{|x-y|^{N+2s_{1}}} \Psi_{j,\epsilon}(x) \, \mathrm{d}x \mathrm{d}y + \\ &\quad +\iint_{\mathbb{R}^{2N}} \frac{(u_{n}(x)-u_{n}(y))(\Psi_{j,\epsilon}(x)-\Psi_{j,\epsilon}(y))}{|x-y|^{N+2s_{1}}} u_{n}(y)\, \mathrm{d}x \mathrm{d}y -\lambda_{1} \int_{\mathbb{R}^{N}}\frac{u_{n}^{2}}{|x|^{2s_{1}}}\Psi_{j,\epsilon}(x)\, \mathrm{d}x -\\&\quad- \int_{\mathbb{R}^{N}} |u_{n}|^{2_{s_{1}}^{*}}\Psi_{j,\epsilon}(x)\,\mathrm{d}x - \nu \alpha \int_{\mathbb{R}^{N}} h(x)|u_{n}|^{\alpha}|v_{n}|^{\beta}\Psi_{j,\epsilon}(x)\, \mathrm{d}x \bigg) \\
      & = \int_{\mathbb{R}^{N}} \Psi_{j,\epsilon}\, \mathrm{d}\mu - \lambda_{1} \int_{\mathbb{R}^{N}}\Psi_{j,\epsilon}\, \mathrm{d}\gamma - \int_{\mathbb{R}^{N}}\Psi_{j,\epsilon}\, \mathrm{d}\rho \ +\\ &\quad+\lim_{n \rightarrow  + \infty} \iint_{\mathbb{R}^{2N}} \frac{(u_{n}(x)-u_{n}(y))(\Psi_{j,\epsilon}(x)-\Psi_{j,\epsilon}(y))}{|x-y|^{N+2s_{1}}} u_{n}(y)\, \mathrm{d}x \mathrm{d}y \\
      &\quad - \nu \alpha \lim_{n \rightarrow 
       + \infty}  \int_{\mathbb{R}^{N}} h(x)|u_{n}|^{\alpha}|v_{n}|^{\beta}\Psi_{j,\epsilon}(x)\, \mathrm{d}x. \hspace{6.5cm}
    \end{split}
\end{align}
Notice that $0 \notin \text{supp}(\Psi_{j,\epsilon})$ for $\epsilon$ being sufficiently small and also it is given that $h \in L^{1}(\mathbb{R}^{N}) \cap L^{\infty}(\mathbb{R}^{N})$.
Now we evaluate each of the integrals mentioned above when $\epsilon$ tends to $0$. 

\begin{align*}
     \int_{\mathbb{R}^{N}} \Psi_{j,\epsilon}\, \mathrm{d}\mu &\geq \int_{\mathbb{R}^{N}} \Psi_{j,\epsilon}|D^{s_{1}}\Tilde{u}|^{2} \,\mathrm{d}x + \Sigma_{j \in \mathcal{J}} \mu_{j}\delta_{x_{j}}(\Psi_{j,\epsilon}) + \mu_{0}\delta_{0}(\Psi_{j,\epsilon})\\
     &= \int_{\mathbb{R}^{N}} \Psi_{j,\epsilon}|D^{s_{1}}\Tilde{u}|^{2} \,\mathrm{d}x + \Sigma_{j \in \mathcal{J}} \mu_{j}\Psi_{j,\epsilon}(x_j) + \mu_{0}\Psi_{j,\epsilon}(0).
\end{align*}
Taking the limit $\epsilon \rightarrow 0$ and since $0 \notin \text{supp}(\Psi_{j,\epsilon})$ for $\epsilon$ being sufficiently small, we get 
\begin{equation} \label{first integral}
    \lim\limits_{\epsilon \rightarrow 0} \int_{\mathbb{R}^{N}} \Psi_{j,\epsilon}\, \mathrm{d}\mu \geq \mu_{j}.
\end{equation}
and
\begin{align*}
     \int_{\mathbb{R}^{N}} \Psi_{j,\epsilon}\, \mathrm{d}\rho &= \int_{\mathbb{R}^{N}} |\Tilde{u}|^{2_{s_1}^*} \Psi_{j,\epsilon} \,\mathrm{d}x + \Sigma_{j \in \mathcal{J}} \rho_{j}\delta_{x_{j}}(\Psi_{j,\epsilon}) + \rho_{0}\delta_{0}(\Psi_{j,\epsilon})\\
     &= \int_{\mathbb{R}^{N}} |\Tilde{u}|^{2_{s_1}^*}\Psi_{j,\epsilon} \,\mathrm{d}x + \Sigma_{j \in \mathcal{J}} \rho_{j}\Psi_{j,\epsilon}(x_j) + \rho_{0}\Psi_{j,\epsilon}(0).
\end{align*}
Taking the limit $\epsilon \rightarrow 0$ and since $0 \notin \text{supp}(\Psi_{j,\epsilon})$ for $\epsilon$ being sufficiently small, we get 
\begin{equation} \label{third integral}
    \lim\limits_{\epsilon \rightarrow 0} \int_{\mathbb{R}^{N}} \Psi_{j,\epsilon}\, \mathrm{d}\rho = \rho_{j}.
\end{equation}
Further,
\begin{equation} \label{second integral}
    \lim\limits_{\epsilon \rightarrow 0}  \int_{\mathbb{R}^{N}} \Psi_{j,\epsilon}\, \mathrm{d}\gamma = \lim\limits_{\epsilon \rightarrow 0} \bigg( \int_{\mathbb{R}^{N}} \Psi_{j,\epsilon} \frac{\Tilde{u}^{2}}{|x|^{2s_{1}}}\, \mathrm{d}x + \gamma_{0}\Psi_{j,\epsilon}(0) \bigg) = 0.
\end{equation}
Next, we claim that:
\begin{equation} \label{fourth integral}
   \lim\limits_{\epsilon \rightarrow 0} \lim_{n \rightarrow  + \infty} \iint_{\mathbb{R}^{2N}} \frac{(u_{n}(x)-u_{n}(y))(\Psi_{j,\epsilon}(x)-\Psi_{j,\epsilon}(y))}{|x-y|^{N+2s_{1}}} u_{n}(y) \,\mathrm{d}x \mathrm{d}y = 0.
\end{equation}
Let
\begin{align*}
    \mathcal{I} &= \iint_{\mathbb{R}^{2N}} \frac{(u_{n}(x)-u_{n}(y))(\Psi_{j,\epsilon}(x)-\Psi_{j,\epsilon}(y))}{|x-y|^{N+2s_{1}}} u_{n}(y) \,\mathrm{d}x \mathrm{d}y \\ &=  \iint_{\mathbb{R}^{2N}} \frac{(u_{n}(x)-u_{n}(y))[(\Psi_{j,\epsilon}(x)-\Psi_{j,\epsilon}(y)) u_{n}(y)]}{|x-y|^{N+2s_{1}}} \,\mathrm{d}x \mathrm{d}y.
\end{align*}
Since $\{u_n \Psi_{j,\epsilon}\}_{n \in \mathbb{N}}$ is a bounded sequence in $\mathcal{D}^{s_1,2}(\mathbb{R}^N)$, by using the H\"{o}lder's inequality we obtain
\begin{align*}
    \mathcal{I} &\leq \bigg( \iint_{\mathbb{R}^{2N}} \frac{|u_{n}(x)-u_{n}(y)|^2}{|x-y|^{N+2s_{1}}} \,\mathrm{d}x \mathrm{d}y \bigg)^{\frac{1}{2}} \bigg( \iint_{\mathbb{R}^{2N}} \frac{|\Psi_{j,\epsilon}(x)-\Psi_{j,\epsilon}(y)|^2|u_{n}(y)|^2}{|x-y|^{N+2s_{1}}} \,\mathrm{d}x \mathrm{d}y \bigg)^{\frac{1}{2}} \\
    &\leq C \bigg( \iint_{\mathbb{R}^{2N}} \frac{|\Psi_{j,\epsilon}(x)-\Psi_{j,\epsilon}(y)|^2|u_{n}(y)|^2}{|x-y|^{N+2s_{1}}} \,\mathrm{d}x \mathrm{d}y \bigg)^{\frac{1}{2}}.
\end{align*}
By Lemma 2.2, Lemma 2.4 of Bonder et al. \cite{Bonder2018}, and Lemma 2.3 of Xiang et al. \cite{Xiang2016}, we obtain that
\[ \lim\limits_{\epsilon \rightarrow 0} \lim_{n \rightarrow  + \infty} \iint_{\mathbb{R}^{2N}} \frac{|\Psi_{j,\epsilon}(x)-\Psi_{j,\epsilon}(y)|^2|u_{n}(y)|^2}{|x-y|^{N+2s_{1}}} \,\mathrm{d}x \mathrm{d}y =0 ,\]
which implies that $ \lim\limits_{\epsilon \rightarrow 0} \lim\limits_{n \rightarrow  + \infty} \mathcal{I} = 0$. Hence, the claim \eqref{fourth integral} is done. Now we show that
\begin{equation} \label{Fifth integral}
    \lim\limits_{\epsilon \rightarrow 0} \lim_{n \rightarrow 
       + \infty}  \int_{\mathbb{R}^{N}} h(x)|u_{n}|^{\alpha}|v_{n}|^{\beta}\Psi_{j,\epsilon}(x)\, \mathrm{d}x=0.
\end{equation}
We notice that $\alpha + \beta < \min\{2_{s_1}^*, 2_{s_2}^* \}$ implies that $\frac{\alpha}{2_{s_1}^*}+ \frac{\beta}{2_{s_2}^*}<1$. Therefore, by applying the H\"{o}lder's inequality, we have
\begin{align*}
\int_{\mathbb{R}^{N}}h(x)|u_{n}|^{\alpha}|v_{n}|^{\beta}\Psi_{j,\epsilon}(x)\,\mathrm{d}x &= \int_{\mathbb{R}^{N}}(h(x)\Psi_{j,\epsilon}(x))^{1-\frac{\alpha}{2_{s_{1}}^{*}} - \frac{\beta}{2_{s_{2}}^{*}}} (h(x) \Psi_{j,\epsilon}(x))^{\frac{\alpha}{2_{s_{1}}^{*}} + \frac{\beta}{2_{s_{2}}^{*}} }|u_{n}|^{\alpha}|v_{n}|^{\beta}\,\mathrm{d}x\\
      &\leq \bigg( \int_{\mathbb{R}^{N}}h(x)\Psi_{j,\epsilon}(x)\,\mathrm{d}x\bigg)^{1-\frac{\alpha}{2_{s_{1}}^{*}} - \frac{\beta}{2_{s_{2}}^{*}}}  \bigg(\int_{\mathbb{R}^{N}}h(x)|u_{n}|^{2_{s_{1}}^{*}}\Psi_{j,\epsilon}(x)\,\mathrm{d}x\bigg)^{\frac{\alpha}{2_{s_{1}}^{*}}}\\ 
      &~~~~~~~~~ \bigg(\int_{\mathbb{R}^{N}}h(x)|v_{n}|^{2_{s_{2}}^{*}}\Psi_{j,\epsilon}(x)\,\mathrm{d}x\bigg)^{\frac{\beta}{2_{s_{2}}^{*}}}.
  \end{align*}
The first integral on the RHS of the last inequality tends to $0$ as $\epsilon \rightarrow 0$, and the rest two integrals are bounded as the limit $n \rightarrow \infty$. Therefore, first taking the limit as $n \rightarrow \infty$ and then taking $\epsilon \rightarrow 0$, we get \eqref{Fifth integral}. 
Now from (\ref{functional with first component}--\ref{Fifth integral}), one leads to the conclusion that $\mu_{j} - \rho_{j} \leq 0$ as $\epsilon \rightarrow 0$. Then using (\ref{inequality with S}), we have
\begin{align} \label{rho concentartion}
    \text{either} ~ \rho_{j} = 0,~~\text{or} ~~ \rho_{j} \geq {S_1}^{\frac{N}{2s_{1}}},~~\text{for all} ~j\in \mathcal{J}~\text{and}~\mathcal{J}~\text{is~finite}.
\end{align}
Analogously, we can also conclude that 
\begin{align} \label{rho bar concentartion}
    \text{either} ~ \Bar{\rho}_{k} = 0,~~\text{or} ~~ \Bar{\rho}_{k} \geq {S_2}^{\frac{N}{2s_{2}}},~~\text{for all}~ k\in \mathcal{K}~\text{and}~\mathcal{K}~\text{is~finite}.
\end{align}
Now for studying the concentration at origin, we consider a cut-off function $\Psi_{0,\epsilon}$ which satisfies the assumption
(\ref{test function phi at j}). Again, testing $J'_{\nu}(u_{n},v_{n})$ with $(u_{n}\Psi_{0,\epsilon}, 0)$ and following the analogous approach, we can easily deduce that $\mu_{0} - \lambda_{1}\gamma_{0}-\rho_{0} \leq 0$ and $\Bar{\mu}_{0} - \lambda_{1}\Bar{\gamma}_{0}-\Bar{\rho}_{0} \leq 0$. Moreover, using the inequality (1.7) of \cite[Theorem 1.2]{Pucci2021}, we have
\begin{align} \label{functional with first component at origin}
     \mu_{0} - \lambda_{1}\gamma_{0} \geq S(\lambda_{1}) \rho_{0}^{\frac{2}{2_{s_{1}}^{*}}} ~~\text{and}~~
     \Bar{\mu}_{0} - \lambda_{2}\Bar{\gamma}_{0} \geq S(\lambda_{2})\Bar{\rho}_{0}^{\frac{2}{2_{s_{2}}^{*}}},
\end{align}
which further implies that
\begin{align} \label{rho node inequalities}
    \begin{split}
     \text{either}~~ \rho_{0} = 0 ~~\text{or}~~\rho_{0} \geq S^{\frac{N}{2s_{1}}}(\lambda_{1}),\\
     \text{either}~~\Bar{\rho}_{0} = 0 ~~\text{or}~~\Bar{\rho}_{0} \geq S^{\frac{N}{2s_{2}}}(\lambda_{2}).
    \end{split}
\end{align}
Next for concentration at the point $\infty$, we choose $\mathcal{R}>0$ large enough so that $\{x_{j}\}_{j \in \mathcal{J}} \cup \{0\}$ is contained in $ B_{\mathcal{R}}(0)$ and we  consider a cut-off function $\Psi_{\infty,\epsilon}$ supported in a neighbourhood of $\infty$ satisfying the following
\begin{align}\label{phi at infinity}
 \Psi_{\infty,\epsilon} = 0 ~~\text{in}~~B_{\mathcal{R}}(0),~~ \Psi_{\infty,\epsilon} = 1 ~~\text{in}~~B^{c}_{\mathcal{R}+1}(0),~0\leq \Psi_{\infty,\epsilon} \leq 1 ~\text{and}~|\nabla  \Psi_{\infty,\epsilon}| \leq \frac{4}{\epsilon}.
\end{align}
 Analogously, it is easy to find that $\mu_{\infty} - \lambda_{1}\gamma_{\infty}-\rho_{\infty} \leq 0$ and $\Bar{\mu}_{\infty} - \lambda_{1}\Bar{\gamma}_{\infty}-\Bar{\rho}_{\infty} \leq 0$ by testing $J'_{\nu}(u_{n},v_{n})$ with $(u_{n}\Psi_{\infty,\epsilon}, 0)$. Next, by the inequality (1.14) of \cite[Theorem 1.3]{Pucci2021} we have
\begin{align} \label{functional with first component at infty}
     \mu_{\infty} - \lambda_{1}\gamma_{\infty} \geq S(\lambda_{1}) \rho_{\infty}^{\frac{2}{2_{s_{1}}^{*}}} ~~\text{and}~~
     \Bar{\mu}_{\infty} - \lambda_{2}\Bar{\gamma}_{\infty} \geq S(\lambda_{2})\Bar{\rho}_{\infty}^{\frac{2}{2_{s_{2}}^{*}}},
\end{align}
which also further concludes that
\begin{align} \label{rho infty inequalities}
    \begin{split}
     \text{either}~ \rho_{\infty} = 0 ~~\text{or}~~\rho_{\infty} \geq S^{\frac{N}{2s_{1}}}(\lambda_{1}),\\
    \text{either}~ \Bar{\rho}_{\infty} = 0 ~~\text{or}~~\Bar{\rho}_{\infty} \geq S^{\frac{N}{2s_{2}}}(\lambda_{2}).
    \end{split}
\end{align}
As we already know that
\begin{align} \label{value of c}
     c = \bigg(\frac{1}{2}  -\frac{1}{\alpha+\beta} \bigg) \|(u_{n},v_{n})\|_{\mathbb{D}}^{2}+ \bigg(\frac{1}{\alpha+\beta} - \frac{1}{2_{s_{1}}^{*}} \bigg) \|u_{n}\|_{{2_{s_{1}}^{*}}}^{2_{s_{1}}^{*}}+\bigg(\frac{1}{\alpha+\beta} - \frac{1}{2_{s_{2}}^{*}} \bigg) \|v_{n}\|_{{2_{s_{2}}^{*}}}^{2_{s_{2}}^{*}} +o(1).
\end{align}
Now using (\ref{Concentration compactness}-\ref{inequality with lambda}), (\ref{rho node inequalities}) and (\ref{rho infty inequalities}), we deduce that
\begin{align}\label{ inequalities with c}
      c &\geq \bigg(\frac{1}{2}  -\frac{1}{\alpha+\beta} \bigg)\bigg(\|(\Tilde{u},\Tilde{v})\|^{2}_{\mathbb{D}} + \sum\limits_{j\in \mathcal{J}}\mu_{j} + (\mu_{0} - \lambda_{1}\gamma_{0}) + (\mu_{\infty} - \lambda_{1}\gamma_{\infty})\notag\\
      &\quad+ \sum\limits_{k \in \mathcal{K}} \Bar{\mu}_{k} + 
      ( \Bar{\mu}_{0} - \lambda_{2}\Bar{\gamma}_{0}) + (\Bar{\mu}_{\infty} - \lambda_{2}\Bar{\gamma}_{\infty}) \bigg)\notag\\
      &\quad+ \bigg( \frac{1}{\alpha+\beta}-\frac{1}{2_{s_{1}}^{*}} \bigg) \bigg(\int_{\mathbb{R}^{N}}|\Tilde{u}|^{2_{s_{1}}^{*}}\,\mathrm{d}x + \sum\limits_{j\in \mathcal{J}}\rho_{j} +\rho_{0} +\rho_{\infty}  \bigg) \notag\\
      &\quad+ \bigg( \frac{1}{\alpha+\beta}-\frac{1}{2_{s_{2}}^{*}} \bigg) \bigg(\int_{\mathbb{R}^{N}}|\Tilde{v}|^{2_{s_{2}}^{*}}\,\mathrm{d}x  +  \sum\limits_{k \in \mathcal{K}} \Bar{\rho}_{k} + \Bar{\rho}_{0} + \Bar{\rho}_{\infty}\bigg) \\
     &\geq \bigg(\frac{1}{2}  -\frac{1}{\alpha+\beta} \bigg)\bigg(  \bigg[ {S_1}\sum\limits_{j\in \mathcal{J}}\rho_{j}^{\frac{2}{2_{s_{1}}^{*}}} + {S_2} \sum\limits_{k \in \mathcal{K}} \Bar{\rho}_{k}^{\frac{2}{2_{s_{2}}^{*}}}\bigg] + S(\lambda_{1}) \bigg[ \rho_{0}^{\frac{2}{2_{s_{1}}^{*}}}
    +{\rho}_{\infty}^{\frac{2}{2_{s_{2}}^{*}}}\bigg] + S(\lambda_{2}) \bigg[ \Bar{\rho}_{0}^{\frac{2}{2_{s_{1}}^{*}}} +\Bar{\rho}_{\infty}^{\frac{2}{2_{s_{2}}^{*}}}\bigg] \bigg) \notag\\
      &\quad+ \bigg( \frac{1}{\alpha+\beta}-\frac{1}{2_{s_{1}}^{*}} \bigg) \bigg( \sum\limits_{j\in \mathcal{J}}\rho_{j} +\rho_{0} +\rho_{\infty}\bigg) + \bigg( \frac{1}{\alpha+\beta}-\frac{1}{2_{s_{2}}^{*}} \bigg) \bigg(\sum\limits_{k \in \mathcal{K}} \Bar{\rho}_{k} + \Bar{\rho}_{0} + \Bar{\rho}_{\infty} \bigg). \notag
\end{align}
If the concentration is considered at the point $x_j$, then $\rho_{j}>0$ and further form above and using (\ref{rho concentartion}) we find that
\begin{align*}
    c \geq \bigg(\frac{1}{2}  -\frac{1}{\alpha+\beta} \bigg) {S_1}^{1+ \frac{N}{2s_{1}} \frac{2}{2_{s_{1}}^{*}}} + \bigg( \frac{1}{\alpha+\beta}-\frac{1}{2_{s_{1}}^{*}} \bigg) {S_1}^{\frac{N}{2s_{1}}} = \frac{s_{1}}{N}{S_1}^{\frac{N}{2s_{1}}}.
\end{align*}
Which contradicts the assumption on energy level given by (\ref{energy level PS}). This leads to $\rho_{j} = \mu_{j} = 0$ for all $j \in \mathcal{J}$. Following the similar way, we further deduce that $\Bar{\rho}_{k} = \Bar{\mu}_{k} = 0$ for all $k \in \mathcal{K}$. If $\rho_{0} \neq 0$, from \eqref{ inequalities with c} and (\ref{rho node inequalities}), we have
\begin{align*}
    c \geq \frac{s_{1}}{N}S^{\frac{N}{2s_{1}}}(\lambda_{1}),
\end{align*}
which again contradicts the hypothesis on the energy level c. Therefore, $\rho_{0} = 0$. By the same token, we also get $\Bar{\rho}_{0} = 0$. Arguing as above and using (\ref{rho infty inequalities}) we also find $\rho_{\infty} = 0$ and $\Bar{\rho}_{\infty} = 0$. Hence, there exists a subsequence that strongly converges in $L^{2^{*}_{s_{1}}}(\mathbb{R}^{N}) \times L^{2^{*}_{s_{2}}}(\mathbb{R}^{N})$. As a consequence, we have 
\begin{align*}
    \|(u_{n} - \Tilde{u}, v_{n} - \Tilde{v})\|^{2}_{\mathbb{D}} = \langle J'_{\nu}(u_{n},v_{n})| (u_{n} - \Tilde{u}, v_{n} - \Tilde{v}) \rangle + o_n(1),
\end{align*}
which infers that the sequence $\{ (u_{n},v_{n})\}$ strongly converges in $\mathbb{D}$ and the (PS) condition holds.
\end{proof}
Now we consider the modified version of the problem \eqref{main problem} to deal with the positive solutions.
\begin{equation} \label{modified problem}
    \left\{
        \begin{array}{ll}
            (-\Delta)^{s_{1}} u - \lambda_{1} \frac{u}{|x|^{2s_{1}}} - (u^{+})^{2_{s_{1}}^{*}-1} = \nu \alpha h(x) (u^{+})^{\alpha-1}(v^{+})^{\beta} & \quad \text{in} ~ \mathbb{R}^{N}, \\
            (-\Delta)^{s_{2}} v - \lambda_{2} \frac{v}{|x|^{2s_{2}}} - (v^{+})^{2_{s_{2}}^{*}-1} = \nu \beta h(x) (u^{+})^{\alpha}(v^{+})^{\beta-1} & \quad \text{in} ~ \mathbb{R}^{N},
        \end{array}
    \right.
\end{equation}
where $u^{+} = \max\{u,0\}$ and $u^{-} = \min\{u,0\}$ such that $u = u^{+} + u^{-}.$ It is clear that the solutions of (\ref{main problem}) are also the solutions of the modified problem (\ref{modified problem}). The modified problem also has a structure of the variational type, and therefore, the solutions are merely the critical points of the associated functional on $\mathbb{D}$ given by
\begin{align} \label{functional associated with modified problem}
   J^{+}_{\nu}(u_{},v_{}) = \frac{1}{2} \|(u_{},v_{})\|_{\mathbb{D}}^{2} - \frac{1}{2_{s_{1}}^{*}}\|u^{+}\|_{{2_{s_{1}}^{*}}}^{2_{s_{1}}^{*}} - \frac{1}{2_{s_{2}}^{*}}\|v^{+}\|_{{2_{s_{2}}^{*}}}^{2_{s_{2}}^{*}} - \nu \int_{\mathbb{R}^{N}} h(x)(u^{+})^{\alpha}(v^{+})^{\beta}\,\mathrm{d}x. 
\end{align}
{ From the above we have
\begin{align*}
    |J^{+}_{\nu}(u,v)| &\leq \frac{1}{2} \|(u,v)\|_{\mathbb{D}}^2 + \frac{1}{2_{s_{1}}^{*}} \|u^+\|_{2_{s_{1}}^{*}}^{2_{s_{1}}^{*}}  + \frac{1}{2_{s_{2}}^{*}}\|v^+\|_{2_{s_{2}}^{*}}^{2_{s_{2}}^{*}} + \nu \int_{\mathbb{R}^{N}}h(x)(u^+)^{\alpha}(v^+)^{\beta}\, \mathrm{d}x \\
    & \leq \frac{1}{2} \|(u,v)\|_{\mathbb{D}}^2 + \frac{1}{2_{s_{1}}^{*}} \|u\|_{2_{s_{1}}^{*}}^{2_{s_{1}}^{*}}  + \frac{1}{2_{s_{2}}^{*}}\|v\|_{2_{s_{2}}^{*}}^{2_{s_{2}}^{*}} + \nu \int_{\mathbb{R}^{N}}h(x)|u|^{\alpha}|v|^{\beta}\, \mathrm{d}x.
\end{align*} 
By applying H\"{o}lder's inequality, it follows that
\begin{align} 
\displaystyle |J^+_{\nu}(u,v)|\leq
\begin{cases}
  \frac{1}{2} \|(u,v)\|_{\mathbb{D}}^2 + \frac{1}{2_{s_{1}}^{*}} \|u\|_{2_{s_{1}}^{*}}^{2_{s_{1}}^{*}}  +\frac{1}{2_{s_{2}}^{*}}\|v\|_{2_{s_{2}}^{*}}^{2_{s_{2}}^{*}} + \nu \|h\|_{1}^{1-\frac{\alpha}{2_{s_1}^{*}} - \frac{\beta}{2_{s_2}^{*}}} \|h\|_{\infty}^{\frac{\alpha}{2_{s_1}^{*}} + \frac{\beta}{2_{s_2}^{*}}} \|u\|_{2_{s_1}^{*}}^{\alpha} \|v\|_{2_{s_2}^{*}}^{\beta},\quad \text{if} \frac{\alpha}{2_{s_1}^{*}} + \frac{\beta}{2_{s_2}^{*}} <1,\\
         \frac{1}{2} \|(u,v)\|_{\mathbb{D}}^2 + \frac{1}{2_{s_{1}}^{*}} \|u\|_{2_{s_{1}}^{*}}^{2_{s_{1}}^{*}}  +\frac{1}{2_{s_{2}}^{*}}\|v\|_{2_{s_{2}}^{*}}^{2_{s_{2}}^{*}} + \nu \|h\|_{\infty}  \|u\|_{2_{s_1}^{*}}^{\alpha} \|v\|_{2_{s_2}^{*}}^{\beta}, \quad \text{if} \frac{\alpha}{2_{s_1}^{*}} + \frac{\beta}{2_{s_2}^{*}} =1.
\end{cases}
\end{align} 
In both the cases the right hand side is finite for every $(u,v) \in \mathbb{D}$ due to the the Sobolev embeddings given by \eqref{embeddings} and thanks to \eqref{condition on h}. Hence, the functional $J^+_\nu$ is well-defined on the product space $\mathbb{D}$. It is easy to verify that the functional $J^{+}_\nu$ is Fr\'{e}chet differentiable on $\mathbb{D}$. The functional is given by
\begin{align*}
    J^{+}_{\nu}(u,v) &= \frac{1}{2} A(u,v) - B(u^+,v^+)-\nu I(u^+,v^+),
\end{align*} 
where $ A(u,v) = \|(u,v)\|_{\mathbb{D}}^2,~ B(u^+,v^+)= \frac{1}{2_{s_{1}}^{*}} \|u^+\|_{2_{s_{1}}^{*}}^{2_{s_{1}}^{*}}  +\frac{1}{2_{s_{2}}^{*}}\|v^+\|_{2_{s_{2}}^{*}}^{2_{s_{2}}^{*}} ,~ I(u^+,v^+)= \int_{\mathbb{R}^{N}}h(x)(u^+)^{\alpha}(v^+)^{\beta}\, \mathrm{d}x$. If $A,B,I \in C^{1}$ on the product space $\mathbb{D}$, then the functional $J_\nu$ is also in $C^1$ on $\mathbb{D}$. It is clear that $A \in C^1$ as it is the square of a norm on $\mathbb{D}$. By following a similar approach as given in \cite[Lemma 1]{Baldelli2021}, we can prove that the functionals $B$ and $I$ are in $C^1$ on the product space $\mathbb{D}$. \\
 For  $(u_0,v_0) \in \mathbb{D}$, the Fr\'{e}chet derivative of $J^{+}_\nu$ at $(u,v) \in \mathbb{D}$ is given as follow 
\begin{align*}
     \langle (J^{+}_{\nu})'(u,v) | (u_0,v_0)\rangle &= \iint_{\mathbb{R}^{2N}} \frac{(u(x)-u(y))(u_0(x)-u_0(y))}{|x-y|^{N+2s_{1}}} \,\mathrm{d}x \mathrm{d}y  \\ &\quad+ \iint_{\mathbb{R}^{2N}} \frac{(v(x)-v(y))(v_0(x)-v_0(y))}{|x-y|^{N+2s_{2}}} \,\mathrm{d}x \mathrm{d}y  - \lambda_1 \int_{\mathbb{R}^{N}} \frac{u\cdot u_0}{|x|^{2s_{1}}}\,\mathrm{d}x - \lambda_2 \int_{\mathbb{R}^{N}} \frac{v \cdot v_0}{|x|^{2s_2}}\,\mathrm{d}x  \\&\quad- \int_{\mathbb{R}^{N}} (u^+)^{{2^{*}_{s_{1}}}-1}  u_0 \,\mathrm{d}x
        - \int_{\mathbb{R}^{N}} (v^{+})^{{2^{*}_{s_{2}}}-1}  v_0 \,\mathrm{d}x\\
       &~~~~~ - \nu \alpha \int_{\mathbb{R}^{N}}h(x)(u^+)^{\alpha -1}  u_0(v^+)^{\beta}\,\mathrm{d}x
        - \nu \beta \int_{\mathbb{R}^{N}}h(x)(u^+)^{\alpha}(v^+)^{\beta-1}  v_0 \,\mathrm{d}x.
\end{align*}}

 Besides, we shall denote $\mathcal{N}_{\nu}^{+}$ the Nehari manifold associated to $J_{\nu}^{+}$ as
$$\mathcal{N}^{+}_{\nu} = \{  (u, v) \in \mathbb{D} \backslash \{(0, 0)\} : \langle (J_{\nu}^{+})'(u,v) | (u,v) \rangle = 0    \}.$$
Also for every $(u,v) \in \mathcal{N}^{+}_{\nu}$, we have the following
\begin{align} \label{equivalent norm for modified problem}
    \|(u,v)\|_{\mathbb{D}}^{2} = \|u^{+}\|_{{2_{s_{1}}^{*}}}^{2_{s_{1}}^{*}} + \|v^{+}\|_{{2_{s_{2}}^{*}}}^{2_{s_{2}}^{*}} + \nu (\alpha + \beta) \int_{\mathbb{R}^{N}} h(x)(u^{+})^{\alpha}(v^{+})^{\beta}\,\mathrm{d}x,
\end{align}
and using this we can write the functional $J_{\nu}^{+}$ restricted on the Nehari manifold $\mathcal{N}^{+}_{\nu}$ as 
\begin{align} \label{restricted functonal on nehari manifold of modified problem}
    J_{\nu}^{+}|_{\mathcal{N}^{+}_{\nu}}(u,v) = \bigg(\frac{1}{2} - \frac{1}{\alpha+\beta}\bigg)\|(u,v)\|_{\mathbb{D}}^{2} + \bigg(\frac{1}{\alpha+\beta} - \frac{1}{2_{s_{1}}^{*}}\bigg)\|u^{+}\|_{{2_{s_{1}}^{*}}}^{2_{s_{1}}^{*}} + \bigg(\frac{1}{\alpha+\beta} - \frac{1}{2_{s_{2}}^{*}}\bigg) \|v^{+}\|_{{2_{s_{2}}^{*}}}^{2_{s_{2}}^{*}}.
\end{align}
In the following lemma, we prove strong convergence when certain Palais-Smale level conditions are imposed. 
\begin{lemma} \label{PS compactness lemma second}
 Assume $\alpha+\beta< \min\{2_{s_{1}}^{*},2_{s_{2}}^{*} \}$ and (\ref{condition on h}). Also let $\alpha \geq 2,~ S^{\frac{N}{2s_{1}}}(\lambda_{1}) \geq S^{\frac{N}{2s_{2}}}(\lambda_{2})$ and 
\begin{align} \label{Sum is less than S}
    S^{\frac{N}{2s_{1}}}(\lambda_{1})+S^{\frac{N}{2s_{2}}}(\lambda_{2}) < \min\{{S_1}^{\frac{N}{2s_{1}}},{S_2}^{\frac{N}{2s_{2}}}\}.
\end{align}
Then, there exists $\nu_0>0$ such that, if $0< \nu \leq \nu_0$ and $\{(u_{n},v_{n})\} \subset \mathbb{D}$ is a (PS) sequence for $J_{\nu}^{+}$ at level $c \in \mathbb{R}$ such that 
\begin{align}\label{c is less than sum}
    \frac{s_{1}}{N}S^{\frac{N}{2s_{1}}}(\lambda_{1}) < c < \frac{\min\{s_{1},s_{2}\}}{N}\bigg(S^{\frac{N}{2s_{1}}}(\lambda_{1})+S^{\frac{N}{2s_{2}}}(\lambda_{2})\bigg),
\end{align}
and 
\begin{align} \label{c is not equal modified problem}
    c \neq \frac{s_{2}l}{N}S^{\frac{N}{2s_{2}}}(\lambda_{2})~\text{for every}~l\in \mathbb{N}\backslash \{0\},
\end{align}
then there exists $ (\Tilde{u},\Tilde{v}) \in \mathbb{D}$ such that up to a subsequence $(u_{n},v_{n}) \rightarrow (\Tilde{u},\Tilde{v}) ~in~\mathbb{D} ~\text{as}~n \rightarrow \infty$.
\end{lemma}
\begin{proof}
Without loss of generality suppose that $s_{1} \geq s_{2}$. Following Lemma \ref{boundedness lemma} it is easy to prove that the (PS) sequence for $J_{\nu}^{+}$ is bounded in $\mathbb{D}$. Thus, there exist a $(\Tilde{u},\Tilde{v}) \in \mathbb{D}$ and a subsequence $\{(u_{n},v_{n})\}$ such that $\{(u_{n},v_{n})\}$ converges weakly to $(\Tilde{u},\Tilde{v})$ in $\mathbb{D}$. Further
\begin{align*}
    \big\langle (J_{\nu}^{+})'(u_{n},v_{n}) | (u^{-}_{n},0) \big\langle = \iint_{\mathbb{R}^{2N}} \frac{(u_{n}(x)-u_{n}(y))(u_{n}^{-}(x)-u_{n}^{-}(y))}{|x-y|^{N+2s_{1}}}\, \mathrm{d}x \mathrm{d}y - \lambda_{1} \int_{\mathbb{R}^{N}}\frac{u_{n}u_{n}^{-}}{|x|^{2s_{1}}}\, \mathrm{d}x.
\end{align*}
Since $u = u^{+}+u^{-}$ and we know that $u^{+}u^{-} = 0$, as both $u^{+}$ and $u^{-}$ can not be positive simultaneously. Hence,
\begin{align*}
     \big\langle (J_{\nu}^{+})'(u_{n},v_{n}) | (u^{-}_{n},0) \big\rangle &= \iint_{\mathbb{R}^{2N}} \frac{(u_{n}^{-}(x)-u_{n}^{-}(y))^{2}}{|x-y|^{N+2s_{1}}}\, \mathrm{d}x \mathrm{d}y + \iint_{\mathbb{R}^{2N}} \frac{\big(-u_{n}^{+}(x)u_{n}^{-}(y)-u_{n}^{+}(y) u_{n}^{-}(x)\big)}{|x-y|^{N+2s_{1}}}\, \mathrm{d}x \mathrm{d}y\\
     &~~~~~~~~~~~~~~~~~~~~~~ - \lambda_{1} \int_{\mathbb{R}^{N}}\frac{(u_{n}^{-})^{2}}{|x|^{2s_{1}}}\,\mathrm{d}x\\
     &\geq \iint_{\mathbb{R}^{2N}} \frac{(u_{n}^{-}(x)-u_{n}^{-}(y))^{2}}{|x-y|^{N+2s_{1}}}\, \mathrm{d}x \mathrm{d}y - \lambda_{1} \int_{\mathbb{R}^{N}}\frac{(u_{n}^{-})^{2}}{|x|^{2s_{1}}}\,\mathrm{d}x.
\end{align*}
From Hardy's inequality \eqref{fractional Hardy inequality}, we have 
\begin{align*}
      (1-\lambda_{1}C) \iint_{\mathbb{R}^{2N}} \frac{(u_{n}^{-}(x)-u_{n}^{-}(y))^{2}}{|x-y|^{N+2s_{1}}}\, \mathrm{d}x \mathrm{d}y \leq \iint_{\mathbb{R}^{2N}} \frac{(u_{n}^{-}(x)-u_{n}^{-}(y))^{2}}{|x-y|^{N+2s_{1}}}\, \mathrm{d}x \mathrm{d}y - \lambda_{1} \int_{\mathbb{R}^{N}}\frac{(u_{n}^{-})^{2}}{|x|^{2s_{1}}}\,\mathrm{d}x,
\end{align*}
where $C = 1/\Lambda_{N,s_{1}}$. Now using the fact that $(J_{\nu}^{+})'(u_{n},v_{n}) \rightarrow 0~\text{in}~\mathbb{D}^*$, we have $\big\langle (J_{\nu}^{+})'(u_{n},v_{n}) | (u^{-}_{n},0) \big\rangle \rightarrow 0 ~\text{as}~n \rightarrow \infty$, and hence we conclude from above inequalities that the sequence $\langle u^-_n \rangle \rightarrow 0$ strongly in $\mathcal{D}^{s_1,2}(\mathbb{R}^{N})$. Proceeding in a similar way, we also have that $\langle v_{n}^{-} \rangle \rightarrow 0 ~\text{strongly in}~\mathcal{D}^{s_2,2}(\mathbb{R}^{N})$.  Therefore, there is no loss for considering $\{(u_{n},v_{n})\}$ as a non-negative (PS) sequence at level $c$ for the functional $J_{\nu}$.\

By using the analogous approach of Lemma \ref{PS compactness lemma}, we can deduce the existence of a subsequence, still denoted by $\{(u_{n},v_{n})\}$, two (at most countable) sets of points $\{x_{j}\}_{j \in \mathcal{J}} \subset \mathbb{R}^{N}$ and $\{y_{k}\}_{k \in \mathcal{K}} \subset \mathbb{R}^{N}$ and non-negative numbers $\{(\mu_{j},\rho_{j})\}_{j \in \mathcal{J}}, \{(\Bar{\mu}_{k},\Bar{\rho}_{k})\}_{k \in \mathcal{K}},$ $\mu_{0},\rho_{0},\gamma_{0},\Bar{\mu}_{0},\Bar{\rho}_{0}$ and $\Bar{\gamma}_{0}$ such that the $weak^*$ convergence given by (\ref{Concentration compactness}) is satisfied as well as the inequalities (\ref{rho concentartion}-\ref{rho node inequalities}) also hold.\\
The concentration at infinity given by the numbers $\mu_{\infty},\rho_{\infty}, \Bar{\mu}_{\infty}$ and $\Bar{\rho}_{\infty}$ as in (\ref{Concentration at infty}), for which (\ref{functional with first component at infty}) and (\ref{rho infty inequalities}) hold, are also defined in a similar way.\

Next, we prove the strong convergence:
\begin{align} \label{component convergence either or case}
   \text{ either}~u_{n} \rightarrow \Tilde{u}~\text{in}~L^{2^{*}_{s_{1}}}(\mathbb{R}^{N})~\text{or}~v_{n} \rightarrow \Tilde{v}~\text{in}~L^{2^{*}_{s_{2}}}(\mathbb{R}^{N}).
\end{align}
The proof follows by using the method of contradiction. So we suppose that both the sequences $\{u_{n}\}$ and $\{v_{n}\}$ do not converge strongly in $L^{2^{*}_{s_{1}}}(\mathbb{R}^{N})$ and $L^{2^{*}_{s_{2}}}(\mathbb{R}^{N})$ respectively. Thus, there exists $j \in \mathcal{J}\cup\{0,\infty\}$ and  $k \in \mathcal{K}\cup\{0,\infty\}$ such that $\rho_{j}>0$ and $\Bar{\rho}_{k}>0$. Using the expressions,(\ref{rho concentartion}--\ref{rho node inequalities}) and (\ref{value of c}) into the equation \eqref{ inequalities with c}, we obtain
\begin{align*} 
     c &= \bigg(\frac{1}{2}  -\frac{1}{\alpha+\beta} \bigg) \|(u_{n},v_{n})\|_{\mathbb{D}}^{2}+ \bigg(\frac{1}{\alpha+\beta} - \frac{1}{2_{s_{1}}^{*}}  \bigg)  \|u_{n}\|_{{2_{s_{1}}^{*}}}^{2_{s_{1}}^{*}}+ \bigg(\frac{1}{\alpha+\beta} - \frac{1}{2_{s_{2}}^{*}}  \bigg)  \|v_{n}\|_{{2_{s_{2}}^{*}}}^{2_{s_{2}}^{*}}  +o(1)\\
     &\geq \bigg(\frac{1}{2}  -\frac{1}{\alpha+\beta} \bigg) \bigg(S(\lambda_{1})\rho_{j}^{\frac{2}{2_{s_{1}}^{*}}} + S(\lambda_{2})\Bar{\rho}_{k}^{\frac{2}{2_{s_{2}}^{*}}}\bigg) + \bigg(\frac{1}{\alpha+\beta} - \frac{1}{2_{s_{1}}^{*}} \bigg)\rho_{j}+ \bigg(\frac{1}{\alpha+\beta} - \frac{1}{2_{s_{2}}^{*}} \bigg)\Bar{\rho}_{k}\\
    & \geq \frac{s_{2}}{N}S^{\frac{N}{2s_{1}}}(\lambda_{1}) + \frac{s_{2}}{N} S^{\frac{N}{2s_{2}}}(\lambda_{2}).
\end{align*}
The aforementioned inequality contradicts assumption (\ref{c is less than sum}), so claim (\ref{component convergence either or case}) is proved. Thereafter, we prove that:
\begin{align} \label{component convergence either or case Dsp Space}
    \text{either}~u_{n} \rightarrow \Tilde{u}~\text{strongly~in}~\mathcal{D}^{s_{1},2}(\mathbb{R}^{N})~\text{or}~v_{n} \rightarrow \Tilde{v}~\text{strongly~in}~\mathcal{D}^{s_{2},2}(\mathbb{R}^{N}).
\end{align}
We assume by the claim (\ref{component convergence either or case}) that the sequence $\{u_{n}\}$ strongly converges in $L^{2_{s_{1}}^{*}}(\mathbb{R}^{N})$. This implies that we have the following convergence
\begin{align*}
    \|u_{n}-\Tilde{u}\|_{\lambda_{1},s_{1}}^{2} = \big\langle J'_{\nu}(u_{n},v_{n})|(u_{n}-\Tilde{u},0)\big\rangle +o_n(1),
\end{align*}
which clearly shows $u_{n} \rightarrow \Tilde{u}$ in $\mathcal{D}^{s_{1},2}(\mathbb{R}^{N})$. Again, if we suppose that  $\{v_{n}\}$ strongly converges in $L^{2_{s_{2}}^{*}}(\mathbb{R}^{N})$, then $v_{n} \rightarrow \Tilde{v}$ in $\mathcal{D}^{s_{2},2}(\mathbb{R}^{N})$. Hence, the proof of claim \eqref{component convergence either or case Dsp Space} is done. Further, we consider two cases to prove the strong convergence of both the components $\{u_{n}\}$,$\{v_{n}\}$ in $\mathcal{D}^{s_{1},2}(\mathbb{R}^{N})$ and $\mathcal{D}^{s_{2},2}(\mathbb{R}^{N})$, respectively.

\noindent \textbf{Case 1:} The sequence $\{v_{n}\}$ converges strongly to $\Tilde{v}$ in $\mathcal{D}^{s_{2},2}(\mathbb{R}^{N})$.

On the contrary, assume that none of its subsequences converge. Let us assume first that the set $\mathcal{J} \cup\{0,\infty\}$ has more than one point, by combining (\ref{ inequalities with c}) with (\ref{rho concentartion}), (\ref{functional with first component at origin}), (\ref{rho node inequalities}), (\ref{functional with first component at infty}) and (\ref{rho infty inequalities}), we find

\begin{align*}
    c \geq \frac{2s_{1}}{N} S^{\frac{N}{2s_{1}}}(\lambda_{1}) \geq \frac{s_{2}}{N} S^{\frac{N}{2s_{1}}}(\lambda_{1})+ \frac{s_{2}}{N}S^{\frac{N}{2s_{2}}}(\lambda_{2}),
\end{align*}
 which contradicts assumption (\ref{c is less than sum}). Thus, there can not be more than one point i.e. $x_{j},~ j \in \mathcal{J}\cup \{0,\infty\}$  contains only one concentration point for the sequence $\{u_{n}\}$. Further, we shall show that $\Tilde{v} \not\equiv 0$. Again by the contradiction we assume that $\Tilde{v} \equiv 0$, then $\Tilde{u}\geq 0$ (as $u_{n}$ is a non-negative sequence) and $\Tilde{u}$ verifies

\begin{align} \label{PDE with v zero}
    (-\Delta)^{s_{1}} \Tilde{u} - \lambda_{1}\frac{\Tilde{u}}{|x|^{2s_{1}}}= \Tilde{u}^{2^{*}_{s_{1}}-1} ~\text{in} ~\mathbb{R}^{N}.
\end{align}
Therefore, for some $\mu > 0$, we have $\Tilde{u} = z_{\mu,s_1}^{\lambda_{1}}$ and $\int_{\mathbb{R}^{N}}\Tilde{u}^{2^{*}_{s_{1}}}\,\mathrm{d}x = S^{\frac{N}{2s_{1}}}(\lambda_{1})$ by \eqref{extremal value at norms}. Using the fact that there is only one concentration point for the sequence $\{u_{n}\}$, and combining  (\ref{ inequalities with c}) with (\ref{rho concentartion}), (\ref{functional with first component at origin}), (\ref{rho node inequalities}), we deduce the following inequality
\begin{align*}
    c \geq \frac{s_{1}}{N} \bigg(  \int_{\mathbb{R}^{N}}\Tilde{u}^{2_{s_{1}}^{*}}\,\mathrm{d}x + S^{\frac{N}{2s_{1}}}(\lambda_{1})\bigg) = \frac{2s_{1}}{N}S^{\frac{N}{2s_{1}}}(\lambda_{1}) \geq \frac{s_{2}}{N} S^{\frac{N}{2s_{1}}}(\lambda_{1})+ \frac{s_{2}}{N}S^{\frac{N}{2s_{2}}}(\lambda_{2}).
\end{align*}
This contradicts the assumption (\ref{c is less than sum}). In case of $\Tilde{v} \equiv 0$  and $\Tilde{u} \equiv 0$, we have that $\{u_{n}\}$ should verify the following

\begin{align*} \label{}
    (-\Delta)^{s_{1}} u_{n} - \lambda_{1}\frac{u_{n}}{|x|^{2s_{1}}} - u_{n}^{2^{*}_{s_{1}}-1} = o(1)~\text{in the~dual~space}~\big(\mathcal{D}^{s_{1},2}(\mathbb{R}^{N})\big)'.
\end{align*}
Which implies that
\begin{equation} \label{convergence one}
     \iint_{\mathbb{R}^{2N}} \frac{(u_{n}(x)-u_{n}(y))^{2}}{|x-y|^{N+2s_{1}}}\, \mathrm{d}x \mathrm{d}y - \lambda_{1}\int_{\mathbb{R}^{N}}\frac{u_{n}^{2}}{|x|^{2s_{1}}}\,\mathrm{d}x - \int_{\mathbb{R}^{N}}|u_{n}|^{2_{s_{1}}^{*}}\, \mathrm{d}x = o_n(1),
\end{equation}

and 
\begin{align*}
    c &= J_{\nu}(u_{n},v_{n}) + o_n(1) = J_{1}(u_{n}) + J_{2}(v_{n}) - \nu \int_{\mathbb{R}^{N}}h(x)(u_{n})^{\alpha}(v_{n})^{\beta}\,\mathrm{d}x + o_n(1)\\
    &= J_{1}(u_{n}) + 0 - 0 + o_n(1),~~as~\{v_{n}\}~\text{strongly~converges~to}~0~\text{and}~h\in L^{1}(\mathbb{R}^{N}) \cap L^{\infty}(\mathbb{R}^{N})\\
    &= \frac{1}{2} \iint_{\mathbb{R}^{2N}} \frac{(u_{n}(x)-u_{n}(y))^{2}}{|x-y|^{N+2s_{1}}} \mathrm{d}x \mathrm{d}y - \frac{\lambda_{1}}{2}\int_{\mathbb{R}^{N}}\frac{u_{n}^{2}}{|x|^{2s_{1}}}\,\mathrm{d}x - \frac{1}{2_{s_1}^{*}}\int_{\mathbb{R}^{N}}(u_{n})^{2_{s_{1}}^{*}}\,\mathrm{d}x + o_n(1)\\
    &= \frac{1}{2} \int_{\mathbb{R}^{N}}(u_{n})^{2_{s_{1}}^{*}} dx - \frac{1}{2_{s_{1}}^{*}}\int_{\mathbb{R}^{N}}(u_{n})^{2_{s_{1}}^{*}}\, \mathrm{d}x + o_n(1)~ \text{by using}~\eqref{convergence one}\\
   &= \frac{s_{1}}{N} \int_{\mathbb{R}^{N}}(u_{n})^{2_{s_{1}}^{*}}\, \mathrm{d}x + o_n(1) = \frac{s_{1}}{N} \rho_{j}, ~as~\Tilde{u} = 0 ~\text{and}~\{u_{n}\}~\text{concentrates~at~one~point}. 
\end{align*}
Also for every $j \in \mathcal{J}$, the sequence $\{u_{n}\}$ is a positive (PS) sequence for the functional given as
\begin{align*}
    J_{j}(u) = \frac{1}{2} \iint_{\mathbb{R}^{2N}} \frac{(u_{n}(x)-u_{n}(y))^{2}}{|x-y|^{N+2s_{1}}} dxdy  - \frac{1}{2_{s_{1}}^{*}}\int_{\mathbb{R}^{N}}(u_{n})^{2_{s_{1}}^{*}} dx
\end{align*}
Using the characterization of (PS) sequence for the functional $J_{j}$ provided by \cite{Pisante2015} (See (2.6) in the proof of Theorem 1.1), we find that $\rho_{j} = l{S_1}^{\frac{N}{2s_{1}}}$ for some $l \in \mathbb{N}$, which is a contradiction to (\ref{Sum is less than S}) and (\ref{c is less than sum}). Thus, $\mathcal{J} = \emptyset$. If $\{u_{n}\}$ is concentrating at points zero or infinity, we argue analogously for the functional $J_{1}$ and use the result provided by \cite{Bhakta2021} (See Theorem 2.1 part (v)) to obtain the following
\begin{align*}
    c = J_{\nu}(u_{n},v_{n}) + o(1) = J_{1}(u_{n}) + o(1) \rightarrow \frac{s_{1}l}{N}S^{\frac{N}{2s_{1}}}(\lambda_{1}),
\end{align*}
for some $l \in \mathbb{N}\cup \{0\}$. Therefore, we get a contradiction to (\ref{c is less than sum}) and hence $\Tilde{v}>0$ in $\mathbb{R}^{N}$. Further, we may assume that there exists $\Tilde{u} \not\equiv 0$ such that $u_{n} \rightharpoonup \Tilde{u}$ in $\mathcal{D}^{s_{1},2}(\mathbb{R}^{N})$. On contrary  let $\Tilde{u} = 0$, then $\Tilde{v}$ is a solution to the problem given by
\begin{align} \label{PDE with u zero}
    (-\Delta)^{s_{2}} \Tilde{v} - \lambda_{2}\frac{\Tilde{v}}{|x|^{2s_{2}}}= \Tilde{v}^{2^{*}_{s_{2}}-1} ~\text{in} ~\mathbb{R}^{N}.
\end{align}
Therefore, for some $\mu > 0$, we have $\Tilde{v} = z_{\mu,s_2}^{\lambda_{2}}$ and $\int_{\mathbb{R}^{N}}\Tilde{v}^{2^{*}_{s_{2}}}\,\mathrm{d}x = S^{\frac{N}{2s_{2}}}(\lambda_{2})$ by \eqref{extremal value at norms}. As a consequence, combining  (\ref{ inequalities with c}) with (\ref{rho concentartion}), (\ref{functional with first component at origin}), (\ref{rho node inequalities}), we conclude that
\begin{align*}
    c \geq \frac{s_{2}}{N} \int_{\mathbb{R}^{N}}\Tilde{v}^{2_{s_2}^{*}}\,\mathrm{d}x + \frac{s_{1}}{N} S^{\frac{N}{2s_1}}(\lambda_{1}) = \frac{s_{2}}{N} S^{\frac{N}{2s_{2}}}(\lambda_{2})+ \frac{s_{1}}{N} S^{\frac{N}{2s_{1}}}(\lambda_{1}) \geq \frac{s_{2}}{N} S^{\frac{N}{2s_{2}}}(\lambda_{2}) + \frac{s_{2}}{N} S^{\frac{N}{2s_{2}}}(\lambda_{2}),
\end{align*}
which contradicts the assumption (\ref{c is less than sum}). We can conclude that $\Tilde{u}, \Tilde{v} \not\equiv 0$. Further, we have
\begin{align} \label{equality with c}
    \begin{split}
        c& = J_{\nu}(u_{n},v_{n}) - \frac{1}{2}\langle J'_{\nu}(u_{n},v_{n})|(u_{n},v_{n})\rangle + o_n(1) \\
        &= \frac{s_{1}}{N}  \|u_{n}\|_{{2_{s_{1}}^{*}}}^{2_{s_{1}}^{*}}+ \frac{s_{2}}{N}  \|v_{n}\|_{{2_{s_{2}}^{*}}}^{2_{s_{2}}^{*}} + \nu \frac{\alpha+\beta-2}{2}\int_{\mathbb{R}^{N}}h(x)u_{n}^{\alpha}v_{n}^{\beta}\,\mathrm{d}x + o_n(1)\hspace{2.3cm}\\
        &= \frac{s_{1}}{N}  \|\Tilde{u}\|_{{2_{s_{1}}^{*}}}^{2_{s_{1}}^{*}}+ \frac{s_{2}}{N}  \|\Tilde{v}\|_{{2_{s_{2}}^{*}}}^{2_{s_{2}}^{*}} + \frac{s_{1}}{N}\rho_{j} + \nu \frac{\alpha+\beta-2}{2}\int_{\mathbb{R}^{N}}h(x)\Tilde{u}^{\alpha}\Tilde{v}^{\beta}\,\mathrm{d}x,
    \end{split}
\end{align}
by the concentration at $j \in \mathcal{J}\cup\{0,\infty\}$. Note that $\{v_{n}\}$ converges strongly to $\Tilde{v}$ in $\mathcal{D}^{s_{2},2}(\mathbb{R}^{N})$.\\
On the other hand, using $\langle J'_{\nu}(u_{n},v_{n})|(\Tilde{u},\Tilde{v})\rangle = o_n(1)$, we obtain the following expression
\begin{align*}
    \|(\Tilde{u},\Tilde{v})\|_{\mathbb{D}}^{2} = \|\Tilde{u}\|_{{2_{s_{1}}^{*}}}^{2_{s_{1}}^{*}} + \|\Tilde{v}\|_{{2_{s_{2}}^{*}}}^{2_{s_{2}}^{*}} + \nu(\alpha+\beta)\int_{\mathbb{R}^{N}}h(x)\Tilde{u}^{\alpha}\Tilde{v}^{\beta}\,\mathrm{d}x,
\end{align*}
which clearly shows that $(\Tilde{u},\Tilde{v}) \in \mathcal{N}_{\nu}$. Indeed, by combining (\ref{equality with c}),\eqref{three three two},(\ref{rho concentartion}),(\ref{functional with first component at origin}),(\ref{rho node inequalities}) and (\ref{ inequalities with c}), we get the following
\begin{align*}
    J_{\nu}(\Tilde{u},\Tilde{v}) &= \frac{s_{1}}{N}  \|\Tilde{u}\|_{{2_{s_{1}}^{*}}}^{2_{s_{1}}^{*}}+ \frac{s_{2}}{N}  \|\Tilde{v}\|_{{2_{s_{2}}^{*}}}^{2_{s_{2}}^{*}} + \nu \frac{\alpha+\beta-2}{2}\int_{\mathbb{R}^{N}}h(x)\Tilde{u}^{\alpha}\Tilde{v}^{\beta}\,\mathrm{d}x\\
    &= c - \frac{s_{1}}{N}\rho_{j} ~<~\frac{s_{2}}{N}S^{\frac{N}{2s_{1}}}(\lambda_{1})+\frac{s_{2}}{N}S^{\frac{N}{2s_{2}}}(\lambda_{2}) -\frac{s_{1}}{N} S^{\frac{N}{2s_{1}}}(\lambda_{1}) ~\leq~ \frac{s_{2}}{N} S^{\frac{N}{2s_{2}}}(\lambda_{2}). 
\end{align*}
Thus, using the above expression we have 
\begin{align*}
    \Tilde{c}_{\nu} = \inf_{(u,v) \in \mathcal{N}_{\nu}}J_{\nu}(u,v) < \frac{s_{2}}{N} S^{\frac{N}{2s_{2}}}(\lambda_{2}).
\end{align*}
But, according to Theorem \ref{third theorem}, we have $ \Tilde{c}_{\nu} = \frac{s_{2}}{N} S^{\frac{N}{2s_{2}}}(\lambda_{2})$ provided that $\nu$ is very small. Thus, we get a contradiction to the former inequality. Hence, the proof of $u_{n} \rightarrow \Tilde{u}$ strongly in $\mathcal{D}^{s_{1},2}(\mathbb{R}^{N})$ is completed.

\noindent\textbf{Case 2:} The sequence $\{u_{n}\}$ strongly converges to $\Tilde{u}$ in $\mathcal{D}^{s_{1},2}(\mathbb{R}^{N})$.\\
Our claim is that the sequence $\{v_{n}\}$ converges strongly to $\Tilde{v}$ in $\mathcal{D}^{s_{2},2}(\mathbb{R}^{N})$. Using the contradiction method, we assume that all of its subsequences do not converge. First we prove that $\Tilde{u} \not\equiv 0$. On contrary suppose that $\Tilde{u} \equiv 0$, then $\{v_{n}\}$ is a (PS) sequence for the functional $J_{2}$ given by \eqref{value of functional componentwise} at level c. Clearly for some $\mu>0$, $\Tilde{v} =  z_{\mu,s_2}^{\lambda_{2}}$ as $\{v_{n}\} \rightharpoonup \Tilde{v}~\text{in}~ \mathcal{D}^{s_{2},2}(\mathbb{R}^{N})$ and $\Tilde{v}$ is a solution to the problem (\ref{PDE with u zero}).Also, by applying the compactness theorem given by \cite{Bhakta2021} and using \eqref{value of functional componentwise} and \eqref{extremal value at norms}, we find that
\begin{align*}
    c = \lim_{n \rightarrow +\infty}J_{2}(v_{n}) = J_{2}( z_{\mu,s_2}^{\lambda_{2}}) + \frac{s_{2}}{N}m {S_2}^{\frac{N}{2s_{2}}} + \frac{s_{2}}{N}l S^{\frac{N}{2s_{2}}}(\lambda_{2}) = \frac{s_{2}}{N}m {S_2}^{\frac{N}{2s_{2}}} + \frac{s_{2}}{N}(l+1) S^{\frac{N}{2s_{2}}}(\lambda_{2}),
\end{align*}
for some $m \in \mathbb{N}$ and $l \in \mathbb{N}\cup\{0\}$. This contradicts the assumptions (\ref{c is less than sum}) and (\ref{c is not equal modified problem}). Therefore, the conclusion $\Tilde{u} \not\equiv 0$ follows immediately. On the other hand, the assumption $\Tilde{v} \equiv 0$ implies that $\Tilde{u}$ is a solution to the problem (\ref{PDE with v zero}) and that for some $\mu>0$, we have $\Tilde{u} = z_{\mu,s_1}^{\lambda_{1}}$ . Thus, we obtain
\begin{align*}
    c \geq \frac{s_{1}}{N} \int_{\mathbb{R}^{N}}\Tilde{u}^{2_{s_{1}}^{*}}\,\mathrm{d}x + \frac{s_{2}}{N} S^{\frac{N}{2s_{2}}}(\lambda_{2}) \geq \frac{s_{2}}{N} S^{\frac{N}{2s_{1}}}(\lambda_{1})+ \frac{s_{2}}{N} S^{\frac{N}{2s_{2}}}(\lambda_{2}),
\end{align*}
contradicting (\ref{c is less than sum}). Hence, we conclude that $\Tilde{u},\Tilde{v}\not\equiv 0$. Now as $(\Tilde{u},\Tilde{v})$ is a solution of (\ref{main problem}), we have
\begin{align} \label{three three two}
   J_{\nu}(\Tilde{u},\Tilde{v}) =  \frac{s_{1}}{N} \int_{\mathbb{R}^{N}} \Tilde{u}^{2_{s_{1}}^{*}}\,\mathrm{d}x + \frac{s_{2}}{N}\int_{\mathbb{R}^{N}}\Tilde{v}^{2_{s_{2}}^{*}}\,\mathrm{d}x + \nu \frac{\alpha+\beta-2}{2}\int_{\mathbb{R}^{N}}h(x)\Tilde{u}^{\alpha}\Tilde{v}^{\beta}\,\mathrm{d}x.
\end{align}
Since $\{v_{n}\}$ does not strongly converge to $\Tilde{v}$ in $\mathcal{D}^{s_{2},2}(\mathbb{R}^{N})$, there exists at least $k \in \mathcal{K}\cup\{0,\infty\}$ such that $\Bar{\rho}_{k}>0$ and using again (\ref{equality with c}), we get 
\begin{align*}
    c =   \bigg(\frac{s_{1}}{N} \int_{\mathbb{R}^{N}} \Tilde{u}^{2_{s_{1}}^{*}}\,\mathrm{d}x + \frac{s_{2}}{N}\int_{\mathbb{R}^{N}}\Tilde{v}^{2_{s_{2}}^{*}}\,\mathrm{d}x +\frac{s_{2}}{N} \sum\limits_{k \in \mathcal{K}}\Bar{\rho}_{k} + \Bar{\rho}_{0} + \Bar{\rho}_{\infty}\bigg) + \nu \frac{\alpha+\beta-2}{2}\int_{\mathbb{R}^{N}}h(x)\Tilde{u}^{\alpha}\Tilde{v}^{\beta}\,\mathrm{d}x.
\end{align*}
By using (\ref{rho bar concentartion})-(\ref{rho node inequalities}), (\ref{three three two}) and (\ref{c is less than sum}), one has
\begin{align} \label{three three three}
\begin{split}
    J_{\nu}(\Tilde{u},\Tilde{v}) &= c - \frac{s_{2}}{N} (\sum\limits_{k \in \mathcal{K}}\Bar{\rho}_{k} + \Bar{\rho}_{0}+\Bar{\rho}_{\infty})\\
    &< \frac{s_{2}}{N} S^{\frac{N}{2s_{1}}}(\lambda_{1})+ \frac{s_{2}}{N}S^{\frac{N}{2s_{2}}}(\lambda_{2})- \frac{s_{2}}{N}S^{\frac{N}{2s}}(\lambda_{2}) = \frac{s_{2}}{N}S^{\frac{N}{2s_{1}}}(\lambda_{1}).
\end{split}
\end{align}
Further, we use the definition of $S^{\frac{N}{2s_{1}}}(\lambda_{1})$ in the first equation of (\ref{main problem}) which implies that
\begin{align}\label{inequality with sigma}
    \sigma_{1} + \nu \alpha \int_{\mathbb{R}^{N}}h(x)\Tilde{u}^{\alpha}\Tilde{v}^{\beta}\,\mathrm{d}x = \iint_{\mathbb{R}^{2N}} \frac{(\Tilde{u}(x)-\Tilde{u}(y))^{2}}{|x-y|^{N+2s_{1}}}\, \mathrm{d}x \mathrm{d}y - \lambda_{1} \int_{\mathbb{R}^{N}}\frac{\Tilde{u}^{2}}{|x|^{2s_{1}}}\,\mathrm{d}x \geq S(\lambda_{1}) \sigma_{1}^{\frac{2}{2_{s_{1}}^{*}}},
\end{align}
such that $\sigma_{1} = \int_{\mathbb{R}^{N}}\Tilde{u}^{2_{s_{1}}^{*}}\,\mathrm{d}x$. Then applying H\"older's inequality leads to the following
\begin{align} \label{Holders inequality}
    \int_{\mathbb{R}^{N}}h(x)\Tilde{u}^{\alpha}\Tilde{v}^{\beta}\,\mathrm{d}x \leq C(h) \bigg(\int_{\mathbb{R}^{N}}\Tilde{u}^{2_{s_{1}}^{*}}\,\mathrm{d}x\bigg)^{\frac{\alpha}{2_{s_{1}}^{*}}}\bigg(\int_{\mathbb{R}^{N}}\Tilde{v}^{2_{s_{2}}^{*}}\,\mathrm{d}x\bigg)^{\frac{\beta}{2_{s_{2}}^{*}}}.
\end{align}
By combining (\ref{three three two}) and (\ref{Holders inequality}), we can transform (\ref{inequality with sigma}) into
\begin{align} \label{sigma with C}
    \sigma_{1} + C_1 \nu \sigma_{1}^{\frac{\alpha}{2} \big(\frac{N-2s_{1}}{N} \big)} \geq S(\lambda_{1}) \sigma_{1}^{\frac{N-2s_{1}}{N}},
\end{align}
where the constant $C_1 >0$ depends only on $N,s,\alpha,\beta,h$ and independent of $\Tilde{u}, \Tilde{v}$ and $\nu$. We know that $\Tilde{v}\not\equiv 0$, we can choose $\Tilde{\epsilon}>0$ such that $\int_{\mathbb{R}^{N}} \Tilde{v}^{2_{s_{2}}^{*}} \,\mathrm{d}x\geq \Tilde{\epsilon}$. Now take $\epsilon>0$ such a way that $\Tilde{\epsilon} \geq \epsilon S^{\frac{N}{2s_{1}}}(\lambda_{1})$. Since $\alpha \geq 2$ and (\ref{sigma with C}) holds, therefore we can apply Lemma \ref{Abdelloui fractional version} to get a fixed $\nu_0>0$ such that
\begin{align*}
    \sigma_{1} \geq (1-\epsilon)S^{\frac{N}{2s_{1}}}(\lambda_{1}) ~~\text{for any}~0<\nu\leq \nu_0.
\end{align*}
Combining (\ref{three three two}) and the last estimate, we get the following inequality
\begin{align*}
    J_{\nu}(\Tilde{u},\Tilde{v}) \geq \frac{s_{1}}{N} (1-\epsilon)S^{\frac{N}{2s_{1}}}(\lambda_{1})+ \frac{s_{2}}{N}\Tilde{\epsilon} \geq \frac{s_{2}}{N} (1-\epsilon)S^{\frac{N}{2s_{1}}}(\lambda_{1})+ \frac{s_{2}}{N} \epsilon S^{\frac{N}{2s_{1}}}(\lambda_{1}) = \frac{s_{2}}{N} S^{\frac{N}{2s_{1}}}(\lambda_{1}),
\end{align*}
 contradicting the inequality (\ref{three three three}). Hence, $\{v_{n}\}$ converges strongly to $\Tilde{v}$ in $\mathcal{D}^{s_2,2}(\mathbb{R}^{N})$. Proceeding in the same way, the result can be proved for the case $s_{1} \leq s_{2}$. Finally, combining both the cases leads to the conclusion that the (PS) sequences strongly converge in $\mathbb{D}$ to a non-trivial limit. Hence the proof is complete.
\end{proof}
In the following lemma, we will derive the Palais-Smale compactness condition for the functional $J_{\nu}^{+}$ associated with the modified problem \eqref{modified problem} when the exponent $\beta \geq 2$ and $S^{\frac{N}{2s_{2}}}(\lambda_{2}) \geq S^{\frac{N}{2s_{1}}}(\lambda_{1})$.

\begin{lemma} \label{lemma three six}
 Assume $\alpha+\beta< \min\{2_{s_{1}}^{*},2_{s_{2}}^{*} \}$ and (\ref{condition on h}), $\beta \geq 2,~ S^{\frac{N}{2s_{2}}}(\lambda_{2}) \geq S^{\frac{N}{2s_{1}}}(\lambda_{1})$ and 
\begin{align} 
    S^{\frac{N}{2s_{1}}}(\lambda_{1})+S^{\frac{N}{2s_{2}}}(\lambda_{2}) < \min\{{S_1}^{\frac{N}{2s_{1}}},{S_2}^{\frac{N}{2s_{2}}}\}.
\end{align}
Then, there exists $\nu_0>0$ such that, if $0< \nu \leq \nu_0$ and $\{(u_{n},v_{n})\} \subset \mathbb{D}$ is a PS sequence for $J_{\nu}^{+}$ at level $c \in \mathbb{R}$ such that 
\begin{align}\label{c is less than sum analogous}
    \frac{s_{2}}{N}S^{\frac{N}{2s_{2}}}(\lambda_{2}) < c < \frac{\min\{s_{1},s_{2}\}}{N}\bigg(S^{\frac{N}{2s_{1}}}(\lambda_{1})+S^{\frac{N}{2s_{2}}}(\lambda_{2})\bigg),
\end{align}
and 
\begin{align} \label{c is not equal modified problem analogous}
    c \neq \frac{s_{1}l}{N}S^{\frac{N}{2s_{1}}}(\lambda_{1})~~~\text{for every}~l\in \mathbb{N}\backslash \{0\},
\end{align}
then $ \exists ~(\Tilde{u},\Tilde{v}) \in \mathbb{D}$ such that $(u_{n},v_{n}) \rightarrow (\Tilde{u},\Tilde{v}) \in \mathbb{D}$ up to a subsequence.
\end{lemma}
Next, we give the behavior of the semi-trivial solutions based on the coupling parameter $\nu$ (small or large) and with different values of $\alpha,\beta$.

\begin{proposition} \label{Semitrivial solution as a minimizer}
Under hypotheses (\ref{ alpha beta condition}) and (\ref{condition on h}), the following holds:
\begin{itemize}

     \item [(i)] The pair $(z_{\mu,s_{1}}^{\lambda_{1}},0)$ is a local minimum of $J_{\nu}$ on $\mathcal{N}_{\nu}$ for $\beta>2~or~ \beta=2$ and $\nu$ sufficiently small.
    \item [(ii)] The pair $(0,z_{\mu,s_{2}}^{\lambda_{2}})$ is a local minimum of $J_{\nu}$ on $\mathcal{N}_{\nu}$ for $\alpha>2~or~ \alpha=2$ and $\nu$ sufficiently small.
     \item [(iii)] The pair $(z_{\mu,s_{1}}^{\lambda_{1}},0)$ is a saddle point for $J_{\nu}$ on $\mathcal{N}_{\nu}$ for $\beta<2~or~ \beta=2$ and $\nu$ sufficiently large.
     \item [(iv)] The pair $(0,z_{\mu,s_{2}}^{\lambda_{2}})$ is a saddle point for $J_{\nu}$ on $\mathcal{N}_{\nu}$ for $\alpha<2~or~ \alpha=2$ and $\nu$ sufficiently large.
\end{itemize}
\end{proposition}
\begin{proof}
$(i)$ Let $\mu>0, \beta>2$ and $(z_{\mu,s_{1}}^{\lambda_{1}}+\phi ,~ \psi) \in \mathcal{N}_{\nu}$ i.e.,
\begin{align} \label{local minimum norm}
    \begin{split}
         \|(z_{\mu,s_{1}}^{\lambda_{1}}+\phi ,~ \psi)\|_{\mathbb{D}}^{2}
    &= \|z_{\mu,s_{1}}^{\lambda_{1}}+\phi\|_{{2_{s_{1}}^{*}}}^{2_{s_{1}}^{*}} + \|\psi\|_{{2_{s_{2}}^{*}}}^{2_{s_{2}}^{*}} + \nu (\alpha + \beta) \int_{\mathbb{R}^{N}} h(x)|z_{\mu,s_{1}}^{\lambda_{1}}+\phi|^{\alpha}|\psi|^{\beta}\,\mathrm{d}x.\\
    \|z_{\mu,s_{1}}^{\lambda_{1}}+\phi\|_{\lambda_{1},s_{1}}^{2} + \|\psi\|_{\lambda_{2},s_{2}}^{2}&=\|z_{\mu,s_{1}}^{\lambda_{1}}+\Psi\|_{{2_{s_{1}}^{*}}}^{2_{s_{1}}^{*}} + \|\psi\|_{{2_{s_{2}}^{*}}}^{2_{s_{2}}^{*}} + \nu (\alpha + \beta) \int_{\mathbb{R}^{N}} h(x)|z_{\mu,s_{1}}^{\lambda_{1}}+\phi|^{\alpha}|\psi|^{\beta}\,\mathrm{d}x.
    \end{split}
\end{align}
Let $t = t_{(\phi,\psi)}>0$ be such that $t(z_{\mu,s_{1}}^{\lambda_{1}}+\phi) \in \mathcal{N}_{\lambda_{1}}$, where $\mathcal{N}_{\lambda_{1}}$  denotes the Nehari manifold associated to $J_{1}$, namely $\mathcal{N}_{\lambda_{1}}$ is defined by 
$$\mathcal{N}_{\lambda_{1}} = \big\{u \in \mathcal{D}^{s_{1},2}(\mathbb{R}^{N})\backslash \{0\} : \|u\|_{\lambda_{1},s_{1}}^{2} = \|u\|_{2_{s_{1}}^{*}}^{{2_{s_{1}}^{*}}}\big\}.$$
Since $t(z_{\mu,s_{1}}^{\lambda_{1}}+\phi) \in \mathcal{N}_{\lambda_{1}}$, we have
\begin{align*}
    \|t(z_{\mu,s_{1}}^{\lambda_{1}}+\phi)\|_{\lambda_{1},s_{1}}^{2} &= \|t(z_{\mu,s_{1}}^{\lambda_{1}}+\phi)\|_{2_{s_{1}}^{*}}^{{2_{s_{1}}^{*}}}\\
    t^2\|z_{\mu,s_{1}}^{\lambda_{1}}+\phi\|_{\lambda_{1},s_{1}}^{2} &=t^{2_{s_{1}}^{*}}\|z_{\mu,s_{1}}^{\lambda_{1}}+\phi\|_{2_{s_{1}}^{*}}^{{2_{s_{1}}^{*}}}\\
    t^{2_{s_{1}}^{*} -2} &= \frac{\|z_{\mu,s_{1}}^{\lambda_{1}}+\phi\|_{\lambda_{1},s_{1}}^{2}}{\|z_{\mu,s_{1}}^{\lambda_{1}}+\phi\|_{2_{s_{1}}^{*}}^{{2_{s_{1}}^{*}}}}
\end{align*}
Therefore, we get
\begin{equation} \label{value of t}
    t = \bigg( \frac{\|z_{\mu,s_{1}}^{\lambda_{1}}+\phi\|_{\lambda_{1},s_{1}}^{2}}{\|z_{\mu,s_{1}}^{\lambda_{1}}+\phi\|_{2_{s_{1}}^{*}}^{{2_{s_{1}}^{*}}}}\bigg)^{\frac{1}{2_{s_{1}}^{*} -2}}.
\end{equation}
Combining \eqref{local minimum norm} and \eqref{value of t}, we get
\begin{align*}
    t &= t_{(\phi, \psi)} = \bigg[\frac{   \|z_{\mu,s_{1}}^{\lambda_{1}}+\phi\|_{{2_{s_{1}}^{*}}}^{2_{s_{1}}^{*}}-\|\psi\|_{\lambda_{2},s_{2}}^{2} + \|\psi\|_{{2_{s_{2}}^{*}}}^{2_{s_{2}}^{*}} + \nu (\alpha + \beta) \int_{\mathbb{R}^{N}} h(x)|z_{\mu,s_{1}}^{\lambda_{1}}+\phi|^{\alpha}|\psi|^{\beta}\,\mathrm{d}x}{\|z_{\mu,s_{1}}^{\lambda_{1}}+\phi\|_{2_{s_{1}}^{*}}^{{2_{s_{1}}^{*}}}} \bigg]^{\frac{1}{2_{s_{1}}^{*} -2}}\\
    &=\bigg[1- \frac{\|\psi\|_{\lambda_{2},s_{2}}^{2} - \|\psi\|_{{2_{s_{2}}^{*}}}^{2_{s_{2}}^{*}} - \nu (\alpha + \beta) \int_{\mathbb{R}^{N}} h(x)|z_{\mu,s_{1}}^{\lambda_{1}}+\phi|^{\alpha}|\psi|^{\beta}\,\mathrm{d}x}{\|z_{\mu,s_{1}}^{\lambda_{1}}+\phi\|_{2_{s_{1}}^{*}}^{{2_{s_{1}}^{*}}}} \bigg]^{\frac{1}{2_{s_{1}}^{*} -2}}.
\end{align*}
Since $h \in L^{1}(\mathbb{R}^{N}) \cap L^{\infty}(\mathbb{R}^{N})$, we have the following H\"older's Inequality
\begin{align*}
     \int_{\mathbb{R}^{N}}h(x)|z_{\mu,s_{1}}^{\lambda_{1}}+\phi|^{\alpha} |\psi|^{\beta}\,\mathrm{d}x &\leq C(h) \bigg(\int_{\mathbb{R}^{N}}|z_{\mu,s_{1}}^{\lambda_{1}}+\phi|^{2_{s_{1}}^{*}}\,\mathrm{d}x\bigg)^{\frac{\alpha}{2_{s_{1}}^{*}}}\bigg(\int_{\mathbb{R}^{N}}|\psi|^{2_{s_{2}}^{*}}\,\mathrm{d}x\bigg)^{\frac{\beta}{2_{s_{2}}^{*}}}\\
     &= C(h)\|z_{\mu,s_{1}}^{\lambda_{1}}+\phi\|_{2_{s_{1}}^{*}}^{\alpha} \|\psi\|_{2_{s_{2}}^{*}}^{\beta}\\
     &\leq C(h) \frac{1}{(S(\lambda_{2}))^{\beta /2}}\|z_{\mu,s_{1}}^{\lambda_{1}}+\phi\|_{2_{s_{1}}^{*}}^{\alpha} \|\psi\|_{\lambda_{2},s_{2}}^{\beta},
\end{align*}
where the last inequality follows from the Sobolev embedding given by
$$S(\lambda_{2}) \|\psi\|_{2_{s_{2}}^{*}}^{2} \leq\|\psi\|_{\lambda_{2},s_{2}}^{2}.$$
Now 
\begin{align*}
     t^2 &= \bigg[1- \frac{\|\psi\|_{\lambda_{2},s_{2}}^{2} - \|\psi\|_{{2_{s_{2}}^{*}}}^{2_{s_{2}}^{*}} - \nu (\alpha + \beta) \int_{\mathbb{R}^{N}} h(x)|z_{\mu,s_{1}}^{\lambda_{1}}+\phi|^{\alpha}|\psi|^{\beta}\,\mathrm{d}x}{\|z_{\mu,s_{1}}^{\lambda_{1}}+\phi\|_{2_{s_{1}}^{*}}^{{2_{s_{1}}^{*}}}} \bigg]^{\frac{2}{2_{s_{1}}^{*} -2}}\\
     &= \bigg[1- \frac{\|\psi\|_{\lambda_{2},s_{2}}^{2} }{\|z_{\mu,s_{1}}^{\lambda_{1}}+\phi\|_{2_{s_{1}}^{*}}^{{2_{s_{1}}^{*}}}} + A(\phi,\psi) \bigg]^{\frac{2}{2_{s_{1}}^{*} -2}}.
\end{align*}
Where 
\begin{align*}
    A(\phi,\psi) &=  \frac{ \|\psi\|_{{2_{s_{2}}^{*}}}^{2_{s_{2}}^{*}} + \nu (\alpha + \beta) \int_{\mathbb{R}^{N}} h(x)|z_{\mu,s_{1}}^{\lambda_{1}}+\phi|^{\alpha}|\psi|^{\beta}\,\mathrm{d}x}{\|z_{\mu,s_{1}}^{\lambda_{1}}+\phi\|_{2_{s_{1}}^{*}}^{{2_{s_{1}}^{*}}}} \\
    &\leq  \frac{ \|\psi\|_{{2_{s_{2}}^{*}}}^{2_{s_{2}}^{*}} + \nu (\alpha + \beta) C(h)\|z_{\mu,s_{1}}^{\lambda_{1}}+\phi\|_{2_{s_{1}}^{*}}^{\alpha} \|\psi\|_{2_{s_{2}}^{*}}^{\beta}}{\|z_{\mu,s_{1}}^{\lambda_{1}}+\phi\|_{2_{s_{1}}^{*}}^{{2_{s_{1}}^{*}}}}\\
    &\leq \frac{ C_2\|\psi\|_{\lambda_{2},s_{2}}^{2_{s_{2}}^{*}} + \nu (\alpha + \beta) C_1(h)\|z_{\mu,s_{1}}^{\lambda_{1}}+\phi\|_{2_{s_{1}}^{*}}^{\alpha} \|\psi\|_{\lambda_{2},s_{2}}^{\beta}}{\|z_{\mu,s_{1}}^{\lambda_{1}}+\phi\|_{2_{s_{1}}^{*}}^{{2_{s_{1}}^{*}}}},~\text{where}~C_2 = \frac{1}{S(\lambda_{2})^{2_{s_{2}}^{*}/2}}.
\end{align*}
Taking $\|\psi\|_{\lambda_{2},s_{2}}^{2}$ as common from the above expression, we have
$$  A(\phi,\psi) \leq \frac{\|\psi\|_{\lambda_{2},s_{2}}^{2}}{\|z_{\mu,s_{1}}^{\lambda_{1}}+\phi\|_{2_{s_{1}}^{*}}^{{2_{s_{1}}^{*}}}} \bigg( C_2\|\psi\|_{\lambda_{2},s_{2}}^{2_{s_{2}}^{*} -2} + \nu (\alpha + \beta) C_1(h)\|z_{\mu,s_{1}}^{\lambda_{1}}+\phi\|_{2_{s_{1}}^{*}}^{\alpha} \|\psi\|_{\lambda_{2},s_{2}}^{\beta -2} \bigg).$$
Since $\beta >2$, we can conclude that $\|\psi\|_{\lambda_{2},s_{2}}^{\beta -2} \rightarrow 0$ as $\|(\phi,\psi)\|_{\mathbb{D}} \rightarrow 0.$ That means 
$$A(\phi,\psi) \rightarrow \frac{\|\psi\|_{\lambda_{2},s_{2}}^{2}}{\|z_{\mu,s_{1}}^{\lambda_{1}}+\phi\|_{2_{s_{1}}^{*}}^{{2_{s_{1}}^{*}}}}o(1) ~\text{as}~\|(\phi,\psi)\|_{\mathbb{D}} \rightarrow 0.$$
Therefore,
\begin{equation} \label{ square of t}
    t^2 = 1 - \frac{2}{2_{s_{1}}^{*}-2} \frac{\|\psi\|_{\lambda_{2},s_{2}}^{2}}{\|z_{\mu,s_{1}}^{\lambda_{1}}+\phi\|_{2_{s_{1}}^{*}}^{{2_{s_{1}}^{*}}}}(1+o(1)) ~\text{as}~\|(\phi,\psi)\|_{\mathbb{D}} \rightarrow 0.
\end{equation}
Similarly,
\begin{equation} \label{ two star of t}
    t^{2_{s_{1}}^{*}} = 1 - \frac{2_{s_{1}}^{*}}{2_{s_{1}}^{*}-2} \frac{\|\psi\|_{\lambda_{2},s_{2}}^{2}}{\|z_{\mu,s_{1}}^{\lambda_{1}}+\phi\|_{2_{s_{1}}^{*}}^{{2_{s_{1}}^{*}}}}(1+o(1)) ~\text{as}~\|(\phi,\psi)\|_{\mathbb{D}} \rightarrow 0.
\end{equation}
As $z_{\mu,s_{1}}^{\lambda_{1}}$ achieves the minimum of $J_{1} = J_{\nu}(\cdot~,~0)$ on $\mathcal{N}_{\lambda_{1}}$, the following holds i.e.,
\begin{equation} \label{inequality one}
    J_{\nu}(t(z_{\mu,s_{1}}^{\lambda_{1}}+\phi),0) - J_{\nu}(z_{\mu,s_{1}}^{\lambda_{1}},0) \geq 0.
\end{equation}
On the other hand, from \eqref{local minimum norm}-\eqref{ two star of t} we deduce that
\begin{align*}
     & J_{\nu}(z_{\mu,s_{1}}^{\lambda_{1}}+\phi,\psi) -J_{\nu}(t(z_{\mu,s_{1}}^{\lambda_{1}}+\phi),0) \\
     &= \frac{1}{2}\|(z_{\mu,s_{1}}^{\lambda_{1}}+\phi,\psi)\|_{\mathbb{D}}^{2} - \frac{1}{2_{s_{1}}^{*}}\|z_{\mu,s_{1}}^{\lambda_{1}}+\phi\|_{2_{s_{1}}^{*}}^{2_{s_{1}}^{*}} - \frac{1}{2_{s_{2}}^{*}}\|\psi\|_{2_{s_{2}}^{*}}^{2_{s_{2}}^{*}} -\nu \int_{\mathbb{R}^{N}}h(x)|z_{\mu,s_{1}}^{\lambda_{1}}+\phi|^{\alpha}|\psi|^{\beta}\,\mathrm{d}x\\
     &~~~~~ - \frac{1}{2}\|(t(z_{\mu,s_{1}}^{\lambda_{1}}+\phi),0)\|_{\mathbb{D}}^{2} + \frac{1}{2_{s_{1}}^{*}}\|t(z_{\mu,s_{1}}^{\lambda_{1}}+\phi)\|_{2_{s_{1}}^{*}}^{2_{s_{1}}^{*}}\\
     &= \frac{1}{2}\|z_{\mu,s_{1}}^{\lambda_{1}}+\phi\|_{\lambda_{1},s_{1}}^{2} + \frac{1}{2}\|\psi\|_{\lambda_{2},s_{2}}^{2} - \frac{1}{2_{s_{1}}^{*}}\|z_{\mu,s_{1}}^{\lambda_{1}}+\phi\|_{2_{s_{1}}^{*}}^{2_{s_{1}}^{*}} - \frac{1}{2_{s_{2}}^{*}}\|\psi\|_{2_{s_{2}}^{*}}^{2_{s_{2}}^{*}} -\nu \int_{\mathbb{R}^{N}}h(x)|z_{\mu,s_{1}}^{\lambda_{1}}+\phi|^{\alpha}|\psi|^{\beta}\,\mathrm{d}x\\
     &~~~~~ -\frac{t^2}{2}\|z_{\mu,s_{1}}^{\lambda_{1}}+\phi\|_{\lambda_{1},s_{1}}^{2} + \frac{t^{2_{s_{1}}^{*}}}{2_{s_{1}}^{*}}\|z_{\mu,s_{1}}^{\lambda_{1}}+\phi\|_{2_{s_{1}}^{*}}^{2_{s_{1}}^{*}}\\
     &= \frac{1}{2}(1-t^2)\|z_{\mu,s_{1}}^{\lambda_{1}}+\phi\|_{\lambda_{1},s_{1}}^{2} - \frac{1}{2_{s_{1}}^{*}}(1-t^{2_{s_{1}}^{*}})\|z_{\mu,s_{1}}^{\lambda_{1}}+\phi\|_{2_{s_{1}}^{*}}^{2_{s_{1}}^{*}}+ \frac{1}{2}\|\psi\|_{\lambda_{2},s_{2}}^{2}-\frac{1}{2_{s_{2}}^{*}}\|\psi\|_{2_{s_{2}}^{*}}^{2_{s_{2}}^{*}} -\\ &\quad-\nu \int_{\mathbb{R}^{N}}h(x)|z_{\mu,s_{1}}^{\lambda_{1}}+\phi|^{\alpha}|\psi|^{\beta}\,\mathrm{d}x.
\end{align*}
Using \eqref{local minimum norm} in the second term of the above equation, we get
\begin{align*}
    &= \frac{1}{2}(1-t^2)\|z_{\mu,s_{1}}^{\lambda_{1}}+\phi\|_{\lambda_{1},s_{1}}^{2} - \frac{1}{2_{s_{1}}^{*}}(1-t^{2_{s_{1}}^{*}})\Bigg[ \|z_{\mu,s_{1}}^{\lambda_{1}}+\phi\|_{\lambda_{1},s_{1}}^{2} + \|\psi\|_{\lambda_{2},s_{2}}^{2}- \|\psi\|_{{2_{s_{2}}^{*}}}^{2_{s_{2}}^{*}} \\
    &~~~~~ - \nu (\alpha + \beta) \int_{\mathbb{R}^{N}} h(x)|z_{\mu,s_{1}}^{\lambda_{1}}+\phi|^{\alpha}|\psi|^{\beta}\,\mathrm{d}x \Bigg] + \frac{1}{2}\|\psi\|_{\lambda_{2},s_{2}}^{2}-\frac{1}{2_{s_{2}}^{*}}\|\psi\|_{2_{s_{2}}^{*}}^{2_{s_{2}}^{*}} -\nu \int_{\mathbb{R}^{N}}h(x)|z_{\mu,s_{1}}^{\lambda_{1}}+\phi|^{\alpha}|\psi|^{\beta}\,\mathrm{d}x\\
    &= \bigg[\frac{1}{2}(1-t^2) - \frac{1}{2_{s_{1}}^{*}}(1-t^{2_{s_{1}}^{*}}) \bigg]\|z_{\mu,s_{1}}^{\lambda_{1}}+\phi\|_{\lambda_{1},s_{1}}^{2} - \frac{1}{2_{s_{1}}^{*}}(1-t^{2_{s_{1}}^{*}}) \Bigg[ \|\psi\|_{\lambda_{2},s_{2}}^{2}- \|\psi\|_{{2_{s_{2}}^{*}}}^{2_{s_{2}}^{*}}\\
    &~~~~~ - \nu (\alpha + \beta) \int_{\mathbb{R}^{N}} h(x)|z_{\mu,s_{1}}^{\lambda_{1}}+\phi|^{\alpha}|\psi|^{\beta}\,\mathrm{d}x \Bigg]+ \frac{1}{2}\|\psi\|_{\lambda_{2},s_{2}}^{2}-\frac{1}{2_{s_{2}}^{*}}\|\psi\|_{2_{s_{2}}^{*}}^{2_{s_{2}}^{*}} -\nu \int_{\mathbb{R}^{N}}h(x)|z_{\mu,s_{1}}^{\lambda_{1}}+\phi|^{\alpha}|\psi|^{\beta}\,\mathrm{d}x\\
    & = \frac{1}{2} \|\psi\|_{\lambda_{2},s_{2}}^{2} (1+o(1)),
\end{align*}
as $\|(\phi,\psi)\|_{\mathbb{D}} \rightarrow 0$. Hence
\begin{equation}\label{inequality two}
     J_{\nu}(z_{\mu,s_{1}}^{\lambda_{1}}+\phi,\psi) -J_{\nu}(t(z_{\mu,s_{1}}^{\lambda_{1}}+\phi),0) \geq 0
\end{equation}
provided $(z_{\mu,s_{1}}^{\lambda_{1}}+\phi,\psi)$ is very near to $(z_{\mu,s_{1}}^{\lambda_{1}},0)$ in $\mathbb{D}$. If $\psi \not\equiv 0$, the inequality given by \eqref{inequality two} holds strictly. We conclude from \eqref{inequality one} and \eqref{inequality two} that
\begin{equation*}
     J_{\nu}(z_{\mu,s_{1}}^{\lambda_{1}}+\phi,\psi) -  J_{\nu}(z_{\mu,s_{1}}^{\lambda_{1}},0) \geq 0 \notag
\end{equation*}
for any $(z_{\mu,s_{1}}^{\lambda_{1}}+\phi,\psi) \in \mathcal{N}_{\nu}$ sufficiently closed to $(z_{\mu,s_{1}}^{\lambda_{1}},0)$, i.e., $(z_{\mu,s_{1}}^{\lambda_{1}},0)$ is a local minimum point of $J_{\nu}$ in $\mathcal{N}_{\nu}$.

In the case $\beta =2$, we obtain that 
\begin{align*}
   & J_{\nu}(z_{\mu,s_{1}}^{\lambda_{1}}+\phi,\psi) -J_{\nu}(t(z_{\mu,s_{1}}^{\lambda_{1}}+\phi),0) \\
   &~~~~ = \bigg(\frac{1}{2}\|\psi\|_{\lambda_{2},s_{2}}^{2}-\nu \int_{\mathbb{R}^{N}}h(x)|z_{\mu,s_{1}}^{\lambda_{1}}+\phi|^{\alpha}|\psi|^{\beta}\,\mathrm{d}x \bigg)(1+o(1)),
\end{align*}
which is still non-negative provided $\|(\phi,\psi)\|_{\mathbb{D}}$ and $\nu$ are sufficiently small.

\noindent $(ii)$ The proof works in the same way as $(i)$, hence we omit it. 

\noindent $(iii)$ To prove this part we suppose $\beta<2$ and we fix $\nu>0,~ \mu>0$ and $v \in \mathcal{D}^{{s_2},2}(\mathbb{R}^{N})\backslash \{0\}.$ For every $t \in \mathbb{R}$, we denote by $s(t)>0$ as the unique number such that $(s(t)z_{\mu,s_{1}}^{\lambda_{1}}, s(t)tv) \in \mathcal{N}_{\nu}$ i.e.,
\small
\begin{align*}
    & \|s(t)z_{\mu,s_{1}}^{\lambda_{1}}\|_{\lambda_{1},s_{1}}^{2} + \|s(t)tv\|_{\lambda_{2},s_{2}}^{2}=\|s(t)z_{\mu,s_{1}}^{\lambda_{1}}\|_{{2_{s_{1}}^{*}}}^{2_{s_{1}}^{*}} + \|s(t)tv\|_{{2_{s_{2}}^{*}}}^{2_{s_{2}}^{*}} + \nu (\alpha + \beta) \int_{\mathbb{R}^{N}} h(x)|s(t)z_{\mu,s_{1}}^{\lambda_{1}}|^{\alpha}|s(t)tv|^{\beta}\,\mathrm{d}x \end{align*}
    and then we get 
     \begin{align*}[s(t)]^2(\|z_{\mu,s_{1}}^{\lambda_{1}}\|_{\lambda_{1},s_{1}}^{2} +  \|tv\|_{\lambda_{2},s_{2}}^{2})&= [s(t)]^{{2_{s_{1}}^{*}}}\|z_{\mu,s_{1}}^{\lambda_{1}}\|_{{2_{s_{1}}^{*}}}^{2_{s_{1}}^{*}} + \\ &\quad + [s(t)]^{{2_{s_{2}}^{*}}}\|tv\|_{{2_{s_{2}}^{*}}}^{2_{s_{2}}^{*}} + [s(t)]^{\alpha + \beta}\nu (\alpha + \beta) \int_{\mathbb{R}^{N}} h(x)|z_{\mu,s_{1}}^{\lambda_{1}}|^{\alpha}|tv|^{\beta}\,\mathrm{d}x,\end{align*} which also implies 
     \begin{align*}\|z_{\mu,s_{1}}^{\lambda_{1}}\|_{\lambda_{1},s_{1}}^{2} +  t^2\|v\|_{\lambda_{2},s_{2}}^{2}&= [s(t)]^{{2_{s_{1}}^{*}}-2}\|z_{\mu,s_{1}}^{\lambda_{1}}\|_{{2_{s_{1}}^{*}}}^{2_{s_{1}}^{*}} + |t|^{2_{s_{2}}^{*}}[s(t)]^{{2_{s_{2}}^{*}}-2}\|v\|_{{2_{s_{2}}^{*}}}^{2_{s_{2}}^{*}} +\\ &\quad + [s(t)]^{\alpha + \beta -2}\nu (\alpha + \beta)|t|^\beta \int_{\mathbb{R}^{N}} h(x)|z_{\mu,s_{1}}^{\lambda_{1}}|^{\alpha}|v|^{\beta}\,\mathrm{d}x.
\end{align*} \normalsize
Now putting $t=0$ in the above expression, we get
\begin{align*}
    \|z_{\mu,s_{1}}^{\lambda_{1}}\|_{\lambda_{1},s_{1}}^{2} &= [s(0)]^{{2_{s_{1}}^{*}}-2}\|z_{\mu,s_{1}}^{\lambda_{1}}\|_{{2_{s_{1}}^{*}}}^{2_{s_{1}}^{*}},\end{align*} further calculations give us \begin{align*}
    [s(0)]^{{2_{s_{1}}^{*}}-2} &= \frac{ \|z_{\mu,s_{1}}^{\lambda_{1}}\|_{\lambda_{1},s_{1}}^{2}}{\|z_{\mu,s_{1}}^{\lambda_{1}}\|_{{2_{s_{1}}^{*}}}^{2_{s_{1}}^{*}}} = 1, ~~(\text{as}~z_{\mu,s_{1}}^{\lambda_{1}}~\text{is a minimizer of} ~J_1~\text{over}~\mathcal{N}_{\lambda_1} )\end{align*}
   which implies $s(0) =1$.
   
Now from the Implicit Function Theorem, it follows that $s \in C^{1}(\mathbb{R})$ and 
\tiny \[ s'(t) = \frac{ 2t\|v\|_{\lambda_{2},s_{2}}^{2}-2_{s_{2}}^{*}|t|^{2_{s_{2}}^{*}-2}t[s(t)]^{{2_{s_{2}}^{*}}-2}\|v\|_{{2_{s_{2}}^{*}}}^{2_{s_{2}}^{*}} -\beta \nu (\alpha + \beta) [s(t)]^{\alpha + \beta -2}|t|^{\beta-2}t \int_{\mathbb{R}^{N}} h(x)|z_{\mu,s_{1}}^{\lambda_{1}}|^{\alpha}|v|^{\beta}\,\mathrm{d}x}{(2_{s_{1}}^{*}-2) [s(t)]^{{2_{s_{1}}^{*}}-3}\|z_{\mu,s_{1}}^{\lambda_{1}}\|_{{2_{s_{1}}^{*}}}^{2_{s_{1}}^{*}} + (2_{s_{2}}^{*}-2)|t|^{2_{s_{2}}^{*}}[s(t)]^{{2_{s_{2}}^{*}}-3}\|v\|_{{2_{s_{2}}^{*}}}^{2_{s_{2}}^{*}} + [s(t)]^{\alpha + \beta -3}\nu (\alpha + \beta)(\alpha + \beta -2)|t|^\beta \int_{\mathbb{R}^{N}} h(x)|z_{\mu,s_{1}}^{\lambda_{1}}|^{\alpha}|v|^{\beta}\,\mathrm{d}x}, \] \normalsize
for all $t \in \mathbb{R}$. Hence, since $\beta<2$,
$$ s'(t) = - \Bigg[ \frac{ \beta \nu (\alpha + \beta) [s(0)]^{\alpha + \beta -2} \int_{\mathbb{R}^{N}} h(x)|z_{\mu,s_{1}}^{\lambda_{1}}|^{\alpha}|v|^{\beta}\,\mathrm{d}x}{(2_{s_{1}}^{*}-2) [s(0)]^{{2_{s_{1}}^{*}}-3}\|z_{\mu,s_{1}}^{\lambda_{1}}\|_{{2_{s_{1}}^{*}}}^{2_{s_{1}}^{*}} }  \Bigg]|t|^{\beta-2}t~ (1 + o(1))~~~\text{as}~~t \rightarrow 0, $$
$$  s'(t) = - \Bigg[ \frac{ \beta \nu (\alpha + \beta) \int_{\mathbb{R}^{N}} h(x)|z_{\mu,s_{1}}^{\lambda_{1}}|^{\alpha}|v|^{\beta}\,\mathrm{d}x}{(2_{s_{1}}^{*}-2) \|z_{\mu,s_{1}}^{\lambda_{1}}\|_{{2_{s_{1}}^{*}}}^{2_{s_{1}}^{*}} }  \Bigg]|t|^{\beta-2}t~ (1 + o(1))~~~\text{as}~~t \rightarrow 0. $$
Further, we have
$$ s(t) = 1 - \Bigg[ \frac{ \nu (\alpha + \beta) \int_{\mathbb{R}^{N}} h(x)|z_{\mu,s_{1}}^{\lambda_{1}}|^{\alpha}|v|^{\beta}\,\mathrm{d}x}{(2_{s_{1}}^{*}-2) \|z_{\mu,s_{1}}^{\lambda_{1}}\|_{{2_{s_{1}}^{*}}}^{2_{s_{1}}^{*}} }  \Bigg]|t|^{\beta}~ (1 + o(1))~~~\text{as}~~t \rightarrow 0. $$
In particular, there holds that
\begin{equation} \label{ function of s at two star}
    [s(t)]^{2_{s_{1}}^{*}} = 1 - \Bigg[ \frac{N \nu (\alpha + \beta) \int_{\mathbb{R}^{N}} h(x)|z_{\mu,s_{1}}^{\lambda_{1}}|^{\alpha}|v|^{\beta}\,\mathrm{d}x}{2s_{1} \|z_{\mu,s_{1}}^{\lambda_{1}}\|_{{2_{s_{1}}^{*}}}^{2_{s_{1}}^{*}} }  \Bigg]|t|^{\beta}~ (1 + o(1))~~~\text{as}~~t \rightarrow 0.
\end{equation}
Recall that for any $(tu, tv) \in \mathcal{N}_{\nu} $, the energy functional $J_\nu$ can be written as
\begin{align*}
    J_{\nu}(tu,tv) = \frac{s_{1}}{N} t^{2_{s_{1}}^{*}} \|u\|_{2_{s_{1}}^{*}}^{2_{s_{1}}^{*}} + \frac{s_{2}}{N} t^{2_{s_{2}}^{*}} \|v\|_{2_{s_{2}}^{*}}^{2_{s_{2}}^{*}} + \frac{\nu (\alpha +\beta -2)}{2} t^{\alpha +\beta} \int_{\mathbb{R}^{N}} h(x)|u|^{\alpha}|v|^{\beta}\,\mathrm{d}x.
\end{align*}
Combining the above equation with \eqref{ function of s at two star}, we get
\begin{align*}
   & J_{\nu}(s(t)z_{\mu,s_{1}}^{\lambda_{1}}, s(t)tv) - J_{\nu}(z_{\mu,s_{1}}^{\lambda_{1}}, 0) \\
   &=  \frac{s_{1}([s(t)]^{2_{s_{1}}^{*}}-1)}{N}  \|z_{\mu,s_{1}}^{\lambda_{1}}\|_{2_{s_{1}}^{*}}^{2_{s_{1}}^{*}} + \frac{s_{2}[s(t)]^{2_{s_{2}}^{*}}}{N} |t|^{2_{s_{2}}^{*}} \|v\|_{2_{s_{2}}^{*}}^{2_{s_{2}}^{*}} + \frac{\nu (\alpha +\beta -2)}{2}[s(t)]^{\alpha + \beta} |t|^{\beta} \int_{\mathbb{R}^{N}} h(x)|z_{\mu,s_{1}}^{\lambda_{1}}|^{\alpha}|v|^{\beta}\,\mathrm{d}x\\
   &=- \frac{\nu(\alpha+\beta)}{2}|t|^\beta \int_{\mathbb{R}^{N}} h(x)|z_{\mu,s_{1}}^{\lambda_{1}}|^{\alpha}|v|^{\beta}\,\mathrm{d}x + \frac{\nu (\alpha +\beta -2)}{2}[s(t)]^{\alpha + \beta} |t|^{\beta} \int_{\mathbb{R}^{N}} h(x)|z_{\mu,s_{1}}^{\lambda_{1}}|^{\alpha}|v|^{\beta}\,\mathrm{d}x + o(|t|^\beta)\\
   &= -\nu ~|t|^{\beta} \int_{\mathbb{R}^{N}} h(x)|z_{\mu,s_{1}}^{\lambda_{1}}|^{\alpha}|v|^{\beta}\,\mathrm{d}x + o(|t|^\beta) ~~\text{as}~~t \rightarrow 0.
\end{align*}
Therefore, 
\begin{equation} \label{maximum point}
    J_{\nu}(s(t)z_{\mu,s_{1}}^{\lambda_{1}}, s(t)tu) - J_{\nu}(z_{\mu,s_{1}}^{\lambda_{1}}, 0) <0 ~~\text{for}~t \neq 0~\text{and}~ t~ \text{is small},
\end{equation}
which implies that the pair $(z_{\mu,s_{1}}^{\lambda_{1}}, 0)$ is a local strict maximum point for $J_{\nu}$ along a path lying in the Nehari manifold $\mathcal{N}_{\nu}$. Now by the definition of $\mathcal{N}_{\lambda_{1}}$,  $w \in \mathcal{N}_{\lambda_{1}}$ if and only if $(w,0) \in \mathcal{N}_{\nu}$. Also, the minimizers of $J_{\nu}(\cdot~,~0)$ on $\mathcal{N}_{\lambda_{1}}$ are given by the pairs $\{(z_{\sigma,s_{1}}^{\lambda_{1}},0)~:~ \sigma>0\}$. Then
\begin{equation} \label{minimum point}
    J_{\nu}(w, 0) - J_{\nu}(z_{\mu,s_{1}}^{\lambda_{1}}, 0)>0, ~~\text{for all}~w \in \mathcal{N}_{\lambda_{1},s_{1}} \backslash \{(z_{\sigma,s_{1}}^{\lambda_{1}},0)~:~ \sigma>0\}
\end{equation}
which shows that the pair $(z_{\mu,s_{1}}^{\lambda_{1}}, 0)$ is a local minimum for $J_{\nu}$ restricted to $\mathcal{N}_{\lambda_{1}} \times \{0\} \subset \mathcal{N}_{\nu}$. Thus, we deduce from \eqref{maximum point} and \eqref{minimum point} that $(z_{\mu,s_{1}}^{\lambda_{1}}, 0)$ is a saddle point for $J_{\nu}$ in $\mathcal{N}_{\nu}$.\

Now we just assume that $\nu$ is sufficiently large and follow the above argument. Then
$$ s'(t) = - 2 \Bigg[ \frac{  \nu (\alpha + 2) \int_{\mathbb{R}^{N}} h(x)|z_{\mu,s_{1}}^{\lambda_{1}}|^{\alpha}|u|^{2}\,\mathrm{d}x - \|u\|_{\lambda_{2},s_{2}}^{2}}{(2_{s_{1}}^{*}-2) \|z_{\mu,s_{1}}^{\lambda_{1}}\|_{{2_{s_{1}}^{*}}}^{2_{s_{1}}^{*}} }  \Bigg]t~ (1 + o(1))~~~\text{as}~~t \rightarrow 0,$$
and hence
$$J_{\nu}(s(t)z_{\mu,s_{1}}^{\lambda_{1}}, s(t)tu) - J_{\nu}(z_{\mu,s_{1}}^{\lambda_{1}}, 0) = \Bigg(\frac{1}{2}\|u\|_{\lambda_{2},s_{2}}^{2} - \nu \int_{\mathbb{R}^{N}} h(x)|z_{\mu,s_{1}}^{\lambda_{1}}|^{\alpha}|u|^{2}\,\mathrm{d}x  \Bigg)|t|^2 + o(|t|^2)~\text{as}~t \rightarrow 0. $$
This implies that the inequality \eqref{maximum point} holds for $\nu$ sufficiently large.

\noindent $(iv)$ Proceeding in the same way as in the proof of $(iii)$, we get the required result.
\end{proof}
\subsection{The case \texorpdfstring{${\alpha} +{\beta} = \min\{2_{s_{1}}^{*}, 2_{s_{2}}^{*}\} $}{TEXT}} 
We assume some extra continuity assumptions on $h$ to deal with this critical case. { Precisely,
more} hypotheses on the function $h$ are supposed to address this case. In particular,
\begin{equation}\label{H one}
  0 \leq h \in L^{1}(\mathbb{R}^{N}) \cap L^{\infty}(\mathbb{R}^{N})
  ,~h~\text{continuous~near}~0~\text{and}~\infty,~\text{and}~h(0) = \lim\limits_{|x| \rightarrow +\infty}h(x) = 0. \tag\mathbb{H1}
\end{equation}
Adding to that, we will distinguish two different cases: one with $h$ radial case and the other when $h$ is non-radial. To deal with the $h$ non-radial case, we need to impose one extra assumption on $\nu$ i.e., it should be small enough.
Let us define the space of radial functions in $\mathbb{D}$ 
\[ \mathbb{D}_r := \mathcal{D}^{s_1,2}_r(\mathbb{R}^{N}) \times \mathcal{D}^{s_2,2}_r(\mathbb{R}^{N}) = \{(u,v) \in \mathbb{D} : u ~\text{and}~v~\text{are radially symmetric}\}. \]

\begin{lemma}\label{critical with h radial}
Assume that ${\alpha} +{\beta} = \min\{2_{s_{1}}^{*}, 2_{s_{2}}^{*}\} $ and (\ref{H one}), and $h$ is a radial function.
\begin{enumerate}
    \item[(i)] If $\{(u_{n},v_{n})\} \subset \mathbb{D}_{r}$ is a PS sequence for $J_{\nu}$ at level $c \in \mathbb{R}$ such that c satisfies (\ref{energy level PS}), then the sequence $\{(u_{n},v_{n})\}$ admits a subsequence strongly converging in $\mathbb{D}$.
   \item[(ii)] If $\alpha \geq 2$ and $S^{\frac{N}{2s_{1}}}(\lambda_{1}) \geq S^{\frac{N}{2s_{2}}}(\lambda_{2})$, and $\{(u_{n},v_{n})\} \subset \mathbb{D}_{r}$ is a PS sequence for $J_{\nu}^{+}$ at level $c \in \mathbb{R}$ such that c satisfies (\ref{c is less than sum}) and (\ref{c is not equal modified problem}), then there exists ${\nu}_{1} > 0$ and $(\Tilde{u},\Tilde{v}) \in \mathbb{D}_{r}$ such that $(u_{n},v_{n}) \rightarrow (\Tilde{u},\Tilde{v})~ \text{in}~ \mathbb{D}_{r}$ up to subsequence for every $0 < \nu \leq  {\nu}_{1}$. 
    \item[(iii)] If $\beta \geq 2$ and $S^{\frac{N}{2s_{2}}}(\lambda_{2})\geq S^{\frac{N}{2s_{1}}}(\lambda_{1})$, and $\{(u_{n},v_{n})\} \subset \mathbb{D}_{r}$ is a PS sequence for $J_{\nu}^{+}$ at level $c \in \mathbb{R}$ such that c satisfies (\ref{c is less than sum analogous}) and (\ref{c is not equal modified problem analogous}), then there exists ${\nu}_{2} > 0$ and $(\Tilde{u},\Tilde{v}) \in \mathbb{D}_{r}$ such that $(u_{n},v_{n}) \rightarrow (\Tilde{u},\Tilde{v})~ \text{in}~ \mathbb{D}_{r}$ up to subsequence for every $0 < \nu \leq  {\nu}_{2}$.
\end{enumerate}
\end{lemma}  
\begin{proof}
 We observe that the functions $\{(u_{n},v_{n})\} \subset \mathbb{D}_{r}$ are radial and therefore we can not have concentrations at points other than $0~\text{or}~\infty$, otherwise, we will get a contradiction to the concentration–compactness principle by Bonder \cite{Bonder2018} as the set of concentration points is not a countable set.\

 Now if we want to avoid concentration at the points $0$ and $\infty$, by following the proof of Lemma \ref{PS compactness lemma} and Lemma \ref{PS compactness lemma second}, it is sufficient to show that (see (\ref{functional with first component}))
 \begin{align}\label{three three nine}
     \lim\limits_{\epsilon\rightarrow 0}\limsup\limits_{n \rightarrow +\infty}\int_{\mathbb{R}^{N}}h(x)|u_{n}|^{\alpha}|v_{n}|^{\beta}\Psi_{0,\epsilon}(x)\,\mathrm{d}x = 0,
 \end{align}
\begin{align}\label{three four zero}
     \lim\limits_{R\rightarrow +\infty}\limsup\limits_{n \rightarrow +\infty}\int_{|x|>R}h(x)|u_{n}|^{\alpha}|v_{n}|^{\beta}\Psi_{\infty,\epsilon}(x)\,\mathrm{d}x = 0,
\end{align}
where the cut-off function $\Psi_{0,\epsilon}$ is centred at $0$ satisfying (\ref{test function phi at j}) and the cut-off function $\Psi_{\infty,\epsilon}$ supported near $\infty$ satisfying (\ref{phi at infinity}). For any $s_1,s_2 \in (0,1)$, the assumption ${\alpha} +{\beta} = \min\{2_{s_{1}}^{*}, 2_{s_{2}}^{*}\}$ implies that $\frac{\alpha}{2_{s_{1}}^{*}} + \frac{\beta}{2_{s_{2}}^{*}} \leq 1$ and the equality holds if and only if $s_1 = s_2$.

If $\frac{\alpha}{2_{s_{1}}^{*}} + \frac{\beta}{2_{s_{2}}^{*}} < 1$, by using the assumption on $h$ in \eqref{H one} and the H\"older's inequality, we get
  \begin{align*}
      \int_{\mathbb{R}^{N}}h(x)|u_{n}|^{\alpha}|v_{n}|^{\beta}\Psi_{0,\epsilon}(x)\,\mathrm{d}x &= \int_{\mathbb{R}^{N}}(h(x)\Psi_{0,\epsilon}(x))^{1-\frac{\alpha}{2_{s_{1}}^{*}} - \frac{\beta}{2_{s_{2}}^{*}}} (h(x) \Psi_{0,\epsilon}(x))^{\frac{\alpha}{2_{s_{1}}^{*}} + \frac{\beta}{2_{s_{2}}^{*}} }|u_{n}|^{\alpha}|v_{n}|^{\beta}\,\mathrm{d}x\\
      &\leq \bigg( \int_{\mathbb{R}^{N}}h(x)\Psi_{0,\epsilon}(x)\,dx\bigg)^{1-\frac{\alpha}{2_{s_{1}}^{*}} - \frac{\beta}{2_{s_{2}}^{*}}}  \bigg(\int_{\mathbb{R}^{N}}h(x)|u_{n}|^{2_{s_{1}}^{*}}\Psi_{0,\epsilon}(x)\,\mathrm{d}x\bigg)^{\frac{\alpha}{2_{s_{1}}^{*}}}\\ 
      &~~~~~~~~~ \bigg(\int_{\mathbb{R}^{N}}h(x)|v_{n}|^{2_{s_{2}}^{*}}\Psi_{0,\epsilon}(x)\,\mathrm{d}x\bigg)^{\frac{\beta}{2_{s_{2}}^{*}}}.
  \end{align*}
Now from (\ref{Concentration compactness}) and (\ref{H one}), we have
\begin{align*}
   \lim\limits_{n \rightarrow +\infty} \int_{\mathbb{R}^{N}}h(x)|u_{n}|^{2_{s_{1}}^{*}}\Psi_{0,\epsilon}(x)\,\mathrm{d}x &= \int_{\mathbb{R}^{N}}h(x)|\Tilde{u}|^{2_{s_{1}}^{*}}\Psi_{0,\epsilon}(x)\,\mathrm{d}x + \rho_{0}h(0) \\
    & \leq \int_{|x|\leq \epsilon}h(x)|\Tilde{u}|^{2_{s_{1}}^{*}} \,\mathrm{d}x, ~\text{since}~h(0)=0.
\end{align*}
and 
\begin{align*}
    \lim\limits_{n \rightarrow +\infty}\int_{\mathbb{R}^{N}}h(x)|v_{n}|^{2_{s_{2}}^{*}}\Psi_{0,\epsilon}(x)\,\mathrm{d}x &= \int_{\mathbb{R}^{N}}h(x)|\Tilde{v}|^{2_{s_{2}}^{*}}\Psi_{0,\epsilon}(x)\,\mathrm{d}x + \Bar{\rho}_{0}h(0) \\
    & \leq \int_{|x|\leq \epsilon}h(x)|\Tilde{v}|^{2_{s_{2}}^{*}} \,\mathrm{d}x, ~\text{since}~h(0)=0.
\end{align*}
By Combining the above three inequalities, we have
\begin{align*}
     \lim\limits_{\epsilon\rightarrow 0}\limsup\limits_{n \rightarrow +\infty}\int_{\mathbb{R}^{N}}h(x)|u_{n}|^{\alpha}|v_{n}|^{\beta}\Psi_{0,\epsilon}(x)\,\mathrm{d}x & \leq \lim\limits_{\epsilon \rightarrow 0} \Bigg[ \bigg( \int_{|x|\leq \epsilon}h(x)\,\mathrm{d}x\bigg)^{1-\frac{\alpha}{2_{s_{1}}^{*}} - \frac{\beta}{2_{s_{2}}^{*}}}  \bigg(\int_{|x|\leq \epsilon}h(x)|\Tilde{u}|^{2_{s}^{*}} dx\bigg)^{\frac{\alpha}{2_{s}^{*}}}\\ &~~~~~~~~~ \bigg(\int_{|x|\leq \epsilon}h(x)|\Tilde{v}|^{2_{s}^{*}} \,\mathrm{d}x\bigg)^{\frac{\beta}{2_{s}^{*}}}\Bigg]\\
     &=0.
 \end{align*}
If $\frac{\alpha}{2_{s_{1}}^{*}} + \frac{\beta}{2_{s_{2}}^{*}} = 1$, by using the fact that $h \in L^{\infty}(\mathbb{R}^{N})$ and the H\"older's inequality, we get
  \begin{align}\label{Holder inequality with phi at origin}
      \int_{\mathbb{R}^{N}}h(x)|u_{n}|^{\alpha}|v_{n}|^{\beta}\Psi_{0,\epsilon}\,\mathrm{d}x \leq \bigg(\int_{\mathbb{R}^{N}}h(x)|u_{n}|^{2_{s_{1}}^{*}}\Psi_{0,\epsilon}\,\mathrm{d}x\bigg)^{\frac{\alpha}{2_{s_{1}}^{*}}} \bigg(\int_{\mathbb{R}^{N}}h(x)|v_{n}|^{2_{s_{2}}^{*}}\Psi_{0,\epsilon}\,\mathrm{d}x\bigg)^{\frac{\beta}{2_{s_{2}}^{*}}}.
  \end{align}
Again using the above approach, we get
\begin{align*}
     \lim\limits_{\epsilon\rightarrow 0}\limsup\limits_{n \rightarrow +\infty}\int_{\mathbb{R}^{N}}h(x)|u_{n}|^{\alpha}|v_{n}|^{\beta}\Psi_{0,\epsilon}(x)\,\mathrm{d}x & \leq \lim\limits_{\epsilon \rightarrow 0} \Bigg[ \bigg(\int_{|x|\leq \epsilon}h(x)|\Tilde{u}|^{2_{s}^{*}} \,\mathrm{d}x\bigg)^{\frac{\alpha}{2_{s}^{*}}} \bigg(\int_{|x|\leq \epsilon}h(x)|\Tilde{v}|^{2_{s}^{*}} \,\mathrm{d}x\bigg)^{\frac{\beta}{2_{s}^{*}}}\Bigg] =0.
 \end{align*}
Similarly, we can prove (\ref{three four zero}) by using the assumption that $\lim\limits_{|x|\rightarrow +\infty} h(x) = 0$.
\end{proof}
Next, we want to prove the Palais-Smale compactness condition for the case when $h$ is non-radial. For this purpose, we further assume that the parameter $\nu$ is sufficiently small and $ s_{1} = s_{2}=s$.
\begin{lemma}\label{critical case with h non radial}
Let us assume that $\alpha +\beta = 2_{s}^{*} $ and (\ref{H one}), and $\{(u_{n},v_{n})\}$ be a (PS) sequence in $\mathbb{D}$ for $J_{\nu}$ at level $c \in \mathbb{R}$ such that c satisfies (\ref{energy level PS}). Then, there exist $(\Tilde{u},\Tilde{v}) \in \mathbb{D} ~\text{and}~ \nu_0>0$ such that $(u_{n},v_{n}) \rightarrow (\Tilde{u},\Tilde{v}) ~\text{in}~ \mathbb{D}$ up to subsequence for every $0<\nu \leq \nu_0$.
\end{lemma}
\begin{proof}
First, we observe that the concentrations at points $0, \infty$ can be excluded due to \eqref{three three nine} and \eqref{three four zero} given in proof of Lemma \ref{critical with h radial}. Therefore, we have only to take care of concentration points $x_{j} \neq 0,\infty$. Furthermore, we can also assume that the index $j \in \mathcal{J} \cap \mathcal{K}$. On contrary suppose that the concentration occurs at $x_{j} \in \mathbb{R}^{N}$ with $j \in \mathcal{J} \cap \mathcal{K}^c$ or $x_{k} \in \mathbb{R}^{N}$ with $k \in \mathcal{K} \cap  \mathcal{J}^c$, then it is easy to prove as done before that 
\begin{align*}
     \lim\limits_{\epsilon\rightarrow 0}\limsup\limits_{n \rightarrow +\infty}\int_{\mathbb{R}^{N}}h(x)|u_{n}|^{\alpha}|v_{n}|^{\beta}\Psi_{j,\epsilon}\,\mathrm{d}x = 0,
 \end{align*}
 for a smooth cut-off function $\Psi_{j,\epsilon}$ centred at $x_{j}$ satisfying (\ref{test function phi at j}). Thus, this concludes that concentrations can not occur at $x_{j} \in \mathbb{R}^{N}$ with $j \in \mathcal{J} \cap \mathcal{K}^c$ or $x_{k} \in \mathbb{R}^{N}$ with $k \in \mathcal{K} \cap  \mathcal{J}^c$.
 
 Now, we test the functional $J'_{\nu}(u_{n},v_{n})$ with $(u_{n}\Psi_{j,\epsilon}, 0)$ with the assumption that $j \in \mathcal{J}\cap\mathcal{K}$ and we obtain

\begin{align} \label{three four two}
       0 &= \lim_{n \rightarrow +\infty} \big\langle J'_{\nu}(u_{n},v_{n})| (u_{n}\Psi_{j,\epsilon}, 0)\big\rangle \notag\\
       &= \lim_{n \rightarrow +\infty} \bigg( \iint_{\mathbb{R}^{2N} } \frac{|u_{n}(x)-u_{n}(y)|^{2}}{|x-y|^{N+2s}} \Psi_{j,\epsilon}(x, \mathrm{d}x \mathrm{d}y
       + \notag\\ &\quad+\iint_{\mathbb{R}^{2N} } \frac{(u_{n}(x)-u_{n}(y))(\Psi_{j,\epsilon}(x)-\Psi_{j,\epsilon}(y))}{|x-y|^{N+2s}} u_{n}(y)\,\mathrm{d}x \mathrm{d}y  -\lambda_{1} \int_{\mathbb{R}^{N}}\frac{u_{n}^{2}}{|x|^{2s}}\Psi_{j,\epsilon}(x)\, \mathrm{d}x -\notag\\ &\quad- \int_{\mathbb{R}^{N}} |u_{n}|^{2_{s}^{*}}\Psi_{j,\epsilon}(x)\,\mathrm{d}x - \nu \alpha \int_{\mathbb{R}^{N}} h(x)|u_{n}|^{\alpha}|v_{n}|^{\beta}\Psi_{j,\epsilon}(x)\,\mathrm{d}x \bigg)\notag\\
       &= \int_{\mathbb{R}^{N}} \Psi_{j,\epsilon}\, \mathrm{d}\mu - \lambda_{1} \int_{\mathbb{R}^{N}}\Psi_{j,\epsilon}\, \mathrm{d}\gamma - \int_{\mathbb{R}^{N}}\Psi_{j,\epsilon}\, \mathrm{d}\rho - \nu \alpha \lim_{n \rightarrow + \infty}  \int_{\mathbb{R}^{N}} h(x)|u_{n}|^{\alpha}|v_{n}|^{\beta}\Psi_{j,\epsilon}(x)\,\mathrm{d}x. 
\end{align} 
Further, we test the functional $J'_{\nu}(u_{n},v_{n})$ with $(0,v_{n}\Psi_{j,\epsilon})$ and we have the following
\begin{align} \label{three four three}
       0 &= \lim_{n \rightarrow +\infty} \big<\langle J'_{\nu}(u_{n},v_{n})| (0,v_{n}\Psi_{j,\epsilon})\big\rangle \notag\\
       &= \lim_{n \rightarrow +\infty} \bigg( \iint_{\mathbb{R}^{2N} } \frac{|v_{n}(x)-v_{n}(y)|^{2}}{|x-y|^{N+2s}} \Psi_{j,\epsilon}(x)\,\mathrm{d}x \mathrm{d}y
       +\\ &\quad +\iint_{\mathbb{R}^{2N} } \frac{(v_{n}(x)-v_{n}(y))(\Psi_{j,\epsilon}(x)-\Psi_{j,\epsilon}(y))}{|x-y|^{N+2s}} v_{n}(y)\,\mathrm{d}x \mathrm{d}y -\lambda_{1} \int_{\mathbb{R}^{N}}\frac{v_{n}^{2}}{|x|^{2s}}\Psi_{j,\epsilon}(x)\, \mathrm{d}x -\notag\\&\quad- \int_{\mathbb{R}^{N}} |v_{n}|^{2_{s}^{*}}\Psi_{j,\epsilon}(x)\,\mathrm{d}x - \nu \alpha \int_{\mathbb{R}^{N}} h(x)|u_{n}|^{\alpha}|v_{n}|^{\beta}\Psi_{j,\epsilon}(x)\,\mathrm{d}x \bigg) \notag\\
       &= \int_{\mathbb{R}^{N}} \Psi_{j,\epsilon}\,\mathrm{d}\Bar{\mu} - \lambda_{1} \int_{\mathbb{R}^{N}}\Psi_{j,\epsilon}\, \mathrm{d}\Bar{\gamma} - \int_{\mathbb{R}^{N}}\Psi_{j,\epsilon}\, \mathrm{d}\Bar{\rho} - \nu \alpha \lim_{n \rightarrow + \infty}  \int_{\mathbb{R}^{N}} h(x)|u_{n}|^{\alpha}|v_{n}|^{\beta}\Psi_{j,\epsilon}(x)\,\mathrm{d}x.
\end{align} 
By assumption $h \in L^{\infty}(\mathbb{R}^{N})$ and using the H\"{o}lder inequality (\ref{Holder inequality with phi at origin}), there exists some constant $\Tilde{C}>0$ such that the following inequality holds.
\begin{align}\label{three four four}
     \lim\limits_{\epsilon\rightarrow 0}\limsup\limits_{n \rightarrow +\infty}\int_{\mathbb{R}^{N}}h(x)|u_{n}|^{\alpha}|v_{n}|^{\beta}\Psi_{j,\epsilon}\,\mathrm{d}x \leq \Tilde{C}\rho_{j}^{\frac{\alpha}{2_{s}^{*}}} \Bar{\rho_{j}}^{{\frac{\beta}{2_{s}^{*}}}}.
 \end{align}
 Hence, by letting $\epsilon \rightarrow 0$, from (\ref{three four two}), (\ref{three four three}) and (\ref{three four four}) we get
 \begin{align} \label{368}
     \mu_{j} - \rho_{j} - \nu \alpha \Tilde{C}\rho_{j}^{\frac{\alpha}{2_{s}^{*}}} \Bar{\rho_{j}}^{{\frac{\beta}{2_{s}^{*}}}} \leq 0,\end{align}
     \begin{align} \label{369}
     \Bar{\mu}_{j} - \Bar{\rho}_{j} - \nu \beta \Tilde{C}\rho_{j}^{\frac{\alpha}{2_{s}^{*}}} \Bar{\rho_{j}}^{{\frac{\beta}{2_{s}^{*}}}} \leq 0.
 \end{align}
 Then, by (\ref{inequality with S}) together with \eqref{368} and \eqref{369}, we find
 \begin{align*}
     S\bigg(\rho_{j}^{\frac{2}{2_{s}^{*}}} + \Bar{\rho}_{j}^{\frac{2}{2_{s}^{*}}}\bigg) \leq \rho_{j} + \Bar{\rho_{j}} + 2_{s}^{*}\nu \Tilde{C}\rho_{j}^{\frac{\alpha}{2_{s}^{*}}} \Bar{\rho_{j}}^{{\frac{\beta}{2_{s}^{*}}}}.
 \end{align*}
 Therefore,
 $$S\bigg( \rho_{j} + \Bar{\rho}_{j}\bigg)^{\frac{2}{2_{s}^{*}}} \leq (\rho_{j} + \Bar{\rho}_{j}) (1+ 2_{s}^{*}\nu \Tilde{C}).$$
consequently, we obtain that either $\rho_{j} + \Bar{\rho}_{j} = 0$ or $\rho_{j} + \Bar{\rho}_{j} \geq \bigg( \frac{S}{1+ 2_{s}^{*}\nu \Tilde{C}}\bigg)^{\frac{N}{2s}}$. If we have a concentration at some point $x_j$, then following the arguments of Lemma \ref{PS compactness lemma} we get
\begin{align*}
    c &\geq \bigg(\frac{1}{2}  -\frac{1}{\alpha+\beta} \bigg) \bigg(\mu_{j}+\Bar{\mu}_{j}\bigg) + \bigg(\frac{1}{\alpha+\beta} - \frac{1}{2_{s}^{*}}  \bigg)(\rho_{j}+\Bar{\rho}_{j})\\
    & \geq S \bigg(\frac{1}{2}  -\frac{1}{\alpha+\beta} \bigg) \bigg( \rho_{j} + \Bar{\rho}_{j}\bigg)^{\frac{2}{2_{s}^{*}}} + \bigg(\frac{1}{\alpha+\beta} - \frac{1}{2_{s}^{*}}  \bigg)(\rho_{j}+\Bar{\rho}_{j}) \\
    & \geq \frac{s}{N} \bigg( \frac{S}{1+ 2_{s}^{*}\nu \Tilde{C}}\bigg)^{\frac{N}{2s}} 
\end{align*}
If we assume that $\nu >0$ is sufficiently small, then
\begin{align*}
    c \geq \frac{s}{N} \bigg( \frac{S}{1+ 2_{s}^{*}\nu \Tilde{C}}\bigg)^{\frac{N}{2s}} \geq \frac{s}{N} \min \{ S(\lambda_{1}), S(\lambda_{2})\}^{\frac{N}{2s}},
\end{align*}
which gives a contradiction to the hypothesis on level $c$. This implies that $ \rho_{j} = 0 =\Bar{\rho}_{j}$. Hence, the result follows from \eqref{three four four}.
\end{proof}

\section{Existence of ground state and bound state solutions}\label{S4}
In this section, we will state and prove the main results of the article concerning the positive ground and bound state solutions of the system \eqref{main problem}. We will use the following hypotheses to prove the main results,
\begin{equation} \label{C}
\begin{split}
     & \text{Either} ~2<\alpha+\beta<\min\{2_{s_{1}}^{*},2_{s_{2}}^{*}\}~\text{and}~h~\text{satisfies}~(\ref{condition on h})\\
    &~~~~~~~~~~~~~~~~~~ \text{or} \\
    & {\alpha} +{\beta} = \min\{2_{s_{1}}^{*}, 2_{s_{2}}^{*}\}~\text{and}~h~\text{is radial and satisfies}~(\ref{H one})   \hspace{1cm}
\end{split} 
\end{equation}
In the following theorem, we will prove the existence of a positive ground state solution of (\ref{main problem}) for $\nu$ large enough and $h$ satisfying the assumption (\ref{C}).
\begin{theorem} \label{First theorem}
If the hypothesis (\ref{C}) is satisfied, then a positive ground state solution exists to the system (\ref{main problem}) for $\nu$ sufficiently large.
\end{theorem}
\begin{proof}
We know that for any $(u,v) \in \mathbb{D}\backslash \{(0,0)\}$, there exists a unique $t = t_{(u,v)}>0$ such that $(tu,tv) \in \mathcal{N}_{\nu}$ and satisfying the algebraic equation
\begin{align*}
     \|(u,v)\|_{\mathbb{D}}^{2} =  t^{2_{s_{1}}^{*} -2} \|u\|_{{2_{s_{1}}^{*}}}^{2_{s_{1}}^{*}} +  t^{2_{s_{2}}^{*} -2} \|v\|_{{2_{s_{2}}^{*}}}^{2_{s_{2}}^{*}} + \nu (\alpha+\beta ) t^{\alpha+ \beta -2} \int_{\mathbb{R}^{N}} h(x)|u|^{\alpha}|v|^{\beta}\,\mathrm{d}x.\\
      t^{\alpha+ \beta -2} = \frac{ \|(u,v)\|_{\mathbb{D}}^{2} -  t^{2_{s_{1}}^{*} -2} \|u\|_{{2_{s_{1}}^{*}}}^{2_{s_{1}}^{*}} -  t^{2_{s_{2}}^{*} -2} \|v\|_{{2_{s_{2}}^{*}}}^{2_{s_{2}}^{*}}}{\nu (\alpha+\beta )\int_{\mathbb{R}^{N}} h(x)|u|^{\alpha}|v|^{\beta}\,\mathrm{d}x}, \hspace{4.4cm}
\end{align*}
Clearly $t = t_{\nu}$ tends to $0$ as $\nu$ increases to $+\infty$ by using the fact $\alpha+\beta>2$. Furthermore, we have
\begin{align*}
      \lim\limits_{\nu \rightarrow +\infty}t_{\nu}^{\alpha+ \beta -2}\nu &=\lim\limits_{\nu \rightarrow +\infty} \frac{ \|(u,v)\|_{\mathbb{D}}^{2} -  t_{\nu}^{2_{s_{1}}^{*} -2} \|u\|_{{2_{s_{1}}^{*}}}^{2_{s_{1}}^{*}} -  t_{\nu}^{2_{s_{2}}^{*} -2} \|v\|_{{2_{s_{2}}^{*}}}^{2_{s_{2}}^{*}}}{(\alpha+\beta )\int_{\mathbb{R}^{N}} h(x)|u|^{\alpha}|v|^{\beta}\,\mathrm{d}x},\\
      \lim\limits_{\nu \rightarrow +\infty}t_{\nu}^{\alpha+ \beta -2}\nu &= \frac{ \|(u,v)\|_{\mathbb{D}}^{2}}{(\alpha+\beta )\int_{\mathbb{R}^{N}} h(x)|u|^{\alpha}|v|^{\beta}\, \mathrm{d}x},~\text{since}~t_{\nu} \rightarrow 0~\text{as}~\nu \rightarrow +\infty .
\end{align*}
We also have the following,
\begin{align*}
    J_{\nu}(t_{\nu}u,t_{\nu}v) &= \frac{1}{2}t_{\nu}^{2} \|(u,v)\|_{\mathbb{D}}^{2} -\frac{t_{\nu}^{2_{s_{1}}^{*}}}{2_{s_{1}}^{*}}\int_{\mathbb{R}^{N}} |u|^{2_{s_{1}}^{*}}\,\mathrm{d}x 
    -\frac{t_{\nu}^{2_{s_{2}}^{*}}}{2_{s_{2}}^{*}}\int_{\mathbb{R}^{N}} |v|^{2_{s_{2}}^{*}}\,\mathrm{d}x - \nu t_{\nu}^{\alpha+\beta} \int_{\mathbb{R}^{N}}h(x)|u|^{\alpha}|v|^{\beta} \,\mathrm{d}x.\\
    &= \frac{1}{2}t_{\nu}^{2} \|(u,v)\|_{\mathbb{D}}^{2} -\frac{t_{\nu}^{2_{s_{1}}^{*}}}{2_{s_{1}}^{*}}\int_{\mathbb{R}^{N}} |u|^{2_{s_{1}}^{*}}\, \mathrm{d}x 
    -\frac{t_{\nu}^{2_{s_{2}}^{*}}}{2_{s_{2}}^{*}}\int_{\mathbb{R}^{N}} |v|^{2_{s_{2}}^{*}}\, \mathrm{d}x - \frac{1}{\alpha+\beta}t_{\nu}^{2} \|(u,v)\|_{\mathbb{D}}^{2} + \\ &\quad +\frac{t_{\nu}^{2_{s_{1}}^{*}}}{\alpha+\beta}\int_{\mathbb{R}^{N}} |u|^{2_{s_{1}}^{*}}\, \mathrm{d}x + \frac{t_{\nu}^{2_{s_{2}}^{*}}}{\alpha+\beta}\int_{\mathbb{R}^{N}} |v|^{2_{s_{2}}^{*}}\, \mathrm{d}x\\
    &= \bigg( \frac{1}{2} - \frac{1}{\alpha+\beta}\bigg)t_{\nu}^{2} \|(u,v)\|_{\mathbb{D}}^{2} + \\&\quad + \bigg( \frac{1}{\alpha+\beta}- \frac{1}{2_{s_{1}}^{*}}\bigg)t_{\nu}^{2_{s_{1}}^{*}}\int_{\mathbb{R}^{N}} |u|^{2_{s_{1}}^{*}}\, \mathrm{d}x + \bigg( \frac{1}{\alpha+\beta}- \frac{1}{2_{s_{2}}^{*}}\bigg)t_{\nu}^{2_{s_{2}}^{*}}\int_{\mathbb{R}^{N}} |v|^{2_{s_{2}}^{*}}\, \mathrm{d}x\\
    &= \bigg( \frac{1}{2} - \frac{1}{\alpha+\beta} + o_\nu(1) \bigg)t_{\nu}^{2} \|(u,v)\|_{\mathbb{D}}^{2}.
\end{align*}
The last expression is due to the fact that
\begin{align*}
    A = \frac{t_{\nu}^{2_{s_{1}}^{*}-2}}{\|(u,v)\|_{\mathbb{D}}^{2}} \bigg( \frac{1}{\alpha+\beta}- \frac{1}{2_{s_{1}}^{*}}\bigg)\int_{\mathbb{R}^{N}} |u|^{2_{s_{1}}^{*}}\,\mathrm{d}x + \frac{t_{\nu}^{2_{s_{2}}^{*}-2}}{\|(u,v)\|_{\mathbb{D}}^{2}} \bigg( \frac{1}{\alpha+\beta}- \frac{1}{2_{s_{2}}^{*}}\bigg)\int_{\mathbb{R}^{N}} |v|^{2_{s_{2}}^{*}}\,\mathrm{d}x = o_\nu(1),\end{align*} as $~\nu~\rightarrow +\infty.$

Hence, $J_{\nu}(t_{\nu}u,t_{\nu}v) = o(1)$ as $\nu \rightarrow +\infty$ and there exists a $\nu_0>0$ such that, if $\nu > \nu_0$, where $\nu_0$ is sufficiently large, then
\begin{align*} 
        \Tilde{c}_{\nu} &= \inf\limits_{(u,v) \in \mathcal{N}_{\nu}} J_{\nu}(u_{},v_{}) \leq J_{\nu}(t_{\nu}u,t_{\nu}v) < \min\bigg\{ J_{\nu}(z_{\mu,s_{1}}^{\lambda_{1}},0), J_{\nu}(0,z_{\mu,s_{2}}^{\lambda_{2}}) \bigg\} 
\end{align*} and hence from above we get the following inequality
\begin{align} \label{four one}
        \Tilde{c}_{\nu} < \min\bigg\{\frac{s_{1}}{N}S^{\frac{N}{2s_{1}}}(\lambda_{1}), \frac{s_{2}}{N}S^{\frac{N}{2s_{2}}}(\lambda_{2})\bigg\}.
\end{align}
For the case $\alpha+\beta < \min\{2_{s_{1}}^{*},2_{s_{2}}^{*} \}$, we use Lemma \ref{PS compactness lemma} to ensure the existence of $(\Tilde{u},\Tilde{v}) \in \mathbb{D}$ such that $\Tilde{c}_{\nu} = J_{\nu}(\Tilde{u},\Tilde{v})$. Next, we will show that the functions $\Tilde{u}$ and $\Tilde{v}$ are indeed positive. For that purpose, we define the function $\psi(t): (0, \infty) \rightarrow \mathbb{R}$ given by $\psi(t) = J_{\nu}(tu,tv),$ for all $t>0.$ Then
\begin{align*}    
    \psi(t) &= \frac{A_1}{2}t^2 - \frac{A_2}{2^*_{s_{1}}}t^{2^*_{s_{1}}} - \frac{A_3}{2^*_{s_{2}}}t^{2^*_{s_{2}}} - A_4\nu t^{\alpha + \beta},\\
    \psi'(t) &= A_1t - A_2t^{2^*_{s_{1}}-1} - A_3 t^{2^*_{s_{2}}-1} - A_4\nu (\alpha +\beta)t^{\alpha +\beta-1},\\
    \psi''(t) &= A_1 - A_2(2^{*}_{s_{1}}-1)t^{2^*_{s_{1}}-2} - A_3(2^*_{s_{2}}-1) t^{2^*_{s_{2}}-2} - A_4\nu (\alpha +\beta)(\alpha +\beta-1)t^{\alpha +\beta-2},\\
    \psi'''(t) &= - A_2(2^*_{s_{1}}-1)(2^*_{s_{1}}-2)t^{2^*_{s_{1}}-3} - A_3(2^*_{s_{2}}-1)(2^*_{s_{2}}-2) t^{2^*_{s_{2}}-3}\\
    & \hspace{5cm}- A_4\nu (\alpha +\beta)(\alpha +\beta-1)(\alpha +\beta-2)t^{\alpha +\beta-3},
\end{align*}
where $A_1 = \| (u,v)\|_{\mathbb{D}}^{2},~A_2= \| u\|_{2^*_{s_{1}}}^{2^*_{s_{1}}},~A_3= \| v\|_{2^*_{s_{2}}}^{2^*_{s_{2}}},~A_4 = \int_{\mathbb{R}^{N}}h(x)|u|^{\alpha}|v|^{\beta}\,\,\mathrm{d}x.$ Clearly, $A_1,A_2,A_3~\text{and}~A_4$ are non-negative numbers. The condition $\alpha + \beta >2$ implies that $\psi'''(t)<0, \forall t>0$. Therefore the function $\psi'(t)$ is strictly concave for $t>0$. Also, we have $\lim\limits_{t \rightarrow 0}\psi'(t) = 0$ and $\lim\limits_{t \rightarrow +\infty}\psi'(t) = - \infty$. Moreover, the function $\psi'(t)>0$ for $t>0$ small enough. Hence, $\psi'(t)$ has a unique global maximum point at $t=t_0$ and $\psi'(t)$ has a unique root at $t_1 > t_0$ and $\psi''(t)<0$ for $t>t_0$, in particular $\psi''(t_1)<0$. Also, from the equation \eqref{second order derivative} and $\psi(t)$, we observe that $(tu,tv) \in \mathcal{N}_\nu$ if and only if $\psi''(t)<0$.\\
Now we consider the function $(|\Tilde{u}|,|\Tilde{v}|) \in \mathbb{D}$, then from the above arguments there exists a unique $t_2>0$ such that $t_2(|\Tilde{u}|,|\Tilde{v}|) = (t_2|\Tilde{u}|,t_2|\Tilde{v}|) \in \mathcal{N}_{\nu}$ and $t_2$ satisfies the following algebraic equation

\begin{align*} 
     \|(|\Tilde{u}|,|\Tilde{v}|)\|_{\mathbb{D}}^{2} =  t_2^{2_{s_{1}}^{*} -2} \|\Tilde{u}\|_{{2_{s_{1}}^{*}}}^{2_{s_{1}}^{*}} +  t_2^{2_{s_{2}}^{*} -2} \|\Tilde{v}\|_{{2_{s_{2}}^{*}}}^{2_{s_{2}}^{*}} + \nu (\alpha+\beta ) t_2^{\alpha+ \beta -2} \int_{\mathbb{R}^{N}} h(x)|\Tilde{u}|^{\alpha}|\Tilde{v}|^{\beta}\,\mathrm{d}x.
\end{align*}
Also we know that $(\Tilde{u},\Tilde{v}) \in \mathcal{N}_{\nu}$, therefore 
\begin{align*}
\begin{split}
    \|(\Tilde{u},\Tilde{v})\|_{\mathbb{D}}^{2} 
    = \|\Tilde{u}\|_{{2_{s_{1}}^{*}}}^{2_{s_{1}}^{*}} + \|\Tilde{v}\|_{{2_{s_{2}}^{*}}}^{2_{s_{2}}^{*}} + \nu (\alpha + \beta) \int_{\mathbb{R}^{N}} h(x)|\Tilde{u}|^{\alpha}|\Tilde{v}|^{\beta}\, \mathrm{d}x.
\end{split}
\end{align*}
Now from the inequality $\|(|\Tilde{u}|,|\Tilde{v}|)\|_{\mathbb{D}}^{} \leq  \|(\Tilde{u},\Tilde{v})\|_{\mathbb{D}}^{}$, one finds that $t_2 \leq 1.$ Since $(\Tilde{u},\Tilde{v}) \in \mathcal{N}_{\nu}$ and $(\Tilde{u},\Tilde{v})$ is the unique maximum point of $\psi(t) = J_{\nu}(t\Tilde{u},t\Tilde{v}), \forall ~t>t_0$. We can deduce that
\[\Tilde{c}_{\nu} = J_{\nu}(\Tilde{u},\Tilde{v}) = \max\limits_{t>t_0}J_{\nu}(t\Tilde{u},t\Tilde{v}) \geq J_{\nu}(t_2\Tilde{u},t_2\Tilde{v}) \geq J_{\nu}(t_2|\Tilde{u}|,t_2|\Tilde{v}|) \geq \Tilde{c}_{\nu}. \]  
Thus, we can assume that $\Tilde{u}\geq 0$ and $\Tilde{v}\geq 0$ in $\mathbb{R}^{N}$. Moreover, $\Tilde{u}, \Tilde{v}$ are not identically equal to $0$. On contrary, if we suppose that $\Tilde{u} \equiv 0$, then $\Tilde{v}$ is a solution to the problem (\ref{PDE with u zero}) and further $\Tilde{v} = z_{\mu,s_{2}}^{\lambda_{2}}$, which is not possible due to the inequality (\ref{four one}). Following the same argument for $\Tilde{v} \equiv 0$, we get a contradiction to (\ref{four one}).
Furthermore, we conclude that $\Tilde{u}>0$ and $\Tilde{v}>0$ using the maximum principle in $\mathbb{R}^{N} \backslash\{0\}$ \cite[Theorem 1.2]{Pezzo2017}. Hence, the existence of a positive ground state solution $(\Tilde{u},\Tilde{v}) \in \mathcal{N}_{\nu}$ is proved.\

In the case of $\alpha+\beta = \min\{2_{s_{1}}^{*},2_{s_{2}}^{*} \}$, we follow the same argument as in the subcritical case and part (i) of Lemma \ref{critical with h radial} to conclude the existence of a positive ground state $(\Tilde{u},\Tilde{v})$.
\end{proof}

Next, we prove the existence of positive ground state solutions based on the order relation between $S^{\frac{N}{2s_{2}}}(\lambda_{2})$ and $S^{\frac{N}{2s_{1}}}(\lambda_{1})$, and considering the case $\alpha \leq 2 $ or $\beta \leq 2$.
\begin{theorem} \label{second theorem}
Assume (\ref{C}). The system (\ref{main problem}) admits a positive ground state $(\Tilde{u}, \Tilde{v}) \in \mathbb{D}$ under one of the following hypotheses:
\begin{itemize}
    \item [(i)] $S^{\frac{N}{2s_{2}}}(\lambda_{2}) \geq S^{\frac{N}{2s_{1}}}(\lambda_{1}), s_{2} \geq s_{1}$ and either $\beta<2$ or $\beta =2$ and $\nu$ large enough,
    \item [(ii)] $S^{\frac{N}{2s_{1}}}(\lambda_{1}) \geq S^{\frac{N}{2s_{2}}}(\lambda_{2}),s_{1} \geq s_{2}$ and either $\alpha<2$ or $\alpha =2$ and $\nu$ large enough.
\end{itemize}
\end{theorem}
\begin{proof}
(i). Under one of the hypotheses i.e.  either $\beta<2$ or $\beta =2$ and $\nu$ large enough, Proposition \ref{Semitrivial solution as a minimizer} states that $(z_{\mu,s_{1}}^{\lambda_{1}},0)$ is a saddle point for the functional $J_{\nu}$ restricted on $\mathcal{N}_{\nu}$. Since $S^{\frac{N}{2s_{2}}}(\lambda_{2}) \geq S^{\frac{N}{2s_{1}}}(\lambda_{1})~\text{and}~s_{2} \geq s_{1}$ , we have
\[\Tilde{c}_{\nu}< J_{\nu}(z_{\mu,s_{1}}^{\lambda_{1}},0) = \frac{s_{1}}{N}S^{\frac{N}{2s_{1}}}(\lambda_{1}) = \min\bigg\{\frac{s_{1}}{N} S^{\frac{N}{2s_{1}}}(\lambda_{1}),\frac{s_{2}}{N} S^{\frac{N}{2s_{2}}}(\lambda_{2}) \bigg\},\]
where $\Tilde{c}_{\nu}$ is defined by (\ref{ground state level}). In the case of $\alpha+\beta< \min\{ 2_{s_{1}}^{*},2_{s_{2}}^{*}\}$, the Lemma \ref{PS compactness lemma} ensures the existence $(\Tilde{u},\Tilde{v}) \in \mathcal{N}_{\nu}$ such that $\Tilde{c}_{\nu} = J_{\nu}(\Tilde{u},\Tilde{v})$. Now we can assume that $\Tilde{u},\Tilde{v}\geq 0$ and $(\Tilde{u},\Tilde{v})$ are not identically $0$ using the same argument as in Theorem \ref{First theorem}. Then, we use the maximum principle by Pezzo and Quaas  \cite[Theorem 1.2]{Pezzo2017} to conclude that $(\Tilde{u},\Tilde{v})$ is a positive ground state solution of (\ref{main problem}).\

In the case of $\alpha+\beta = \min\{2_{s_{1}}^{*},2_{s_{2}}^{*} \}$ and $h$ radial, we use part (i) of Lemma \ref{critical with h radial} and deduce the existence of a positive ground state $(\Tilde{u},\Tilde{v})$ of (\ref{main problem}).\
 
 (ii) We can also deduce the same conclusion for this part by repeating an analogous argument.
\end{proof}

The following theorem proves that the semi-trivial solutions are the ground state if the order relation between $S^{\frac{N}{2s_{2}}}(\lambda_{2})$ and $S^{\frac{N}{2s_{1}}}(\lambda_{1})$ is strict with $\alpha \geq 2$ or $\beta \geq 2$ and $\nu$ sufficiently small.
\begin{theorem} \label{third theorem}
Assume (\ref{C}). Then the following statements are true:
\begin{itemize}
    \item [(i)] If $S^{\frac{N}{2s_{1}}}(\lambda_{1})> S^{\frac{N}{2s_{2}}}(\lambda_{2}), s_{1} \geq s_{2}$ and $\alpha \geq 2$, then there exists $\nu_0>0$ such that for any $0<\nu<\nu_0$ the pair $(0, z_{\mu,s_{2}}^{\lambda_{2}})$ is a ground state of (\ref{main problem}).
    \item [(ii)] If $S^{\frac{N}{2s_{2}}}(\lambda_{2})> S^{\frac{N}{2s_{1}}}(\lambda_{1}), s_{2} \geq s_{1}$ and $\beta \geq 2$, then there exists $\nu_0>0$ such that for any $0<\nu<\nu_0$ the pair $( z_{\mu,s_{1}}^{\lambda_{1}},0)$ is a ground state of (\ref{main problem}).
\end{itemize}
\end{theorem}
\begin{proof}
(i) We recall that the pair $(0, z_{\mu,s_{2}}^{\lambda_{2}})$ is a local minimum of $J_{\nu}$ over $\mathcal{N}_{\nu}$ by Proposition \ref{Semitrivial solution as a minimizer} provided $\nu$ sufficiently small. Then $c_\nu \leq J_{\nu_{}}(0, z_{\mu,s_{2}}^{\lambda_{2}})$. Indeed, we want to prove that equality holds. On contrary, we suppose that there exists a sequence $\{\nu_{n}\}$ decreasing to $0$ satisfying $\Tilde{c}_{\nu_{n}} < J_{\nu_{n}}(0, z_{\mu,s_{2}}^{\lambda_{2}})$. By assumption $S^{\frac{N}{2s_{1}}}(\lambda_{1})> S^{\frac{N}{2s_{2}}}(\lambda_{2}),s_{1} \geq s_{2}$, we have the following inequality
\begin{equation} \label{four two}
    \Tilde{c}_{\nu_{n}} < \frac{s_{2}}{N} S^{\frac{N}{2s_{2}}}(\lambda_{2})= \min\bigg\{\frac{s_{1}}{N} S^{\frac{N}{2s_{1}}}(\lambda_{1}),\frac{s_{2}}{N} S^{\frac{N}{2s_{2}}}(\lambda_{2}) \bigg\},
\end{equation}
where $\Tilde{c}_{\nu_{n}}$ is given in (\ref{ground state level}) with $\nu = \nu_{n}$. If $\alpha+\beta< \min\{ 2_{s_{1}}^{*},2_{s_{2}}^{*}\}$, the (PS) condition holds by Lemma \ref{PS compactness lemma} at level $\Tilde{c}_{\nu_{n}}$. For the case ${\alpha} +{\beta} = \min\{2_{s_{1}}^{*}, 2_{s_{2}}^{*}\},$ we apply part (i) of Lemma \ref{critical with h radial} for $h$ radial to reach the same conclusion. Hence, for each $n \in \mathbb{N}$, there exists $(\Tilde{u}_{n},\Tilde{v}_{n}) \in \mathbb{D}$ which solves \eqref{main problem} such that $\Tilde{c}_{\nu_{n}} = J_{\nu_{n}}(\Tilde{u}_{n},\Tilde{v}_{n})$.  By following the same argument as in Theorem \ref{First theorem}, we can assume that $\Tilde{u}_{n} \geq 0$ and $\Tilde{v}_{n} \geq 0$. Moreover, $\Tilde{u}_{n} \not\equiv 0$ and $\Tilde{v}_{n} \not\equiv 0$ in $\mathbb{R}^{N}$ as the assumption either  $\Tilde{u}_{n}\equiv 0$ or  $\Tilde{v}_{n}\equiv 0$  contradicts (\ref{four two}). Indeed, we can conclude by the maximum principle of Pezzo and Quaas  \cite[Theorem 1.2]{Pezzo2017} that $\Tilde{u}_{n} > 0$ and $\Tilde{v}_{n} > 0$ in $\mathbb{R}^{N} \backslash \{0\}$. Further, Let us define the following integrals
$$ \sigma_{s_{1},n} = \int_{\mathbb{R}^{N}} \Tilde{u}_{n}^{2_{s_{1}}^{*}}\,\mathrm{d}x~~~~~~~\text{and}~~~~~\sigma_{s_{2},n} = \int_{\mathbb{R}^{N}} \Tilde{v}_{n}^{2_{s_{2}}^{*}}\,\mathrm{d}x.$$
By (\ref{energy functional on Nehari manifold}), we obtain
\begin{equation} \label{four three}
    \Tilde{c}_{\nu_{n}} = J_{\nu_{n}}(\Tilde{u}_{n},\Tilde{v}_{n}) = \frac{s_{1}}{N} \sigma_{s_{1},n} + \frac{s_{2}}{N} \sigma_{s_{2},n} + \nu_{n} \bigg( \frac{\alpha+\beta-2}{2}\bigg)\int_{\mathbb{R}^{N}}h(x)\Tilde{u}_{n}^{\alpha} \Tilde{v}_{n}^{\beta}\,\mathrm{d}x.
\end{equation}
From (\ref{four two}), (\ref{four three}) and \eqref{condition on h}, we obtain that
\begin{equation} \label{four four}
    \sigma_{s_{1},n} + \sigma_{s_{2},n} < S^{\frac{N}{2s_{2}}}(\lambda_{2}).
\end{equation}
We combine the definition of $S(\lambda_1)$ with the first equation in the system \eqref{main problem} as we have given that $(\Tilde{u}_{n},\Tilde{v}_{n})$ is a solution to (\ref{main problem}). we deduce
\begin{equation} \label{four five}
    S(\lambda_{1}) (\sigma_{s_{1},n})^{\frac{N-2s_{1}}{N}} \leq \sigma_{s_{1},n} + \nu_{n} \alpha \int_{\mathbb{R}^{N}}h(x)\Tilde{u}_{n}^{\alpha} \Tilde{v}_{n}^{\beta}\,\mathrm{d}x.
\end{equation}
Now by the H\"older's inequality and the inequality (\ref{four four}), one finds that
\[\int_{\mathbb{R}^{N}}h(x)\Tilde{u}_{n}^{\alpha} \Tilde{v}_{n}^{\beta}\,\mathrm{d}x \leq C(h) \bigg(\int_{\mathbb{R}^{N}}|\Tilde{u}_{n}|^{2_{s_{1}}^{*}}\,\mathrm{d}x\bigg)^{\frac{\alpha}{2_{s_{1}}^{*}}} \bigg(\int_{\mathbb{R}^{N}}|\Tilde{v}_{n}|^{2_{s_{2}}^{*}}\,\mathrm{d}x\bigg)^{\frac{\beta}{2_{s_{2}}^{*}}} \hspace{8mm}\]
\[\leq  C(h) (S(\lambda_{2}))^{\beta \frac{N-2s_{2}}{4s_{2}}} (\sigma_{s_{1},n})^{\frac{\alpha}{2}\frac{N-2s_{1}}{N}}.\]
Using the above expression in (\ref{four five}), we get the following
\[ S(\lambda_{1}) (\sigma_{s_{1},n})^{\frac{N-2s_{1}}{N}} < \sigma_{s_{1},n} + \nu_{n} \alpha  C(h) (S(\lambda_{2}))^{\beta \frac{N-2s_{2}}{4s_{2}}} (\sigma_{s_{1},n})^{\frac{\alpha}{2}\frac{N-2s_{1}}{N}}.\]
Since $S^{\frac{N}{2s_{1}}}(\lambda_{1})> S^{\frac{N}{2s_{2}}}(\lambda_{2})$, then there exists $\epsilon>0$ such that 
\begin{equation} \label{four six}
    (1-\epsilon) S^{\frac{N}{2s_{1}}}(\lambda_{1}) \geq S^{\frac{N}{2s_{2}}}(\lambda_{2}).
\end{equation}
Since $\alpha \geq 2$, by using Lemma \ref{Abdelloui fractional version} with $\sigma = \sigma_{s_{1},n}$, there exists a $\nu_0 = \nu_0(\epsilon)>0$ such that 
\begin{equation*}
    \sigma_{s_{1},n} \geq (1-\epsilon) S^{\frac{N}{2s_{1}}}(\lambda_{1}) ~~\text{for~any}~0<\nu_{n}<\nu_0.
\end{equation*}
By using (\ref{four six}), one finds that $\sigma_{s_{1},n} \geq S^{\frac{N}{2s_{2}}}(\lambda_{2}),$ which gives a contradiction to the inequality (\ref{four four}). Hence, we have  
\begin{equation}\label{four eight}
     \Tilde{c}_{\nu_{}} = \frac{s_{2}}{N} S^{\frac{N}{2s_{2}}}(\lambda_{2}) = J_{\nu_{}}(0, z_{\mu,s_{2}}^{\lambda_{2}}),
\end{equation}
provided $\nu$ sufficiently small.
Thus, the pair $(0, z_{\mu,s_{2}}^{\lambda_{2}})$ is a ground state of (\ref{main problem}) for $\nu$ small enough.
Similarly, we can prove part (ii).
\end{proof}

Now for dealing with the case $\alpha+ \beta = \min\{2_{s_{1}}^{*}, 2_{s_{2}}^{*}\}$ with non-radial $h$, we set $s_{1} = s_{2}=s$ and $\nu$ sufficiently small. We have the following results.

\begin{theorem}
Assume that $\alpha+\beta = 2_{s}^{*}$ with $\nu$ sufficiently small and $h$ is a non-radial function. Then the system (\ref{main problem}) has a positive ground state solution $(\Tilde{u}, \Tilde{v}) \in \mathbb{D}$ under one of the following hypotheses:
\begin{itemize}
    \item [(i)] $S^{\frac{N}{2s_{}}}(\lambda_{2}) \geq S^{\frac{N}{2s_{}}}(\lambda_{1})$ and $\beta<2$,
    \item [(ii)] $S^{\frac{N}{2s_{}}}(\lambda_{1}) \geq S^{\frac{N}{2s_{}}}(\lambda_{2})$ and $\alpha<2$.
\end{itemize}

\end{theorem}
\begin{proof}
The proof is direct by using the approach of Theorem \ref{second theorem} and Lemma \ref{critical case with h non radial}.
\end{proof}

\begin{theorem}
Assume that $\alpha+\beta = 2_{s}^{*}$ with $\nu$ sufficiently small and $h$ is a non-radial function. Then the following holds:
\begin{itemize}
 \item [(i)] If $S^{\frac{N}{2s_{1}}}(\lambda_{1})> S^{\frac{N}{2s_{2}}}(\lambda_{2})$ and $\alpha \geq 2$, then the pair $(0, z_{\mu,s_{2}}^{\lambda_{2}})$ is a ground state of (\ref{main problem}).
    \item [(ii)] If $S^{\frac{N}{2s_{2}}}(\lambda_{2})> S^{\frac{N}{2s_{1}}}(\lambda_{1})$ and $\beta \geq 2$, then the pair $( z_{\mu,s_{1}}^{\lambda_{1}},0)$ is a ground state of (\ref{main problem}).
\end{itemize}
\end{theorem}
\begin{proof}
The proof follows the approach of Theorem \ref{third theorem} and Lemma \ref{critical case with h non radial}.
\end{proof}
In the next theorem, we show the existence of a positive bound state solution of the Mountain pass type.
\begin{theorem} \label{Bound state}
Assume (\ref{C}) with $s_{1} = s_{2} = s$. If 
\begin{itemize}
    \item [(i)] Either \begin{equation} \label{inequalities in bound state theorem}
        \alpha \geq 2 ~~and~~  \frac{1}{2}<\bigg(\frac{S(\lambda_{2})}{S(\lambda_{1})}\bigg)^{\frac{N}{2s}} <1, 
    \end{equation}
     \item [(ii)] or \begin{equation*}
        \beta \geq 2~~and~~  \frac{1}{2}<\bigg(\frac{S(\lambda_{1})}{S(\lambda_{2})}\bigg)^{\frac{N}{2s}} <1,
    \end{equation*}
\end{itemize}
then for $\nu$ sufficiently small, there exists a  Mountain pass type positive bound state solution to the problem (\ref{main problem}).
\end{theorem}
\begin{proof}
$(i)$  We start with constructing a Mountain pass level so that the functional $J_{\nu}^{+}$ restricted on ${\mathcal{N}_{\nu}^{+}}$ satisfies the Mountain pass geometry, and the Palais-Smale condition is also satisfied at this level. 
So we consider the set of paths connecting continuously $(z_{\mu,s}^{\lambda_{1}},0)$ to $(0,z_{\mu,s}^{\lambda_{2}})$, namely
\[ \Sigma_{\nu} = \big\{ \varphi(t) = (\varphi_{1}(t),\varphi_{2}(t)) \in C^{0}([0,1],\mathcal{N}_{\nu}^{+}): \varphi(0)=(z_{\mu,s}^{\lambda_{1}},0)~\text{and}~\varphi(1)= (0,z_{\mu,s}^{\lambda_{2}}) \big\},\]
and define the associated Mountain pass level as
$$ \mathcal{C}_{\mathcal{MP}} = \inf\limits_{\varphi \in \Sigma_{\nu}} \max\limits_{t \in [0,1]} J_{\nu}^{+}(\varphi(t)).$$
Assumption (\ref{inequalities in bound state theorem}) implies that
\begin{equation} \label{Separability condition implies}
  \frac{2s}{N}S^{\frac{N}{2s}}(\lambda_{2}) > \frac{s}{N}S^{\frac{N}{2s}}(\lambda_{1}).   
\end{equation}
Further, by the monotonicity of $S(\lambda)$, we can choose $\epsilon>0$ sufficiently small such that
\begin{equation} \label{four ten}
    \frac{2s}{N}(1-\epsilon) \bigg( \frac{S(\lambda_{1})+S(\lambda_{2})}{2}\bigg)^{\frac{N}{2s}} > \frac{2s}{N}S^{\frac{N}{2s}}(\lambda_{2}) > \frac{s(1+\epsilon)}{N}S^{\frac{N}{2s}}(\lambda_{1}).
\end{equation}
Now we claim the existence of a $\nu_0 = \nu_0(\epsilon)>0$ such that the following inequality 
\begin{equation} \label{four elevan}
    \max\limits_{t \in [0,1]} J_{\nu}^{+}(\varphi(t)) \geq \frac{2s}{N}(1-\epsilon) \bigg( \frac{S(\lambda_{1})+S(\lambda_{2})}{2}\bigg)^{\frac{N}{2s}} ~\text{with}~\varphi\in \Sigma_{\nu},
\end{equation}
holds for every $0<\nu<\nu_0$.
If $\varphi=(\varphi_{1},\varphi_{2}) \in \Sigma_{\nu}$, then by using identity (\ref{equivalent norm for modified problem}), we have
\begin{align} \label{four one two}
    \begin{split}
        \iint_{\mathbb{R}^{2N}} \frac{|\varphi_{1}(x)-\varphi_{1}(y)|^{2} + |\varphi_{2}(x)-\varphi_{2}(y)|^{2}}{|x-y|^{N+2s_{}}} \,\mathrm{d}x \mathrm{d}y - \lambda_{1} \int_{\mathbb{R}^{N}}\frac{\varphi_{1}^{2}}{|x|^{2s}}\,\mathrm{d}x - \lambda_{2} \int_{\mathbb{R}^{N}}\frac{\varphi_{2}^{2}}{|x|^{2s}}\,\mathrm{d}x\\
        = \int_{\mathbb{R}^{N}} \big( (\varphi_{1}^{+}(t))^{2_{s}^{*}} + (\varphi_{2}^{+}(t))^{2_{s}^{*}}\big)\, \mathrm{d}x + \nu (\alpha + \beta) \int_{\mathbb{R}^{N}}h(x)(\varphi_{1}^{+}(t))^{\alpha}(\varphi_{2}^{+}(t))^{\beta}\,\mathrm{d}x,
    \end{split}
\end{align}
and we use (\ref{restricted functonal on nehari manifold of modified problem}) to get
\begin{align}\label{four one three}
\begin{split}
    J_{\nu}^{+}(\varphi(t)) = \frac{s}{N} \int_{\mathbb{R}^{N}} \big( (\varphi_{1}^{+}(t))^{2_{s}^{*}} + (\varphi_{2}^{+}(t))^{2_{s}^{*}}\big)\, \mathrm{d}x \hspace{3cm}\\
    + \nu \bigg( \frac{\alpha+\beta-2}{2} \bigg)\int_{\mathbb{R}^{N}}h(x)(\varphi_{1}^{+}(t))^{\alpha}(\varphi_{2}^{+}(t))^{\beta}\,\mathrm{d}x.
\end{split}
\end{align}
Now we define $\sigma_{s}(t) = (\sigma_{1,s}(t),\sigma_{2,s}(t))$ with $\sigma_{j,s}(t) = \int_{\mathbb{R}^{N}}(\varphi_{j}^{+}(t))^{2_{s}^{*}}\, \mathrm{d}x$ for $j = 1,2$. Observe that if $\sigma_{j,s}^{}(t) > 2 S^{\frac{N}{2s}}(\lambda_{j})$, then the inequality (\ref{four elevan}) holds. Therefore, we assume that $\sigma_{j,s}^{}(t) \leq 2 S^{\frac{N}{2s}}(\lambda_{j})$, $j =1,2$ for all $t \in [0,1]$. We combine the definition of $S(\lambda)$ with (\ref{four one two}) and obtain
\begin{align}\label{four one four}
    \begin{split}
        S(\lambda_{1})(\sigma_{1,s}(t))^{\frac{N-2s}{N}} + S(\lambda_{2})(\sigma_{2,s}(t))^{\frac{N-2s}{N}} \leq \iint_{\mathbb{R}^{2N}} \frac{|\varphi_{1}(x)-\varphi_{1}(y)|^{2} + |\varphi_{2}(x)-\varphi_{2}(y)|^{2}}{|x-y|^{N+2s_{}}}\, \mathrm{d}x \mathrm{d}y \\
        - \lambda_{1} \int_{\mathbb{R}^{N}}\frac{\varphi_{1}^{2}}{|x|^{2s}}\,\mathrm{d}x - \lambda_{2} \int_{\mathbb{R}^{N}}\frac{\varphi_{2}^{2}}{|x|^{2s}}\,\mathrm{d}x \hspace{3cm}\\
        = \sigma_{1,s}(t) + \sigma_{2,s}(t) \hspace{6cm}\\
        +\nu (\alpha + \beta) \int_{\mathbb{R}^{N}}h(x)(\varphi_{1}^{+}(t))^{\alpha}(\varphi_{2}^{+}(t))^{\beta}\,\mathrm{d}x. \hspace{1cm}
    \end{split}
\end{align}
Using the H\"older’s inequality, we get
\begin{equation} \label{four one five}
    \int_{\mathbb{R}^{N}}h(x)(\varphi_{1}^{+}(t))^{\alpha}(\varphi_{2}^{+}(t))^{\beta} \, \mathrm{d}x \leq  ~C(h) (\sigma_{1,s}(t))^{\frac{\alpha}{2}\frac{N-2s}{N}} (\sigma_{2,s}(t))^{\frac{\beta}{2}\frac{N-2s}{N}}.
\end{equation}
Also, by the definition of $\Sigma_\nu$ and since $\varphi=(\varphi_{1},\varphi_{2}) \in \Sigma_{\nu}$, we have
\[ \sigma_{s}(0) = \bigg( \int_{\mathbb{R}^{N}} (z_{\mu,s}^{\lambda_{1}})^{2_{s}^{*}}\,\mathrm{d}x,0\bigg)~~~~\text{and}~~~\sigma_{s}(1) = \bigg(0, \int_{\mathbb{R}^{N}} (z_{\mu,s}^{\lambda_{2}})^{2_{s}^{*}}\, \mathrm{d}x\bigg).\]
Thus, by the continuity of $\sigma_{s}$, there is a $t_0 \in (0,1)$ such that $\sigma_{1,s}(t_0) = \Tilde{\sigma}_{s} = \sigma_{2,s}(t_0)$. Combining (\ref{four one four}) and (\ref{four one five}), and taking $t=t_0$, we deduce the following
\[ (S(\lambda_{1})+ S(\lambda_{2}))\Tilde{\sigma}_{s}^{\frac{N-2s}{N}} \leq 2\Tilde{\sigma}_{s} +C\nu (\alpha + \beta) \Tilde{\sigma}_{s}^{\frac{\alpha + \beta}{2}\frac{N-2s}{N}}.\]
Now by using Lemma \ref{Abdelloui fractional version}, there exists a $\nu_0 = \nu_0(\epsilon)$ such that
\begin{equation} \label{four one six}
    \Tilde{\sigma}_{s} \geq (1-\epsilon) \bigg( \frac{S(\lambda_{1})+S(\lambda_{2})}{2}\bigg)^{\frac{N}{2s}}~~~\text{for~every}~0<\nu\leq \nu_0.
\end{equation}
Consequently, we combine (\ref{four one three}) and (\ref{four one six}) to get 
\[\max\limits_{t \in [0,1]} J_{\nu}^{+}(\varphi(t)) \geq \frac{s}{N}(\sigma_{1,s}(t_0)+\sigma_{2,s}(t_0)) \geq \frac{2s}{N}(1-\epsilon) \bigg( \frac{S(\lambda_{1})+S(\lambda_{2})}{2}\bigg)^{\frac{N}{2s}},\]
which proves claim (\ref{four elevan}). Moreover, by (\ref{four ten}) and (\ref{four elevan}), one can state that
\begin{equation} \label{four one seven}
    \mathcal{C}_{\mathcal{MP}} > \frac{s(1+\epsilon)}{N}S^{\frac{N}{2s}}(\lambda_{1}) = (1+\epsilon) J_{\nu}^{+}(z_{\mu,s}^{\lambda_{1}},0).
\end{equation}
Thus, the functional $J_{\nu}^{+}$ admits a Mountain-Pass-geometry on $\mathcal{N}_{\nu}$.

Now we show that the Palais-Smale compactness condition is satisfied at the Mountain pass level $\mathcal{C}_{\mathcal{MP}}$. We consider $\varphi(t) = (\varphi_{1}(t),\varphi_{2}(t)) = \big((1-t)^{1/2}z_{\mu,s}^{\lambda_{1}}, t^{1/2}z_{\mu,s}^{\lambda_{2}}\big)$ for $t \in [0,1]$. By the definition of the Nehari manifold, there exists a continuous positive function $\eta: [0,1] \rightarrow (0, +\infty
)$ such that the $\eta\varphi \in \mathcal{N}_{\nu}^{} \cap \mathcal{N}_{\nu}^{+}$ for $t \in [0,1]$. We notice that $\eta(0) = \eta(1) =1$.\

Now we define
\[ \sigma_{s}(t) = (\sigma_{1,s}(t),\sigma_{2,s}(t)) = \bigg(\int_{\mathbb{R}^{N}}(\eta\varphi_{1}^{}(t))^{2_{s}^{*}}\,\mathrm{d}x,~ \int_{\mathbb{R}^{N}}(\eta\varphi_{2}^{}(t))^{2_{s}^{*}}\,\mathrm{d}x \bigg).\]
Then, we have
\begin{equation}\label{four one eight}
    \sigma_{1,s}(0) = \int_{\mathbb{R}^{N}} (z_{\mu,s}^{\lambda_{1}})^{2_{s}^{*}}\,\mathrm{d}x = S^{\frac{N}{2s}}(\lambda_{1})~~~\text{and}~~~\sigma_{2,s}(1) = \int_{\mathbb{R}^{N}} (z_{\mu,s}^{\lambda_{2}})^{2_{s}^{*}}\,\mathrm{d}x = S^{\frac{N}{2s}}(\lambda_{2}).
\end{equation}
Since $\eta\varphi(t) \in \mathcal{N}_{\nu}^{+} \cap \mathcal{N}_{\nu}^{}$, by using the algebraic equation (\ref{Algebraic equation}), we obtain
\begin{align*}
    \|\big((1-t)^{1/2}z_{\mu,s}^{\lambda_{1}}, t^{1/2}z_{\mu,s}^{\lambda_{2}}\big)\|_{\mathbb{D}}^{2} = \eta^{2_{s}^{*} -2}(t) \big((1-t)^{2_{s}^{*}/2}\sigma_{1,s}(0) + t^{2_{s}^{*}/2}\sigma_{2,s}(1)\big) \hspace{3cm}\\
    + \nu(\alpha+\beta) (\eta(t))^{\alpha+\beta -2} (1-t)^{\alpha /2}t^{\beta/2} \int_{\mathbb{R}^{N}}h(x)(z_{\mu,s}^{\lambda_{1}})^{\alpha} (z_{\mu,s}^{\lambda_{2}})^{\beta}\,\mathrm{d}x,
\end{align*}
and therefore,
\begin{equation}\label{four one nine}
    \eta^{2_{s}^{*} -2}(t) < \frac{\|(\varphi_{1}(t), \varphi_{2}(t))\|_{\mathbb{D}}^{2}}{ \int_{\mathbb{R}^{N}} \big( (\varphi_{1}^{}(t))^{2_{s}^{*}} + (\varphi_{2}^{}(t))^{2_{s}^{*}}\big) \,\mathrm{d}x} = \frac{(1-t)\sigma_{1,s}(0) + t\sigma_{2,s}(1)}{(1-t)^{2_{s}^{*}/2}\sigma_{1,s}(0) + t^{2_{s}^{*}/2}\sigma_{2,s}(1)}
\end{equation}
for every $t \in (0,1)$. It is followed by the definition of $\eta$, (\ref{two one four}) and (\ref{four one nine}) that
\begin{align} \label{four two zero}
    \begin{split}
        J_{\nu}^{+}(\eta\varphi(t)) = \bigg(\frac{1}{2}- \frac{1}{\alpha+\beta} \bigg)\|\eta\varphi(t)\|_{\mathbb{D}}^{2} \hspace{6cm}\\
        + \bigg( \frac{1}{\alpha+\beta}-\frac{1}{2_{s}^{*}} \bigg) \eta^{2_{s}^{*}}(t) \bigg(\int_{\mathbb{R}^{N}} \big( (\varphi_{1}^{}(t))^{2_{s}^{*}} + (\varphi_{2}^{}(t))^{2_{s}^{*}}\big)\, \mathrm{d}x \bigg)\\
        = \eta^{2_{}^{}}(t) \bigg(\frac{1}{2}- \frac{1}{\alpha+\beta} \bigg) [(1-t)\sigma_{1,s}(0) + t\sigma_{2,s}(1)] \hspace{2.5cm}\\
        + \bigg( \frac{1}{\alpha+\beta}-\frac{1}{2_{s}^{*}} \bigg) \eta^{2_{s}^{*}}(t) [(1-t)^{2_{s}^{*}/2}\sigma_{1,s}(0) + t^{2_{s}^{*}/2}\sigma_{2,s}(1)] \hspace{5mm}\\
        < \frac{s\eta^{2_{}^{}}(t)}{N} [(1-t)\sigma_{1,s}(0) + t\sigma_{2,s}(1)] \hspace{4.7cm}
    \end{split}
\end{align}
Then, by (\ref{four one nine}) and (\ref{four two zero}), and for every $t \in (0,1)$, we obtain that
\[ J_{\nu}^{+}(\eta\varphi(t)) < G(t) : = \frac{s}{N}[(1-t)\sigma_{1,s}(0) + t\sigma_{2,s}(1)] \bigg[\frac{(1-t)\sigma_{1,s}(0) + t\sigma_{2,s}(1)}{(1-t)^{2_{s}^{*}/2}\sigma_{1,s}(0) + t^{2_{s}^{*}/2}\sigma_{2,s}(1)} \bigg]^{\frac{N-2s}{2s}}.\]
Clearly, the function $G(t)$ is maximum at point $t = \frac{1}{2}$. Also, from (\ref{four one eight}), we have
\[ G \bigg(\frac{1}{2}\bigg) = \frac{s}{N}(\sigma_{1,s}(0)+ \sigma_{2,s}(1)) = \frac{s}{N} ( S^{\frac{N}{2s}}(\lambda_{1}) + S^{\frac{N}{2s}}(\lambda_{2})).\]
we conclude 
$$\mathcal{C}_{\mathcal{MP}} \leq \max\limits_{t \in [0,1]} J_{\nu}^{+}(\eta\varphi(t)) < \frac{s}{N} ( S^{\frac{N}{2s}}(\lambda_{1}) + S^{\frac{N}{2s}}(\lambda_{2})).$$
If $S^{\frac{N}{2s}}(\lambda_{1})>S^{\frac{N}{2s}}(\lambda_{2})$, using the separability condition (\ref{Separability condition implies}) and the inequality (\ref{four one seven}), it follows that
$$ \frac{s}{N} S^{\frac{N}{2s}}(\lambda_{2}) < \frac{s}{N} S^{\frac{N}{2s}}(\lambda_{1})<\mathcal{C}_{\mathcal{MP}}< \frac{s}{N} ( S^{\frac{N}{2s}}(\lambda_{1}) + S^{\frac{N}{2s}}(\lambda_{2}))< \frac{3s}{N} S^{\frac{N}{2s}}(\lambda_{2}).$$
From the above expression, it is clear that the Mountain pass level $\mathcal{C}_{\mathcal{MP}}$ satisfies the assumptions of Lemma \ref{PS compactness lemma second} and Lemma \ref{critical with h radial}. Therefore, by the Mountain-Pass theorem, the functional $J_{\nu}^{+}|_{\mathcal{N}_{\nu}^{+}}$ admits a Palais-Smale sequence $\{(u_n,v_n)\} \subset \mathcal{N}_{\nu}^{+}$ at level $\mathcal{C}_{\mathcal{MP}}$.\

For the subcritical case $\alpha+\beta< 2_{s}^{*}$, from analogous versions of Lemmas \ref{equivalent of critical points lemma} and Lemma \ref{PS compactness lemma second} for the functional $J_{\nu}^{+}$, we imply that the sequence $\{(u_{n},v_{n})\}$ has a subsequence which strongly converges to a critical point $(\Tilde{u},\Tilde{v})$ of $J_{\nu}^{+}$ on $\mathcal{N}_{\nu}^{+}$. Therefore, it is also a critical point of $J_{\nu}^{+}$ defined in $\mathbb{D}$. Further, we have $\Tilde{u},\Tilde{v}\geq 0$ in $\mathbb{R}^{N}$ and $\Tilde{u},\Tilde{v} \neq (0,0)$. Indeed, we can conclude by the maximum principle of Pezzo and Quaas  \cite[Theorem 1.2]{Pezzo2017} that $\Tilde{u}_{} > 0$ and $\Tilde{v}_{} > 0$ in $\mathbb{R}^{N} \backslash \{0\}$. Hence, $(\Tilde{u},\Tilde{v})$ is a bound state solution to the system \eqref{main problem}.
To deal with the critical case, i.e., $\alpha+\beta = 2_{s}^{*}$, we follow the same approach for the compactness of the Palais-Smale sequence using Lemma \ref{critical with h radial}.\

(ii) Similarly, this part can be proved using Lemma \ref{lemma three six} and Lemma \ref{critical with h radial}.
\end{proof}
\section*{Acknowledgments}
RK wants to thank the support of the CSIR fellowship, file no. 09/1125(0016)/2020--EMR--I for his Ph.D. work. AS was supported by the DST-INSPIRE Grant DST/INSPIRE/04/2018/002208. T. Mukherjee acknowledges the support of the Start up Research Grant from DST-SERB, sanction no. SRG/2022/000524
\bibliography{Ref}
\bibliographystyle{abbrv}
\begin{enumerate}
    \item[E-mail:] kumar.174@iitj.ac.in
    \item[E-mail:] tuhina@iitj.ac.in
    \item[E-mail:] abhisheks@iitj.ac.in
\end{enumerate}
\end{document}